\documentclass{amsart}
\usepackage{amsmath}
\usepackage{amssymb}
\usepackage{amsthm}
\usepackage{enumerate}
\usepackage[pdftex]{graphicx}
\usepackage{caption}
\theoremstyle{definition}
\newtheorem{definition}{Definition}[section]
\theoremstyle{plain}
\newtheorem{lemma}[definition]{Lemma}
\newtheorem{theorem}[definition]{Theorem}

\newtheorem{proposition}[definition]{Proposition}
\newtheorem{corollary}[definition]{Corollary}
\theoremstyle{remark}
\newtheorem{remark}[definition]{Remark}

\newtheorem{example}[definition]{Example}

\makeatletter
\@namedef{subjclassname@2020}{
  \textup{2020} Mathematics Subject Classification}
\makeatother

\newcommand{\mycl}{\operatorname{cl}}
\newcommand{\myint}{\operatorname{int}}
\newcommand{\myx}{\fbox{\phantom{123}}}
\newcommand{\myp}{\frak P}
\newcommand{\myfdim}{\operatorname{fdim}}
\newcommand{\mydcl}{\operatorname{discl}}

\newcommand{\myrk}{\operatorname{rk}}
\newcommand{\mydim}{\operatorname{D}}

\begin{document}
\title[Uniform local weak o-minimality]{Uniform local weak o-minimality and $*$-local weak o-minimality}
\author[M. Fujita]{Masato Fujita}
\address{Department of Liberal Arts,
Japan Coast Guard Academy,
5-1 Wakaba-cho, Kure, Hiroshima 737-8512, Japan}
\email{fujita.masato.p34@kyoto-u.jp}

\begin{abstract}
We propose the notions of uniform local weak o-minimality and $*$-local weak o-minimality.
Local monotonicity theorems hold in definably complete locally o-minimal structures and uniformly locally o-minimal structures of the second kind.
In this paper, we demonstrate new local monotonicity theorems for uniformly locally weakly o-minimal structures of the second kind and for locally o-minimal structures under the assumption called the univariate $*$-continuity property.
We also prove that several formulas for dimension of definable sets which hold in definably complete locally o-minimal structures also hold in $*$-locally weakly o-minimal structures possessing the univariate $*$-continuity property.
\end{abstract}


\maketitle

\section{Introduction}\label{sec:intro}
Local o-minimality \cite{TV} is a variant of o-minimality.
Readers who are not familiar with o-minimal structures should consult the standard textbook \cite{vdD}. 
Locally o-minimal structures have been studied in the papers \cite{Fornasiero,KTTT,Fuji_uniform,Fuji_almost,Fuji_tame,FKK,Fuji_pregeo}, and we assumed definable completeness in most of them.
The purpose of this paper is to prove several assertions in them still hold without the assumption of definable completeness.
More concretely, we conduct the following in this paper:
\begin{itemize}
	\item We propose several new category of expansions of dense linear order without endpoints in a unified way and in the same manner as we defined local o-minimality and other variants of local o-minimality using the definition of o-minimality;
	\item We show that the local monotonicity theorems proved in \cite{Fuji_uniform,FKK} hold under more relaxed condition without assuming definable completeness;
	\item We show that the formulas on dimension of definable sets proved in \cite{Fuji_tame,FKK,Fuji_pregeo} hold under some reasonable conditions but not assuming definable completeness. 
\end{itemize}

This paper is organized as follows:
We first define localized versions of several variants of o-minimality such as weak o-minimality in a unified way in Section \ref{sec:definition}.
For instance, it is already known that an expansion of a dense linear order without endpoints is locally o-minimal if and only if every unary definable set is a union of an open set and a discrete closed set under the assumption that the structure is definably complete \cite[Lemma 2.3]{Fuji_tame}.
The structures satisfying the latter condition are called locally l-visceral in this paper and we investigate features of sets definable in locally l-visceral structures. 
We propose the notion of $*$-locally weakly o-minimal structures which is a subcategory of locally o-minimal structures in this section.
New structures introduced in this section are defined in Section \ref{sec:definition}. 

Section \ref{sec:uniform} treats local monotonicity theorems.
We have already obtained local monotonicity theorems in two cases. 
In the first case, the structures in consideration are uniformly locally o-minimal structures of the second kind \cite[Section 3]{Fuji_uniform}.
They are not necessarily definably complete.
Note that the monotonicity theorem obtained in this case is parametrized.
We show this parametrized local monotonicity theorem holds under a slightly relaxed assumption
that the structure is a uniformly locally weakly o-minimal structure of the second kind in Section \ref{sec:uniform}.
As a corollary, we get a new result on uniformly locally weakly o-minimal structures of the first kind (Corollary \ref{cor:mono1}).
We have already demonstrated a local monotonicity theorem for definably complete locally o-minimal structures in \cite{FKK}.
We prove that the same assertion holds even if we do not assume definable completeness but we assume the univariate $*$-continuity property defined in Section \ref{sec:star-cont}.

Section \ref{sec:dimension} is devoted for dimension theory.
We recall the definition of topological dimension, which is often adopted as dimension of definable sets, in Section \ref{sec:topological}.
We treat locally l-visceral structures satisfying technical conditions in Section \ref{sec:preliminary} and $*$-locally weakly o-minimal structures with the univariate $*$-continuity property in Section \ref{sec:*-local}, respectively.
The formulas proven in this section are the same as the formulas given in \cite{FKK, Fuji_tame}, but we emphasize that definable completeness is not assumed differently from previous studies.
Among these formulas, Corollary \ref{cor:wencel} deserves a special attention.
It is a generalization of Wencel's weakly o-minimal result in \cite[Theorem 4.2]{W} to the $*$-locally weakly o-minimal cases, which asserts that a dimension formula called the addition property is equivalent to the univariate $*$-continuity property.
In Section \ref{sec:decomposition}, we prove a naive decomposition theorem similar to \cite[Theorem 4.5]{Fuji_tame} for later use.
In Section \ref{sec:other}, we introduce alternative definitions of dimension and prove that they coincide with the topological dimension when the structure is a $*$-locally weakly o-minimal structure possessing the univariate $*$-continuity property.

As it is already introduced, many assertions in this paper just say that some assertions in \cite{Fuji_uniform,Fuji_tame,FKK,Fuji_pregeo} hold under some assumptions which are not identical to the assumptions in original papers.
Some of them are proved by the same proof but using some assertions shown in this paper in place of the assertions in the original papers.
We omit the proofs in such cases by just referring the counterparts in the original paper and which assertions in this paper are used for the proof.   
In a preprint version of this paper available in arXiv, we give proofs to such assertions as an appendix, but they are almost copies of the proofs in the original papers.

We introduce the terms and notations used in this paper. 
Throughout, the term ‘definable’ means ‘definable in the given structure with parameters’ unless  we describe the parameter set explicitly.
For any expansion of a linear order whose universe is $M$, we assume that $M$ is equipped with the order topology induced from the linear order $<$ and the topology on $M^n$ is the product topology of the order topology on $M$ unless the topology in consideration is explicitly described.
An open interval is a set of the form $(b_1,b_2):=\{x \in M\;|\; b_1<x<b_2\}$ for some $b_1,b_2 \in M \cup \{\pm \infty\}$ and open box is the product of open intervals.
For a topological space $X$, $\myint_{X}(A)$, $\mycl_{X}(A)$ and $\partial_{X}(A)$ denote the interior, the closure and the frontier of a subset $A$ of $X$, respectively.
We drop the subscript $X$ when the topological space $X$ is obvious from the context.

\section{Definitions and examples}\label{sec:definition}

The purpose of this section is to introduce some families of expansions of dense linear orders without endpoints called locally l-visceral structures and $*$-locally weakly o-minimal structures etc.
We first recall the notions of o-minimality, local o-minimality and their relatives.
\begin{definition}\label{def:o-min}
	An \textit{o-minimal} structure $\mathcal M=(M,<,\ldots)$ is an expansion of a dense linear order without endpoints such that every definable subset $X$ of $M$ is a union of a finite set and finitely many open intervals \cite{vdD}.
	
	A \textit{locally o-minimal} structure $\mathcal M=(M,<,\ldots)$ is an expansion of a dense linear order without endpoints such that, for every definable subset $X$ of $M$ and for every point $a\in M$, we can choose an open interval $I$ containing the point $a$ so that $X \cap I$ is a union of a finite set  and finitely many open intervals \cite{TV}.
	
	A locally o-minimal structure $\mathcal M=(M,<,\ldots)$ is a \textit{uniformly locally o-minimal structure of the second kind} if, for any positive integer $n$, any definable set $X \subseteq M^{n+1}$, $a \in M$ and $b \in M^n$, there exist an open interval $I$ containing the point $a$ and an open box $B$ containing $b$ such that the definable set $X_x \cap I$ is a union of a finite set  and finitely many open intervals for each $x \in B$, where $X_x$ is the fiber given by $X_x:=\{y \in M\;|\; (x,y) \in X\}$ \cite{Fuji_uniform}.
	We call it a \textit{uniformly locally o-minimal structure of the first kind} when we can choose $B=M^n$ \cite{Fuji_uniform}. 
	A uniformly locally o-minimal structure of the first kind is called a uniformly locally o-minimal structure in \cite{KTTT}.
	
	A locally o-minimal structure $\mathcal M=(M,<,\ldots)$ is called a \textit{strongly locally o-minimal structure} if, for every point $a \in M$, there exists an open interval $I$ containing the point $a$ such that $X \cap I$ is a union of a finite set  and finitely many open intervals for every univariate definable set $X$ \cite{TV}.  
	
	An expansion of a dense linear order without endpoints called \textit{almost o-minimal} if any univariate bounded definable set is a union of a finite set  and finitely many open intervals \cite{Fuji_almost}.
\end{definition}

In the above definition, the term `o-minimal' and the phrase `a union of a finite set  and finitely many intervals' are repeated many times.
We can generate similar notions by changing them with other terms and phrases.

Observe that the definition of o-minimality is of the following form:
A $\myx$ structure is an expansion of a dense linear order without endpoints $\mathcal M=(M,<,\ldots)$ such that each  definable subset $X$ of $M$ enjoys property $\myp(X)$.
Here, the blank $\myx$ and the property $\myp(X)$ are given as follows:
\begin{align*}
&\myx = \text{o-minimal}; \\
&\myp(X)=\text{``}X \text{ is a union of a finite set  and finitely many open intervals.''}
\end{align*}
The definition of weak o-minimality \cite{MMS} is of the same form, and the blank $\myx$ and the property $\myp(X)$ are given by 
\begin{align*}
	&\myx = \text{weakly o-minimal}; \\
	&\myp(X)=\text{``}X \text{ is a union of a finite set and finitely many open convex sets.''}
\end{align*}
Mathews proposed the notion of t-minimality in \cite[Definition 2.4]{Mathews}.
The blank $\myx$ and the property $\myp(X)$ are given as follows:
\begin{align*}
	&\myx = \text{t-minimal}; \\
	&\myp(X)=\text{``the boundary of }X \text{ is a finite set.''}
\end{align*}

The definitions of d-minimality and viscerality have slightly different tastes but similar structures.
They are defined so that every structure elementary equivalent to $\mathcal M$ rather than $\mathcal M$ alone possesses a simple property.
Miller and Fornasiero proposed the notion of d-minimality \cite{Fornasiero,Miller-dmin}.
We set
\begin{align*}
	&\myp(X)=\text{``}X \text{ is a union of an open set and finitely many discrete sets.''}
\end{align*}
An expansion of a dense linear order without endpoints $\mathcal M=(M,<,\ldots)$ is \textit{d-minimal} if it is definably complete and every univariate set $X$ definable in every structure elementary equivalent to $\mathcal M$ possesses the above property $\myp(X)$.

Dolich and Goodrick proposes the notion of viscerality \cite{DG}, and  $\myp(X)$ is given as follows:
\begin{align*}
&\myp(X)=\text{``}X \text{ is a union of an open set and a finite set.''} 
\end{align*}
An expansion of a dense linear order without endpoints $\mathcal M=(M,<,\ldots)$ is \textit{visceral} if every univariate set $X$ definable in every structure elementary equivalent to $\mathcal M$ possesses the above property $\myp(X)$.
Note that we consider t-minimality and viscerality in a restricted case. They are defined in \cite{DG,Mathews} even when the structure is not an expansion of an order.
The reference \cite{Johnson_visceral} also treats viscerality comprehensively.
We define a lesser version of viscerality for later use.
\begin{definition}
	An expansion of a dense linear order without endpoints $\mathcal M=(M,<,\ldots)$ is called \textit{l-visceral} (lesser-visceral) if every definable subset of $M$ is a union of a finite set and an open set. 
\end{definition}
Note that, in the above definition, we do not give any constraint to structures elementarily equivalent to $\mathcal M$. 
%

We say that \textit{$\myx$ structures are characterized by property $\myp(X)$ for every univariate definable set $X$} when a structure is a $\myx$ structures if and only if the property $\myp(X)$ holds for every univariate definable set $X$.
O-minimal structures, weakly o-minimal structures, t-minimal structures and l-visceral structures fall into this category, but d-minimal structures and visceral structures do not.

In the same manner as Definition \ref{def:o-min}, we can define locally $\myx$ structure etc. in a generalized situation.
\begin{definition}
	Assume that $\myx$ structures are characterized by property $\myp(X)$ for every univariate definable set $X$.
	
	An expansion of a dense linear order without endpoints $\mathcal M=(M,<,\ldots)$ is \textit{locally $\myx$} if, for every definable subset $X$ of $M$ and for every point $a\in M$, there exists an open interval $I$ containing the point $a$ such that $\myp(X \cap I)$ holds.
	
	A locally $\myx$ structure $\mathcal M=(M,<,\ldots)$ is a \textit{uniformly locally $\myx$ structure of the second kind} if, for any positive integer $n$, any definable set $X \subseteq M^{n+1}$, $a \in M$ and $b \in M^n$, there exist an open interval $I$ containing the point $a$ and an open box $B$ containing $b$ such that $\myp(X_x \cap I)$ hold for all $x \in B$.
	We call it a \textit{uniformly locally $\myx$ structure of the first kind} when we can choose $B=M^n$. 
	
	A locally $\myx$ structure $\mathcal M=(M,<,\ldots)$ is called a \textit{strongly locally $\myx$ structure} if, for every point $a \in M$, there exists an open interval $I$ containing the point $a$ such that $\myp(X \cap I)$ holds for every univariate definable set $X$.  
	
	An expansion of a dense linear order without endpoints called \textit{almost $\myx$} if any univariate bounded definable set $X$ possesses the property $\myp(X)$.
\end{definition}

The following implication is obvious:
\begin{align*}
	&\myx \Rightarrow \text{almost } \myx \Rightarrow \text{strong local } \myx\\ & \Rightarrow \text{uniform local } \myx \text{ of the first kind}\\
	 & \Rightarrow \text{uniform local }  \myx \text{ of the second kind} \Rightarrow \text{local } \myx
\end{align*}

We do not always obtain a new concept every time when we substitute different terms into the blank. 
For instance, the notion of local weak o-minimality coincides with that of local o-minimality.
\begin{proposition}\label{prop:local_ominmal}
A locally weakly o-minimal structure is locally o-minimal.
\end{proposition}
\begin{proof}
	We proceed the proof similarly to \cite[Proposition 2.2]{TV}.
	Let $M$ be the universe of a locally weakly o-minimal structure and fix an arbitrary point $a$ in $M$.
	Let $X$ be a definable subset of $M$.
	We prove that there exists an open interval $I$ containing the point $a$ such that $X \cap I$ is a union of a finite set and finitely many open intervals.
	
	We can take an open interval $J$ such that $X \cap J$ is a union of a finite set and finitely many convex open sets.
	We can find $b_1 \in J$ such that $b_1<a$ and either $(b_1,a) \cap X=\emptyset$ or $(b_1,a) \subseteq X$.
	We can also take $b_2 \in J$ such that $a<b_2$ and either $(a,b_2) \cap X=\emptyset$ or $(a,b_2) \subseteq X$.
	Set $I=(b_1,b_2)$.
	Then $X \cap I$ is a union of elements of a subfamily of $\{\{a\},(b_1,a),(a,b_2)\}$.
	It means that $X \cap I$ is a union of a finite set and finitely many open intervals.
\end{proof}

\begin{remark}\label{rem:local_omin}
The proof of Proposition \ref{prop:local_ominmal} implies that locally o-minimal structures $\mathcal M=(M,<,\ldots)$ are characterized as follows:
\begin{quote}
	For any point $a \in M$ and any univariate definable set $X$, we can find $b_1,b_2 \in M$ such that $b_1<a<b_2$ and the intersection $X \cap (b_1,b_2)$ is a union of elements of a subfamily of $\{\{a\},(b_1,a),(a,b_2)\}$.
\end{quote} 
\end{remark}

Proposition \ref{prop:local_ominmal} means that the notion of local weak o-minimality does not provide a new notion.
It is not the case for the notion of uniform local weak o-minimality of the second kind.
\begin{example}
	We gave an example of a weakly o-minimal structure which is not a uniformly locally o-minimal structure of the second kind in \cite[Example 2.3]{Fuji_uniform}.
	Recall that a weakly o-minimal structure is a uniformly locally weakly o-minimal structure of the second kind.
	It illustrates that a uniformly locally weakly o-minimal structure of the second kind is not necessarily a uniformly locally o-minimal structure of the second kind.
\end{example}

We next consider local l-viscerality.
By definition, a locally l-visceral structure $\mathcal M=(M,<,\ldots)$ is an expansion of a dense linear order without endpoints such that, for every element $a \in M$ and every univariate definable set $X$, the intersection $I \cap X$ is a union of a finite set and an open set for some open interval $I$ containing the point $a$.
We give another characterization of local l-viscerality.
\begin{proposition}\label{prop:char_lvisceral}
An expansion of a dense linear order without endpoints is locally l-visceral if and only if any univariate definable set is a union of an open set and a discrete closed set. 
In particular, local l-viscerality is preserved under elementary equivalence and an ultraproduct of locally l-visceral structures is locally l-visceral.
\end{proposition}
\begin{proof}
	Let $M$ be the universe.
	We first prove `only if' part.
	Let $X$ be a univariate definable set.
	Set $Y=X \setminus \myint(X)$.
	We have only to prove that $Y$ is discrete and closed.
	Let $x \in M$ be an arbitrary point.
	By local l-viscerality, there exists an open interval $I$ containing the point $x$ such that $I \cap Y$ is a finite set.
	It implies that $Y$ is discrete and closed.
	
	We next consider the `if' part.
	Let $X$ be a univariate definable set.
	The definable set $X$ is the union of an open set $U$ and a discrete closed set $D$.
	Let $x \in M$ be an arbitrary point.
	Since $D$ is discrete and closed, we can take an open interval $I$ containing the point $x$ such that $I \cap D$ consists of at most one point.
	The intersection $I \cap X$ is a union of the finite set $I \cap D$ and the open set $I \cap U$.
	It implies that the structure is locally l-visceral.
	
	We finally prove `in particular' part.
	We can describe that a univariate definable set is discrete and closed after removing its interior by a first-order formula using the formula defining the univariate definable set.
	Therefore, local l-viscerality is preserved under elementary equivalence, and an ultraproduct of locally l-visceral structures is locally l-visceral.
\end{proof}

We next recall the definition of definable completeness.
\begin{definition}[\cite{M}]
	An expansion of a dense linear order without endpoints $\mathcal M=(M,<,\ldots)$ is \textit{definably complete} if any definable subset $X$ of $M$ has the supremum and  infimum in $M \cup \{\pm \infty\}$.
\end{definition}

When we study structures which are not necessarily definably complete, we sometimes need to consider definable Dedekind completion such as the study of weakly o-minimal structures \cite{MMS}.
\begin{definition}
	Let $\mathcal M=(M,<,\ldots)$ be an expansion of a dense linear order without endpoints.
	A \textit{definable gap} is a pair $(A,B)$ of nonempty definable subsets of $M$ such that 
	\begin{itemize}
		\item $M=A \cup B$;
		\item $a<b$ for all $a \in A$ and $b \in B$;
		\item $A$ does not have a largest element and $B$ does not have a smallest element.
	\end{itemize}
	Note that 'being a definable gap' is expressed by a first-order formula when formulas defining $A$ and $B$ are given.
	
	Set $\overline{M}=M \cup \{\text{definable gaps in }M\}$. 
	We can naturally extend the order $<$ in $M$ to an order in $\overline{M}$, which is denoted by the same symbol $<$.
	For instance, we define $a<(A,B)$ by $a \in A$ when $a \in M$ and $(A,B)$ is a definable gap.
	We define $(A_1,B_1) \leq (A_2, B_2)$ by $A_1 \subseteq A_2$ for definable gaps  $(A_1,B_1)$ and $(A_2, B_2)$.
	The linearly ordered set $(\overline{M},<)$ is called the \textit{definable Dedekind completion} of $(M,<)$.
	
	An arbitrary open interval $I$ in $M$ is of the form $(b_1,b_2)$, where $b_1,b_2 \in M \cup \{\pm \infty\}$.
	We set $\overline{I}=\{x \in \overline{M}\;|\; b_1<x<b_2\}$.
	We define $\overline{I}$ for every interval $I$ in $M$ in the same manner.
	Let $B$ be a definable open box in $M^n$; that is, there are open intervals $I_i$ in $M$ for each $1 \leq i \leq n$ such that $B=\prod_{i=1}^n I_i \subseteq M^n$.
	We set $\overline{B}=\prod_{i=1}^n \overline{I_i} \subseteq \overline{M}^n$.
	Note that $I$ and $M$ are subsets of $M$ and $M^n$, respectively, but overlined ones $\overline{I}$ and $\overline{B}$ are subsets of $\overline{M}$ and $\overline{M}^n$, respectively.
	Throughout, we use these overlined notations to represent Dedekind completions and their subsets defined above. 
	
	For a nonempty definable subset $X$ of $M$, we define $\sup X \in \overline{M} \cup \{+\infty\}$ as follows:
	We set $\sup X=+\infty$ when, for any $a \in M$, there exists $x \in X$ with $x>a$.
	Assume that there exists $z \in M$ such that $x<z$ for every $x \in X$.
	Set $B=\{y \in M\;|\; \forall x \in X\ y>x\}$ and $A=M \setminus B$.
	If $B$ has a smallest element $m$, we set $\sup X=m$.
	If $A$ has a largest element $m'$, we set $\sup X=m'$.
	Finally, if $(A,B)$ is a definable gap, we set $\sup X=(A,B) \in \overline{M}$.
	We define $\inf X$ similarly.   
\end{definition}

We also give a definition of definable function into the definable Dedekind completion of the underlying set.
\begin{definition}
	Let $\mathcal M=(M,<,\ldots)$ be an expansion of a dense linear order without endpoints.
	Let $X$ be a definable subset of $M^n$.
	A \textit{definable function} $F:X \to \overline{M}$ is defined as follows:
	Let $\pi:M^{n+1} \to M^n$ be the coordinate projection forgetting the last coordinate.
	There exists a definable subset $Y$ of $M^{n+1}$ such that $\pi(Y)=X$ and $F(x)=\sup Y_x$ for $x \in X$, where $Y_x:=\{y \in M\;|\; (x,y) \in Y\}$.
	We consider that definable subsets $Y$ and $Y'$ of $M^{n+1}$ define an identical definable function $F:X \to \overline{M}$ if $\pi(Y)=\pi(Y')=X$ and $\sup Y_x=\sup Y'_x$ for each $x \in X$. 
		
		A definable function $F:X \to \overline{M} \cup \{\pm \infty\}$ is a pair of a decomposition $X=X_{\overline{M}} \cup X_{+\infty} \cup X_{-\infty}$ into definable sets and a definable function $f:X_{\overline{M}} \to \overline{M}$.
		We set $F(x)=f(x)$ when $x \in X_{\overline{M}}$, $F(x)=+\infty$ when $x \in X_{+\infty}$ and $F(x)=-\infty$ when $x \in X_{-\infty}$.
		We consider that $+\infty$ is larger than all element in $\overline{M}$ and $-\infty$, and $-\infty$ is smaller than all element in $\overline{M}$ and $+\infty$.
\end{definition}

We next give an alternative version of definition of a locally $\myx$ structure called a $*$-locally $\myx$ structure.
\begin{definition}
	Assume that $\myx$ structures are characterized by property $\myp(X)$ for every univariate definable set $X$.
	An expansion of a dense linear order without endpoints $\mathcal M=(M,<,\ldots)$ is a \textit{$*$-locally $\myx$ structure} if, for every definable subset $X$ of $M$ and for every point $\overline{a}\in \overline{M}$, there exists an open interval $I$  such that $\overline{a} \in \overline{I}$ and $\myp(X \cap I)$ holds.
	It is trivial that a $*$-locally $\myx$ structure is a locally $\myx$ structure.
	
	A $*$-locally $\myx$ structure $\mathcal M=(M,<,\ldots)$ is called a \textit{strongly $*$-locally $\myx$ structure} if, for every point $\overline{a} \in \overline{M}$, there exists an open interval $I$ such that $\overline{a} \in \overline{I}$ and $\myp(X \cap I)$ holds for every univariate definable set $X$.  
\end{definition}

Note that a $\myx$ structure is a strongly $*$-locally $\myx$ structure because we can take $I=M$ in the definition.
We can slightly relax the assumption of this observation.

\begin{proposition}\label{prop:almostimplies}
	Assume that $\myx$ structures are characterized by property $\myp(X)$ for every univariate definable set $X$.
	An almost $\myx$ structure is a strongly $*$-locally $\myx$ structure.
\end{proposition}
\begin{proof}
	Let $\mathcal M=(M,<,\ldots)$ be an almost $\myx$ structure.
	Take an arbitrary element $\overline{a} \in \overline{M}$.  
	Take a bounded open interval $I$ such that $\overline{a} \in \overline{I}$.
	Let $X$ be an arbitrary univariate definable set.
	We have $\myp(X \cap I)$ because $\mathcal M$ is an almost $\myx$ structure.
	This means that $\mathcal M$ is a strongly $*$-locally $\myx$ structure.
\end{proof}

The notion of $*$-locally o-minimal structure does not provide a new category of locally o-minimal structures.
\begin{proposition}\label{prop:*-local_omin}
A $*$-locally o-minimal structure is definably complete.
\end{proposition}
\begin{proof}
	Let $\mathcal M=(M,<,\ldots)$ be a $*$-locally o-minimal structure.
	Let $X$ be a definable subset of $M$.
	We have only to prove that $\sup X, \inf X \in M \cup \{\pm \infty\}$.
	We only prove $\sup X \in M \cup \{\pm \infty\}$ here, and we can prove $\inf X \in M \cup \{\pm \infty\}$ similarly.
	When $X$ is unbounded above, we always have $\sup X=+\infty$.
	Assume that $X$ is bounded above, we have $\overline{x} = \sup X \in \overline{M}$ by the definition of definable Dedekind completion.
	Take an open interval $I$ such that $\overline{x} \in \overline{I}$ and $I \cap X$ is a union of finitely many open intervals and a finite set.
	We have $\overline{x} = \sup X = \sup (I \cap X)$ because $\overline{x} \in \overline{I}$.
	The supremum $\sup (I \cap X)$ belongs to $M$ because $I \cap X$ is a union of finitely many open intervals and a finite set.
	This means that $\sup X \in M$.
\end{proof}

The equality $\overline{M}=M$ holds if $\mathcal M=(M,<,\ldots)$ is definably complete.
A locally o-minimal structure is definably complete if and only if it is $*$-locally o-minimal by Proposition \ref{prop:*-local_omin} and the above fact.

The following lemma illustrates that an assertion similar to Proposition \ref{prop:*-local_omin} holds for $*$-locally l-visceral structures  in a special case.

\begin{lemma}\label{lem:supinf}
	Consider a $*$-locally l-visceral structure $\mathcal M=(M,<,\ldots)$ and $X$ be a univariate definable subset having an empty interior.
	Then the supremum $\sup X$ and the infimum $\inf X$ belong to $M \cup \{\pm \infty\}$.
	In addition, they belong to $X$ if they do not belong to $\{\pm \infty\}$.
\end{lemma}
\begin{proof}
	We only consider $\sup X$.
	We can treat $\inf X$ similarly.
	The lemma is obvious when $X$ is not bounded from above.
	We consider the case in which $X$ is bounded from above.
	Set $s=\sup X \in \overline{M}$.
	We can take an open interval $I$ such that $s \in \overline{I}$ and $X \cap I$ is a finite set.
	It is obvious that $s= \sup(X \cap I)$ by the definition of $s$ and $\sup(X \cap I) \in X \cap I \subseteq X$ because the set $X \cap I$ is a finite set.
\end{proof}

We next show that $*$-local weak o-minimality is preserved under elementary equivalence and an ultraproduct of $*$-locally weakly o-minimal structures is $*$-locally weakly o-minimal.
We first prove the following technical lemma:

\begin{lemma}\label{lem:*-technical}
	Consider a $*$-locally weakly o-minimal structure $\mathcal M=(M,<,\ldots)$.
	Let $\overline{a}=(A,B)$ be a definable gap in $M$ and $X$ be a definable subset of $M$.
	Then, there exists $c_1,c_2 \in M$ such that $c_1<\overline{a}<c_2$ and the intersection $I \cap X$ coincides with one of the sets $\emptyset$, $I$, $I \cap A$ and $I \cap B$, where $I$ is the interval $(c_1,c_2)$.
\end{lemma}
\begin{proof}
	By $*$-local weak o-minimality, we can take an open interval $J$ such that $\overline{a} \in \overline{J}$ and the intersection $X \cap J$ is a union of a finite set $F$ and finitely many convex open sets $C_1, \ldots, C_m$.
	We may assume that $J$ is bounded without loss of generality.
	We have $\overline{a} \notin F$ because $\overline{a} \notin M$.
	We can take an open interval $J'$ such that $\overline{a} \in \overline{J'}$ and $J' \cap F=\emptyset$ because $F$ is a finite set.
	We may assume that $F=\emptyset$ by replacing $J$ with $J'$.
	Furthermore, we may assume that $C_i \cup C_j$ is not convex whenever $i \neq j$ by replacing $C_i$ and $C_j$  with $C_i \cup C_j$ when $C_i \cup C_j$ is convex. 
	We may further assume that $C_1<C_2< \cdots <C_m$ by permuting the order if necessary.
	Here, the notation $C_i<C_{i+1}$ means that $x_1<x_2$ whenever $x_1 \in C_i$ and $x_2 \in C_{i+1}$.
	We can easily prove that $C_i$ are definable, but we omit the proof. 
	
	Let $\alpha$ and $\beta$ be the left and right endpoints of $J$, respectively.
	Set $\alpha_i=\inf C_i \in \overline{M}$ and $\beta_i = \sup C_i \in \overline{M}$ for $1 \leq i \leq m$.
	Observe that $\alpha_i<\beta_i$ for each $1 \leq i \leq m$ and $\beta_i < \alpha_{i+1}$ for each $1 \leq i <m$.
	We used the assumption that $C_i \cup C_{i+1}$ is not convex when we deduce the second inequality.
	Put $c_1=\alpha$ and $c_2=\beta$ when $X \cap J$ is an empty set.
	We assume that $X \cap J$ is nonempty in the rest of proof.
	We consider the following cases separately:
	\begin{enumerate}
		\item[(1)] The first case is the case in which $\overline{a}<\alpha_1$.
		We can take $c_2 \in M$ so that $\overline{a} < c_2 < \alpha_1$.
		Set $c_1=\alpha$.
		It is obvious that $c_1<\overline{a}<c_2$.
		The intersection $X \cap (c_1,c_2)$ is an empty set in this case.
		\item[(2)] We next consider the case in which $\overline{a}>\beta_m$.
		We can take $c_1,c_2 \in M$ so that $c_1<\overline{a}<c_2$ and the intersection $X \cap (c_1,c_2)$ is an empty set in the same manner as case (1).
		\item[(3)] We consider the case in which $\alpha_k<\overline{a}<\beta_k$ for some $1 \leq k \leq m$.
		Take $c_1,c_2 \in M$ so that $\alpha_k<c_1<\overline{a}<c_2<\beta_k$.
		We have $(c_1,c_2) \cap X=(c_1,c_2)$ in this case.
		\item[(4)] We consider the case in which $\beta_k<\overline{a}<\alpha_{k+1}$ for some $1 \leq k < m$.
		Choose $c_1,c_2 \in M$ so that $\beta_k<c_1<\overline{a}<c_2<\alpha_{k+1}$.
		The intersection $(c_1,c_2) \cap X$ is an empty set.
		\item[(5)] We consider the case in which $\overline{a}=\alpha_k$ for some $1 \leq k \leq m$.
		When $k=1$, put $c_1 =\alpha$ and take $c_2 \in C_1$.
		Then, $(c_1,c_2) \cap X=(c_1,c_2) \cap B$.
		When $k>1$, then we have $\alpha_k=\overline{a}>\beta_{k-1}$.
		Choose $c_1$ so that $\beta_{k-1}<c_1<\overline{a}$. Take $c_2 \in C_k$.
		Then, we get $(c_1,c_2) \cap X=(c_1,c_2) \cap B$.
		\item[(6)] The remaining case is the case in which $\overline{a}=\beta_k$ for some $1 \leq k \leq m$.
		We can choose $c_1,c_2 \in M$ so that the equality $(c_1,c_2) \cap X=(c_1,c_2) \cap A$ holds in the same manner as case (5).
	\end{enumerate}
\end{proof}

The above lemma and Remark \ref{rem:local_omin} imply that 'being *-locally weakly o-minimal' can be represented by first-order formulas.

\begin{proposition}\label{prop:*-local_elementary}
	$*$-local weak o-minimality is preserved under elementary equivalence.
	In other word, an $\mathcal L$-structure $\mathcal N=(N,<,\ldots)$ is a $*$-locally weakly o-minimal structure when it is elementarily equivalent to a $*$-locally weakly o-minimal $\mathcal L$-structure $\mathcal M=(M,<,\ldots)$.
\end{proposition}
\begin{proof}
We have only to show that, for any $\overline{a} \in \overline{N}$ and a definable subset $X$ of $N$, there exists an open interval $I$ such that $\overline{a} \in \overline{I}$ and $I \cap X$ is a union of a finite set and finitely many open convex sets.
Here, $\overline{N}$ denotes the definable Dedekind completion of $N$. 

We consider two separate cases.
The first case is the case in which $\overline{a} \in N$.
The structure $\mathcal M$ is locally o-minimal by Proposition \ref{prop:local_ominmal} because a $*$-locally weakly o-minimal structure is obviously locally weakly o-minimal.
Local o-minimality is preserved under elementary equivalence by \cite[Corollary 2.5]{TV}.
This implies that $\mathcal N$ is locally o-minimal.
We can take an open interval $I$ containing the point $\overline{a}$ so that $I \cap X$ is a union of a finite set and finitely many open intervals.

The remaining case is the case in which $\overline{a}$ is a definable gap, say $\overline{a}=(A,B)$.
We can take $\mathcal L$-formulas $\phi_A(x,\overline{y})$, $\phi_B(x,\overline{y})$ and $\phi_X(x,\overline{y})$ so that  $\phi_A(x,\overline{c})$, $\phi_B(x,\overline{c})$ and $\phi_X(x,\overline{c})$ define the definable sets $A$, $B$ and $X$, respectively.
Here, $\overline{c}$ is a tuple of parameters from $N$.
For every $\overline{y}$, we denote the definable set defined by $\phi_A(x,\overline{y})$ by $A(\overline{y})$.
We define $B(\overline{y})$ and $X(\overline{y})$ in the same manner.
Let $\Phi(\overline{y})$ be a first-order $\mathcal L$-formula representing ``the pair $(A(\overline{y}),B(\overline{y}))$ is a definable gap."
Such a formula exists by the definition of definable gaps.
We next consider an $\mathcal L$-formula $\Psi(\overline{y})$ saying that there exist $c_1 \in A(\overline{y})$ and $c_2 \in B(\overline{y})$ such that $I \cap X(\overline{y})$ is one of the sets $\emptyset$, $I$, $I \cap A(\overline{y})$ and $I \cap B(\overline{y})$, where we denote the open interval $(c_1,c_2)$ by $I$. 
Since $\mathcal M$ is $*$-locally weakly o-minimal, we have $\mathcal M \models \forall \overline{y} (\Phi(\overline{y} ) \rightarrow \Psi(\overline{y} ))$ by Lemma \ref{lem:*-technical}.
We have $\mathcal N \models \forall \overline{y} (\Phi(\overline{y} ) \rightarrow \Psi(\overline{y} ))$ because $\mathcal N$ is elementarily equivalent to $\mathcal M$.
We have $\mathcal N \models \Psi(\overline{c} )$ because  $\mathcal N \models \Phi(\overline{c} )$. 
We can find an open interval $I$ so that $\overline{a} \in \overline{I}$ and $I \cap X$ is either an empty set or an open convex set.
\end{proof}

\begin{proposition}\label{prop:*-local_ultraproduct}
	An ultraproduct of $*$-locally weakly o-minimal structures is $*$-locally weakly o-minimal.
\end{proposition}
\begin{proof}
	We can prove it similarly to Proposition \ref{prop:*-local_elementary} using the \L o\'s's theorem, Remark \ref{rem:local_omin} and Lemma \ref{lem:*-technical}.
	We omit the proof.
\end{proof}

A model of DCTC is defined in  \cite{S}.
We define models of TC and $*$-TC in the same manner as \cite{S}.
\begin{definition}
	Consider an expansion of a dense linear order without endpoints $\mathcal M=(M,<,\ldots)$.
	For any $a \in \overline{M}$, we define a partial type $a^+$ as follows:
	Let $\phi(x)$ be a unary formula with parameters.
	We say $\phi \in a^+$ if there exists $b \in M$ such that $b>a$ and $\mathcal M \models \phi(c)$ for any element $c$ with $a<c<b$.
	We define partial types $a^-$, $\infty^-$ and $(-\infty)^+$, similarly.
	 
	 The structure $\mathcal M$ is called a \textit{model of TC} if $a^+$ and $a^-$ are complete types for every $a \in M \cup \{\pm \infty\}$.
	 It is called a \textit{model of $*$-TC} if $a^+$ and $a^-$ are complete types for every $a \in \overline{M} \cup \{\pm \infty\}$.
	 Here, TC is an abbreviation of 'type completeness.' 
\end{definition}

By Remark \ref{rem:local_omin} and Lemma \ref{lem:*-technical}, the definitions of models of TC and $*$-TC are almost the same as those of local o-minimality and $*$-local weak o-minimality.
The only difference between them is that the additional condition on $\infty^-$ and $(-\infty)^+$ are imposed in the definition of models of TC and $*$-TC.
The following assertion is proven in the same manner as the locally o-minimal case and the $*$-locally weakly o-minimal case. We omit the proof.
\begin{proposition}\label{prop:TC}
The following assertions hold:
\begin{enumerate}
	\item A model of TC is locally o-minimal;
	\item A model of $*$-TC is $*$-locally weakly o-minimal;
	\item A structure elementarily equivalent to a model of $*$-TC is a model of $*$-TC;
	\item An ultraproduct of models of $*$-TC is a model of $*$-TC. 
\end{enumerate}
\end{proposition}

We treat $*$-locally weakly o-minimal structures in Section \ref{sec:star-cont} and Section \ref{sec:*-local}.
Proposition \ref{prop:*-local_ultraproduct} gives a procedure to construct $*$-local weakly o-minimal structures.
Ultraproducts of weakly o-minimal structures are $*$-locally o-minimal structures by Proposition \ref{prop:*-local_ultraproduct}.
In addition, they are model of $*$-TC by Proposition \ref{prop:TC} because weakly o-minimal structures are models of $*$-TC.

We can define three alternative versions of uniformly locally $\myx$ structures.
We do not treat these concepts in this paper, but introduce them for possible future use.

\begin{definition}
	The first version is as follows:
	A $*$-locally $\myx$ structure $\mathcal M=(M,<,\ldots)$ is a \textit{uniformly $*$-locally $\myx$ structure of the second kind} if, for any positive integer $n$, any definable set $X \subseteq M^{n+1}$, $\overline{a} \in \overline{M}$ and $b \in M^n$, there exist an open interval $I$ with $\overline{a} \in \overline{I}$ and an open box $B$ containing $b$ such that $\myp(X_y \cap I)$ holds for all $y \in B$.
	We call it a \textit{uniformly $*$-locally $\myx$ structure of the first kind} when we can choose $B=M^n$. 
	
	The second version is as follows:
	A locally $\myx$ structure $\mathcal M=(M,<,\ldots)$ is a \textit{$*$-uniformly locally $\myx$ structure of the second kind} if, for any positive integer $n$, any definable set $X \subseteq M^{n+1}$, ${a} \in {M}$ and $\overline{b} \in \overline{M}^n$, there exist an open interval $I$ containing the point $a$ and an open box $B$ with $\overline{b} \in \overline{B}$ such that $\myp(X_y \cap I)$ holds for all $y \in B$.
	
	The final version is as follows:
	A $*$-locally $\myx$ structure $\mathcal M=(M,<,\ldots)$ is a \textit{$*$-uniformly $*$-locally $\myx$ structure of the second kind} if, for any positive integer $n$, any definable set $X \subseteq M^{n+1}$, $\overline{a} \in \overline{M}$ and $\overline{b} \in \overline{M}^n$, there exist an open interval $I$ containing with $\overline{a} \in \overline{I}$ and an open box $B$ with $\overline{b} \in \overline{B}$ such that $\myp(X_y \cap I)$ holds for all $y \in B$.
	
	We decorate the term `locally' by the symbol $*$ when we are allowed to choose $\overline{a}$ from the definable Dedekind completion $\overline{M}$, and we decorate the term `uniformly' by $*$ when we may choose $\overline{b}$ from the Cartesian product of the definable Dedekind completions $\overline{M}^n$.
\end{definition}

We can easily construct various locally o-minimal structures using the notion of simple product.
We first recall the definition of simple products.
\begin{definition}[\cite{KTTT}]
	Let $\mathcal M_1=(M_1,\ldots)$ and $\mathcal M_2=(M_2, \ldots)$ be model-theoretic structures.
	The simple product of $\mathcal M_1$ and $\mathcal M_2$ is the structure such that the underlying set is the Cartesian product $M:=M_1 \times M_2$, and a subset of $M^n$ is definable in the simple product if and only if it is a union of finitely many sets of the form:
	\begin{align*}
		\{((x_1,y_1),\ldots, (x_n,y_n)) \in M^n\;|\; (x_1,\ldots, x_n) \in X_1, (y_1,\ldots, y_n) \in X_2\},
	\end{align*}
	where $X_i$ are subsets of $M_i^n$ definable in $\mathcal M_i$ for $i=1,2$.
	When $\mathcal M_1$ and $\mathcal M_2$ have definable orders $<_1$ and $<_2$, respectively, we consider the lexicographical order $<$ on $\mathcal M$.
\end{definition}

\begin{remark}
	Every model of a theory $T$ is not necessarily a simple product even if $T$ has a model which is a simple product.
	The author introduced a procedure to construct a theory whose models are all simple products in \cite{Fuji_product}.
\end{remark}

We construct an almost weakly o-minimal structure which is not l-visceral.
It is also an example of $*$-locally weakly o-minimal structure by Proposition \ref{prop:almostimplies}. 
\begin{example}\label{ex:1}
	Set $\mathcal M_1=(\mathbb Z,<_{\mathbb Z})$ and let $\mathcal M_2=(M_2,<_2,\ldots)$ be a weakly o-minimal structure.
	Their simple product $\mathcal M=(M,<,\ldots)$ is an almost weakly o-minimal structure, but it is not l-visceral.
\end{example}
	\begin{proof}
	We set
	\begin{align*}
		&\myp(X)=\text{``}X \text{ is a union of a finite set and finitely many open convex sets''}.
	\end{align*}
	Let $X$ be a univariate bounded definable set of $M$.
	We can take $(a_1,b_1),  (a_2,b_2)\in M$ so that $(a_1,b_1) <  x < (a_2,b_2)$ for every $x \in X$ because $X$ is bounded.
	By the definition of simple product, there exists a definable subset $X_i$ of $M_2$ such that $(\{i\} \times M_2) \cap X = \{i\} \times X_i$ for every $a_1 \leq i \leq a_2$.
	We have $X=\bigcup_{a_1 \leq i \leq a_2}\{i\} \times X_i$.
	It is obvious $\myp(X)$ holds because $\myp(X_i)$ holds for every $a_1 \leq i \leq a_2$.
	We have shown that $\mathcal M$ is an almost weakly o-minimal structure.
	
	Take $b \in M_2$.
	The set $\mathbb Z \times \{b\}$ is definable in $\mathcal M$.
	It is infinite and discrete.
	It means that $\mathcal M$ is not l-visceral.
\end{proof}

We can construct a strongly locally weakly o-minimal structure which is not $*$-locally weakly o-minimal in the same manner as Example \ref{ex:1}.
\begin{example}\label{ex:2}
	Set $\mathcal M_1=(\mathbb Z,<_{\mathbb Z})$ and let $\mathcal M _2=(M_2,<_2,\ldots)$ be an almost weakly o-minimal structure.
	Assume that there exists a univariate $\mathcal M_2$-definable set $Z$ such that $Z \cap \{x \in M\;|\; x<_2a\}$ and $Z \cap \{x \in M\;|\; a<_2x\}$ is infinite and discrete for every $a \in M_2$. 
	Such an $\mathcal M_2$ exists thanks to Example \ref{ex:1}.
	The simple product of $\mathcal M_1$ and $\mathcal M_2$ is a strongly locally weakly o-minimal structure which is not $*$-locally weakly o-minimal.
\end{example}
\begin{proof}
	We set
	\begin{align*}
		&\myp(X)=\text{``}X \text{ is a union of a finite set and finitely many open convex sets''}.
	\end{align*}
	We first show that $\mathcal M$ is a strongly locally weakly o-minimal structure.
	Take an arbitrary point $(a,b) \in \mathbb Z \times M_2=M$.
	Take $b_1,b_2 \in M_2$ so that $b_1<b<b_2$ and set $I=\{x \in M\;|\; (a,b_1)<x<(a,b_2)\}$.
	For any univariate $\mathcal M$-definable set $X$ in $M$, the intersection $X \cap I$ is of the form $\{a\} \times Y$, where $Y$ is a bounded $\mathcal M_2$-definable subset of $M_2$.
	The property $\myp(Y)$ holds because $\mathcal M_2$ is almost weakly o-minimal.
	We have shown that $\myp(X \cap I)$ holds for every univariate $\mathcal M$-definable set $X$.
	It means that $\mathcal M$ is strongly locally weakly o-minimal.
	
	We next prove that $\mathcal M$ is not $*$-locally weakly o-minimal.
	Take $a_1 \in \mathbb Z$ and set $a_2=a_1+1$.
	Set $A=\{(x,y) \in M=\mathbb Z \times M_2\;|\; x \leq a_1\}$ and $B=\{(x,y) \in M=\mathbb Z \times M_2\;|\; x \geq a_2\}$.
	The pair $\overline{x}=(A,B)$ is a definable gap and an element in $\overline{M}$.
	Consider an $\mathcal M$-definable set $X=\mathbb Z \times Z$, where $Z$ is an $\mathcal M_2$-definable set given in the example.
	Let $I$ be an arbitrary interval with $\overline{x} \in \overline{I}$.
	We want to show that $\myp(X \cap I)$ does not hold.
	Let $J$ be an open interval contained in $I$.
	The property $\myp(X \cap I)$ does not hold if $\myp(X \cap J)$ does not hold.
	Therefore, by shrinking $I$ if necessary, we may assume that the left and right endpoints of $I$ are of the form $(a_1,b_1)$ and $(a_2,b_2)$, respectively, for some $b_1,b_2 \in M_2$.
	Since both $(b_1,\infty) \cap Z$ and $(-\infty,b_2) \cap Z$ are discrete and infinite, $X \cap I$ is discrete and infinite.
	It means that $\myp(X \cap I)$ does not hold.
\end{proof}

\section{Monotonicity}\label{sec:monotonicity}
We prove local monotonicity theorems for uniformly locally weakly o-minimal structures of the first/second kind and locally o-minimal structures enjoying the univariate $*$-continuity property.
We first study uniformly locally weakly o-minimal structures of the first/second kind.

\subsection{Uniform locally weakly o-minimal case}\label{sec:uniform}

We prove a local monotonicity theorem when the structure is uniformly locally weakly o-minimal of the second kind.
We first recall a well-known fact:
\begin{lemma}\label{lem:local_dim0}
	Consider a locally o-minima structure $\mathcal M=(M,<,\ldots)$.
	A definable subset of $M$ is discrete and closed if it has an empty interior.
\end{lemma}
\begin{proof}
	It follows from Proposition \ref{prop:char_lvisceral} because locally o-minimal structures are locally l-visceral.
\end{proof}

We begin to prove a main result in this section following the argument in \cite[Section 3]{Fuji_uniform}.
We first prove the following lemma:

\begin{lemma}\label{lem:mono_key_lemma2}
	Let $\mathcal M=(M,<,\ldots)$ be a uniformly locally weakly o-minimal structure of the second kind.
	No definable functions defined on open intervals into $\overline{M}$ have local minimums throughout the intervals.
\end{lemma}
\begin{proof}
	We prove the lemma in the same way as weakly o-minimal structures \cite{A}.
	We lead to a contradiction assuming that $f:I \rightarrow \overline{M}$ is a definable function on an open interval $I$ which have the local minimum throughout $I$.
	
	For any $a \in I$, set $U_a$ as follows:
	\begin{align*}
		U_a&=\{x \in I\;|\; x>a \text{ and } f(y)>f(a) \text{ for all } a < y \leq x\} \cup \{a \} \cup\\
		& \{x \in M\;|\; x<a \text{ and }   f(y)>f(a) \text{ for all } x \leq y <a\}\text{.}
	\end{align*}
	The definable set $U_a$ is convex and it is a neighborhood of the point $a$ because $f$ is locally minimal at the point $a$.
	The notation $a \prec b$ denotes the relation $U_a \varsupsetneq U_b$.
	We can prove the following assertions (i) through (ix) only using the assumption that $f$ has local minimums throughout $I$.
	We can find the proof of the following claims in \cite{A} and \cite[Lemma 3.2]{Fuji_uniform}.
	In \cite[Lemma 3.2]{Fuji_uniform}, $f$ is assumed to be injective, but we can check that the injectivity is not necessary.
	We omit the proofs here.
	\begin{enumerate}[(i)]
		\item $U_a \not= U_b$ if $a \not= b$;
		\item $a \in \myint(U_a)$;
		\item $a \prec b \Leftrightarrow b \in U_a$;
		\item $U_a \cap U_b \not= \emptyset \Rightarrow a \prec b \text{ or } b \prec a$;
		\item $b \prec a \text{ and } c \prec a \Rightarrow b \prec c \text{ or } c \prec b$;
		\item $a \prec b \prec c \text{ and } a < c \Rightarrow a \leq b $;
		\item $a \prec b \prec c \text{ and } a > c \Rightarrow a \geq b$;
		\item $a \prec b \prec c \text{ and } a < b \Rightarrow a<c$;
		\item the definable set 
		\begin{equation*}
			C_a = \{x \in I\;|\; x \prec a\}
		\end{equation*}
		does not contain an open interval for any $a \in I$.
	\end{enumerate}
	Consider the definable set $C=\{(a,x) \in I^2\;|\; x \in C_a\}$.
	By shrinking $I$ if necessary, we may assume that $C_a$ is a union of an open set and a finite set because $\mathcal M$ is uniformly locally weakly o-minimal of the second kind.
	By (ix), $C_a$ is a finite set.

	Consider the following sets:
	\begin{align*}
		K &= \{ x \in I\;|\; y \not\prec x \text{ for all } y \in I\}\text{ and }\\
		\widetilde{a} &= \{ x \in I\;|\; a \prec x \text{ and } \nexists y \in I \ a \prec y \prec x\}\text{,}
	\end{align*}
	where $a$ is an element of $I$.
	Using (iii), (v) and the fact that $C_a$ is finite, we can easily show the following equality:
	\begin{equation}\label{eq:5}
		I \setminus K = \displaystyle\mathop{\dot{\bigcup}}_{a \in I}\widetilde{a}\text{.}
	\end{equation}
	The symbol $\mathop{\dot{\bigcup}}$ represents disjoint union.
	We can also demonstrate that neither $K$ nor $\widetilde{a}$ contain an open interval without difficulty using assertions (ii) and (iii).
	The definable sets $K$ and $\widetilde{a}$ are locally finite because $\mathcal M$ is locally weakly o-minimal.
	Since $\mathcal M$ is a uniformly locally weakly o-minimal structure of the second kind, shrinking $I$ if necessary, we may assume that $\widetilde{a}$ is a finite set in the same manner as \cite[Lemma 3.2]{Fuji_uniform}.
	
	By equality (\ref{eq:5}), the family $\{ \widetilde{a} \}_{a \in I}$ is infinite because $I \setminus K$ is infinite and the set $\widetilde{a}$ is finite for any $a \in I$.
	
	We define a definable relation $E$ on $I \setminus K$ by
	\begin{equation*}
		E(a,b) \Leftrightarrow \mathcal M \models \exists c\ ( a \in \widetilde{c} \wedge b \in \widetilde{c})\text{.}
	\end{equation*}
	It is an equivalence relation by the equality (\ref{eq:5}).
	Set 
	\begin{equation*}
		Y=\{x \in I \setminus K\;|\;\  \mathcal M \models \forall y \in I \setminus K \ (E(x,y) \rightarrow x \leq y)\}\text{.}
	\end{equation*}
	The smallest element of $\widetilde{a}$ belongs to the set $Y$ for all $a \in I$.
	Therefore, the definable set $Y$ is an infinite set because the family $\{ \widetilde{a} \}_{a \in I}$ is infinite.
	We may assume that $Y$ is a finite union of points and open convex sets, shrinking the interval $I$ if necessary.
	We can lead to a contradiction in the same manner as \cite[Lemma 3.2]{Fuji_uniform} using (vi) through (viii). 
	We omit the details of the proof, but we give a comment.
	
	In the course of this proof, we get the rightmost convex set in $Y$ and it is possible because $Y$ is a union of a finite set  and finitely many open convex sets. 
	We need the assumption that $\mathcal M$ is a uniformly locally weakly o-minimal structure of the second kind at this step only.
	The other part of the proof works under the assumption that the structure is uniformly locally l-visceral of the second kind.
\end{proof}

We next recall the definition of local monotonicity.
\begin{definition}[Local monotonicity]
	Consider an expansion of a dense linear order without endpoints $\mathcal M=(M,<,\ldots)$.
	A function $f$ defined on an open subset $I$ of $M$ is \textit{locally constant} if, for any $x \in I$, there exists an open interval $J$ such that $x \in J \subseteq I$ and the restriction $f|_J$ of $f$ to $J$ is constant.
	
	A function $f$ defined on an open subset $I$ of $M$ is \textit{locally strictly increasing} if, for any $x \in I$, there exists an open interval $J$ such that $x \in J \subseteq I$ and $f$ is strictly increasing on the interval $J$.
	We define a \textit{locally strictly decreasing} function similarly. 
	A \textit{locally strictly monotone} function is a locally strictly increasing function or a locally strictly decreasing function.
	A \textit{locally monotone} function is locally strictly monotone or locally constant.
\end{definition}

We next prove the following lemma in a different way from \cite[Lemma 3.1]{Fuji_uniform}.
The author was inspired by \cite{DG} and the strategy employed there was used in our proof of the lemma.  
\begin{lemma}\label{lem:mono_key_lemma1}
	Consider a locally o-minimal structure $\mathcal M=(M,<,\ldots)$.
	Assume that no definable functions defined on open intervals into $\overline{M}$ have local minimums throughout the intervals.
	Let $f:I \to \overline{M}$ be a definable function from an open subset $I$ of $M$.
	Assume that, for each $a \in I$, there exists an open interval $J$ containing the point $a$ and contained in $I$ such that $f(x)<f(a)$ whenever $x \in J$ and $x<a$ and $f(x)>f(a)$ whenever $x \in J$ and $x>a$.
	Then there exists a discrete and definable subset $Z$ of $I$ such that the restriction of $f$ to $I \setminus Z$ is locally strictly increasing.
\end{lemma}
\begin{proof}
	Consider two formulas defined below:
	\begin{align*}
		\chi_1(x) &= \forall x_1>x \ (\exists y,z \ (x<y<z<x_1) \wedge (f(z) \leq f(y)));\\
		\chi_2(x) &= \forall x_2<x \ (\exists y,z \ (x_2<y<z<x) \wedge (f(z) \leq f(y))).
	\end{align*}
	We first prove that $\chi_2(I):=\{x \in I\;|\; \mathcal M \models \chi_2(x) \}$ has an empty interior.
	Assume for contradiction that a nonempty open interval $J$ is contained in $\chi_2(I)$.
	Set 
	\begin{align*}
		V_a &:=\{x \in J\;|\; x<a \wedge \forall y \in [x,a) \ f(y)<f(a)\} \cup\\
		& \quad \{x \in J\;|\; x>a \wedge \forall y \in (a,x] \ f(y)>f(a)\} \cup \{a\} 
	\end{align*}
	for all $a \in J$.
	Observe that $V_a$ is definable and convex, and we have $a \in \myint(V_a)$ by the assumption.  
	Define a definable function $g:J \to \overline{M}$ by $g(a)=\inf V_a$.
	We show that $g$ has local minimums throughout $J$, which contradicts the assumption.
	
	Fix an arbitrary point $a \in J$.
	Since $\mathcal M$ is locally o-minimal by the assumption, there exists $b \in J$ such that $b<a$ and exactly one of the inequalities $g(x)>g(a)$ and $g(x) \leq g(a)$ holds on the interval $(b,a)$.
	Assume for contradiction that the inequality $g(x) \leq g(a)$ holds for all $b<x<a$.
	We may assume that $b \in V_a$ by choosing a larger $b$ if necessary.
	Since $\mathcal M \models \chi_2(a)$, we can choose $c,d \in V_a$ such that $b<c<d<a$ and $f(d) \leq f(c)$.
	The definition of $g$ implies that $b<c \leq g(d)$.
	We have $g(a) \leq b<g(d)$ because $b \in V_a$.
	This contradicts the assumption that $g(x) \leq g(a)$ holds for all $b<x<a$.
	We have demonstrated that $g(x)>g(a)$ for all $b<x<a$.
	Take $b' \in J$ so that $b'>a$. 
	We can prove $g(x)>g(a)$ for all $a<x<b'$ in the same way as above using the facts $\mathcal M \models \chi_2(b')$ and $a \in V_a$.
	We have demonstrated that $g(a)$ is a local minimum of $g$.
	Since $a$ is arbitrary, $g$ has local minimums throughout $J$, which is a contradiction to the assumption.
	
	We have proven that $\chi_2(I)$ has an empty interior.
	We can prove that $\chi_1(I):=\{x \in I\;|\; \mathcal M \models \chi_1(x) \}$ has an empty interior similarly. 
	The definable set $Z:=\{x \in I\;|\; \mathcal M \models \chi_1(x) \vee \chi_2(x)\}=\chi_1(I) \cup \chi_2(I)$ does not contain an open interval because of local o-minimality.
	Therefore $Z$ is discrete and closed by Lemma \ref{lem:local_dim0}.
	It is obvious that the restriction of $f$ to $I \setminus Z$ is locally increasing by the definition of $Z$.
\end{proof}

\begin{lemma}\label{lem:mono_key_lemma3}
	Let $\mathcal M=(M,<,\ldots)$ be a locally o-minimal structure.
	Strictly monotone definable functions defined on open intervals which are discontinuous everywhere have discrete images.
\end{lemma}
\begin{proof}
	See \cite[Lemma 3.3]{Fuji_uniform}.
\end{proof}

We are now ready to prove the first main theorem of this section.
\begin{theorem}[Parameterized local monotonicity theorem]\label{thm:mono}
	Consider a uniformly locally weakly o-minimal structure of the second kind $\mathcal M=(M,<,\ldots)$.
	Let $A \subseteq M$ and $P \subseteq M^n$ be definable subsets.
	Assume that $A$ is open.
	Let $f:A \times P \rightarrow \overline{M}$ be a definable function.
	For any $(a,b,p) \in M \times M \times M^n$, any sufficiently small open intervals $I$ and $J$ with $a \in I$ and $b \in J$ and any sufficiently small open box $B$ with $p \in B$, the following assertions hold true:
	\begin{enumerate}[(1)]
		\item The set $f^{-1}(\overline{J}) \cap (I \times \{z\})$ is a finite union of points and open convex sets for all $z \in B$.
		\item There exists a mutually disjoint definable partition $\{X_f, X_-, X_+, X_c\}$ of $f^{-1}(\overline{J}) \cap (I \times B)$ satisfying the following conditions for any $z \in B \cap P$:
		\begin{enumerate}[(i)]
			\item the definable set $X_f \cap (f^{-1}(\overline{J}) \cap (I \times \{z\}))$ is a finite set;
			\item the definable set $X_c \cap (f^{-1}(\overline{J}) \cap (I \times \{z\}))$ is a finite union of open convex sets and $f$ is locally constant on the set;
			\item the definable set $X_- \cap (f^{-1}(\overline{J}) \cap (I \times \{z\}))$ is a finite union of open convex sets and $f$ is locally strictly decreasing  on the set;
			\item the definable set $X_+ \cap (f^{-1}(\overline{J}) \cap (I \times \{z\}))$ is a finite union of open convex sets and $f$ is locally strictly increasing on the set.
		\end{enumerate}
	\end{enumerate}
	Furthermore, we can take $J=M$ and $B=M^n$ when $\mathcal M$ is uniformly locally weakly o-minimal of the first kind.
\end{theorem}
\begin{proof}
	Consider the definable set 
	\begin{equation*}
		X=\{(x,y_1,y_2,z) \in M^3 \times P\;|\; y_1 < b, y_2 > b, y_1 < f(x,z) <y_2\}\text{.}
	\end{equation*}
	Since $\mathcal M$ is a uniformly locally weakly o-minimal structure of the second kind, there exist an open interval $I$ with $a \in I$, open intervals $J_1$, $J_2$ containing $b$ and an open box $B$ with $p \in B$ such that the definable set $X_{(y_1,y_2,z)} \cap I$ is a finite union of points and open convex sets for any $y_1 \in J_1$, $y_2 \in J_2$ and $z \in B$.
	Here, the notation $X_ {(y_1,y_2,z)}$ denotes the fiber of $X$ at the point  $(y_1,y_2,z)$.
	Take $b_1 \in J_1$ and $b_2 \in J_2$ with $b_1 <b$ and $b_2>b$, and set $J=(b_1,b_2)$.
	Then, the set $f^{-1}(\overline{J}) \cap (I \times \{z\})$ is a finite union of points and open convex sets for all $z \in B \cap P$. 
	We have shown the existence of $I$, $J$ and $B$ satisfying the condition (1) of the theorem.
	Note that, even if we shrink $I$, $J$ and $B$, the condition (1) remains true.
	
	When $\mathcal M$ is uniformly locally weakly o-minimal of the first kind, we can take $J=M$ and $B=M^n$ in the above construction.
	In fact, we consider $A \times P$ instead of $X$ in this case.
	We shrink $I$, $J$ and $B$ several times in the rest of the proof.
	For the proof of `furthermore' part, we have only to check that we can take $J=M$ and $B=M^n$ every time when we shrink $I$, $J$ and $B$.
	These checks are done though they are not explicitly described in the rest of proof.
	
	We prove the following claim.
	\medskip
	
	\textbf{Claim 1.}
	Shrinking $I$, $J$ and $B$ if necessary, there exists a partition $f^{-1}(\overline{J}) \cap (I \times B) = X_f' \cup X_c \cup X_n$ such that, for any $z \in B \cap P$,
	\begin{enumerate}[(a)]
		\item the definable set $X_f' \cap (f^{-1}(\overline{J}) \cap (I \times \{z\}))$ is a finite set;
		\item the definable set $X_c$ satisfies the condition (ii) of the theorem; 
		\item the definable set $X_n \cap (f^{-1}(\overline{J}) \cap (I \times \{z\}))$ is a finite union of open convex sets and $f$ is locally injective on the open convex sets.
	\end{enumerate}
	Here, a function $g:I \rightarrow \overline{M}$ is called \textit{locally injective} if, for any $x \in I$, there exists an open interval $I'$ such that  $x \in I' \subseteq I$ and the restriction of $g$ to $I'$ is injective.
	\begin{proof}[Proof of Claim 1]
		The notation $f_z$ denotes the function given by $f(\cdot,z)$.
		Set 
		\begin{align*}
			X'_c&=\{(x,z) \in I \times (B \cap P)\;|\; x \in I \cap f_z^{-1}(\overline{J}) \text{ and }\exists x_1,x_2 \in I, x_1 < x < x_2 \text{ and }\\
			& (x' \in I \cap f_z^{-1}(\overline{J})) \wedge (f_z(x)=f_z(x')) \text{ for all } x' \text{ with } x_1 < x' < x_2\}\text{.}
		\end{align*}
		We put $X_c=\{(x,z) \in X'_c\;|\; x \in \myint((X'_c)_z)\}$.
		The set $X_c$ clearly satisfies the condition (b) of the claim by shrinking $I$ and $B$ if necessary.
		
		Set $E=\{(x,z) \in I \times B\;|\; x \in \partial((X_c)_z)\}$ and $Y=(f^{-1}(\overline{J}) \cap (I \times B)) \setminus (E \cup X_c)$.
		Note that the fiber $Y_z$ is a finite union of open convex sets for all $z \in B$.
		We consider the definable set $F=\{(x,x',z) \in I \times I\times P\;|\; f_z(x)=f_z(x') \text{ and } (x,z) \in Y\ \}$.
		Shrinking $I$ and  $B$ if necessary, we may assume that the fiber $F_{(x',z)}$ is a finite union of points and open convex sets for any $x' \in I$ and $z \in B \cap P$.
		If the fiber $F_{(x',z)}$ contains an open convex set, the function $f_z$ is constant on the interval.
		It contradicts the definition of $Y$.
		The fiber $F_{(x',z)}$ is a finite set; that is, the function $f_z$ is finite-to-one on the fiber $Y_z$ for any $z \in B$.
		
		Set $X_n=\{(x,z) \in Y\;|\; \exists x_1 < x, \exists x_2>x, f_z \text{ is injective on the interval }(x_1,x_2)\}$.
		We showed that $(X_n)_z$ is dense in $Y_z$ for any $z \in B$ in Claim 1 of \cite[Theorem 3.2]{Fuji_uniform}.
		We can prove that $(X_n)_z$ is dense in $Y_z$ for any $z \in B$ in the same manner.
		We omit the details.

		We may assume that $(X_n)_z$ and $(Y \setminus X_n)_z$ are unions of finitely many open convex sets and a finite set for any $z \in B$ by shrinking $I$ and $B$ if necessary.
		Note that $(Y \setminus X_n)_z$ is a finite set because $(X_n)_z$ is dense in $Y_z$.
		The definable set $X_n$ satisfies the condition (c) of the claim.
		Set $X_f'=E \cup (Y \setminus X_n)$, then the set $X_f'$ satisfies the condition (a) of the claim.
	\end{proof}

\textbf{Claim 2.}
Shrinking $I$ and $B$ if necessary, there exists a partition $X_n = X_f'' \cup X_+ \cup X_-$ such that, for any $z \in B$, 
\begin{enumerate}[(a)]
	\item the definable set $X_f'' \cap (f^{-1}(\overline{J}) \cap (I \times \{z\}))$ is a finite set;
	\item the definable set $X_- \cap (f^{-1}(\overline{J}) \cap (I \times \{z\}))$ is a finite union of open convex sets and $f$ is locally strictly decreasing on the open convex sets;  
	\item the definable set $X_+ \cap (f^{-1}(\overline{J}) \cap (I \times \{z\}))$ is a finite union of open convex sets and $f$ is locally strictly increasing on the open convex sets.
\end{enumerate}
\begin{proof}[Proof of Claim 2]
	Define the definable subsets $X_-'$, $X_+'$, $X_{\text{max}}$ and  $X_{\text{min}}$ of $X_n$ as follows:
	\begin{align*}
		X_-' &= \{(x,z) \in X_n\;|\;\exists x_1<x\ \exists x_2>x\ (x_1,z) \in X_n, (x_2,z) \in X_n, \\
		& \qquad \forall x' ((x_1<x'<x) \rightarrow (f(x',z)>f(x,z)) \wedge (x',z) \in X_n),\\
		& \qquad \forall x' ((x<x'<x_2) \rightarrow (f(x,z)>f(x',z)) \wedge (x',z) \in X_n)\}\\
		X_+' &= \{(x,z) \in X_n\;|\;\exists x_1<x\ \exists x_2>x\ (x_1,z) \in X_n, (x_2,z) \in X_n, \\
		& \qquad \forall x' ((x_1<x'<x) \rightarrow (f(x',z)<f(x,z)) \wedge (x',z) \in X_n),\\
		& \qquad \forall x' ((x<x'<x_2) \rightarrow (f(x,z)<f(x',z)) \wedge (x',z) \in X_n)\}\\
		X_{\text{max}}&= \{(x,z) \in X_n\;|\;\exists x_1<x\ \exists x_2>x\ (x_1,z) \in X_n, (x_2,z) \in X_n, \\
		& \qquad \forall x' ((x_1<x'<x) \rightarrow (f(x',z)<f(x,z)) \wedge (x',z) \in X_n),\\
		& \qquad \forall x' ((x<x'<x_2) \rightarrow (f(x,z)>f(x',z)) \wedge (x',z) \in X_n)\}
	\end{align*}
	\begin{align*}	
		X_{\text{min}} &= \{(x,z) \in X_n\;|\;\exists x_1<x\ \exists x_2>x\ (x_1,z) \in X_n, (x_2,z) \in X_n, \\
		& \qquad \forall x' ((x_1<x'<x) \rightarrow (f(x',z)>f(x,z)) \wedge (x',z) \in X_n),\\
		& \qquad \forall x' ((x<x'<x_2) \rightarrow (f(x,z)<f(x',z)) \wedge (x',z) \in X_n)\}
	\end{align*}
	Since $\mathcal M$ is locally o-minimal by Proposition \ref{prop:local_ominmal} and $f$ is locally injective on $X_n$, we have 
	\begin{equation*}
		X_n = X_-' \cup X_+' \cup X_{\text{max}} \cup X_{\text{min}}\text{.}
	\end{equation*}
	Set $E'=\{(x,z)\;|\; x \in ((X_-')_z \setminus \myint((X_-')_z)) \cup ((X_+')_z \setminus \myint((X_+')_z))\}$.
	Set $X_+:=\{(x,z) \in X_+'\setminus E'\;|\; f_z \text{ is locally strictly increasing}\}$, which satisfies the condition (c) of Claim 2 by shrinking $I$ and $B$ if necessary.
	Note that $(X'_+ \setminus X_+)_z$ is discrete and closed by Lemma \ref{lem:mono_key_lemma1} and Lemma \ref{lem:mono_key_lemma2}.
	We may assume that $(X'_+ \setminus X_+)_z$ is a finite set by shrinking $I$ and $B$ again.
	We define $X_-$ in the same manner and we may assume that $(X'_- \setminus X_-)_z$ is a finite set for the same reason as above.
	
	Shrinking $I$ and $B$ if necessary, the fiber $(X_{\text{min}})_z$ is a finite set for any $z \in B$ by Lemma \ref{lem:mono_key_lemma2}.
	In the same way, we may assume that $(X_{\text{max}})_z$ is a finite set.
	The definable set $X'_f=E' \cup (X'_+ \setminus X_+) \cup (X'_- \setminus X_-) \cup X_{\text{min}} \cup X_{\text{max}}$ satisfies the condition (a) of Claim 2.
\end{proof}

Claim 1 and Claim 2 obviously imply the theorem.
\end{proof}

We give two corollaries of Theorem \ref{thm:mono} announced in Section \ref{sec:intro}.
They treat the cases in which the structure is a uniformly locally weakly o-minimal structure of the first and second kind, respectively.
The following corollary is not proved in \cite{Fuji_uniform} even when the structure is uniformly locally o-minimal of the first kind.

\begin{corollary}\label{cor:mono1}
Consider a uniformly locally weakly o-minimal structure of the first kind $\mathcal M=(M,<,\ldots)$.
Let $A \subseteq M$ and $P \subseteq M^n$ be definable subsets.
Let $f:A \times P \rightarrow \overline{M}$ be a definable function.
There exists a decomposition $A \times P = X_+ \cup X_- \cup X_c \cup X_d$ satisfying the following for each $z \in P$:
\begin{enumerate}[(i)]
	\item the definable set $(X_d)_z$ is discrete and closed;
	\item the definable set $(X_c)_z$ is open and $f$ is locally constant on the set;
	\item the definable set $(X_-)_z$ is open and $f$ is locally strictly decreasing  on the set;
	\item the definable set $(X_+)_z$ is open and $f$ is locally strictly increasing on the set.
\end{enumerate}
\end{corollary}
\begin{proof}
	Set $A'=\myint(A)$.
	Define the definable sets $X_c$, $X_+$, $X_-$ as follows:
	\begin{align*}
		X_c &= \{(x,z) \in A' \times P\;|\; f_z \text{ is constant on an open interval } I \\
		& \qquad \text { containing the point }x\}\\
		X_+ &= \{(x,z) \in A' \times P\;|\; f_z \text{ is strictly increasing on an open interval } I \\
		& \qquad \text { containing the point }x\}\\
		X_- &= \{(x,z) \in A' \times P\;|\; f_z \text{ is strictly decreasing on an open interval } I \\
		& \qquad \text { containing the point }x\}
	\end{align*}
	It is obvious that $X_c$, $X_+$ and $X_-$ satisfy conditions (ii) through (iv).
	We set $X_d=(A \times P) \setminus (X_c \cup X_+ \cup X_-)$.
	Fix a point $x \in A'$.
	We can take an open interval $I$ containing the point $x$ such that $(X_d \cap A')_z \cap I$ is a finite set for each $z \in P$ by the `furthermore' part of Theorem \ref{thm:mono}.
	It means that $(X_d \cap A')_z$ has an empty interior. 
	It is obvious that $A \setminus A'$ has an empty interior.
	The definable set $(X_d)_z =  (X_d \cap A')_z \cup (A \setminus A')$ has an empty interior.
	The definable set $(X_d)_z$ is discrete and closed by Lemma \ref{lem:local_dim0} because $\mathcal M$ is locally o-minimal by Proposition \ref{prop:local_ominmal}. 
\end{proof}

We may assume continuity in Theorem \ref{thm:mono} as follows:
\begin{corollary}\label{cor:mono2}
	Consider a uniformly locally weakly o-minimal structure of the second kind $\mathcal M=(M,<,\ldots)$.
	Let $A \subseteq M$ and $P \subseteq M^n$ be definable subsets.
	Assume that $A$ is open.
	Let $f:A \times P \rightarrow M$ be a definable function.
	For any $(a,b,p) \in M \times M \times M^n$, any sufficiently small open intervals $I$ and $J$ with $a \in I$ and $b \in J$ and any sufficiently small open box $B$ with $p \in B$, the following assertions hold true:
	\begin{enumerate}[(1)]
		\item The set $f^{-1}(J) \cap (I \times \{z\})$ is a finite union of points and open convex sets for all $z \in B$.
		\item There exists a mutually disjoint definable partition $\{X_f, X_-, X_+, X_c\}$ of $f^{-1}(J) \cap (I \times B)$ satisfying the following conditions for any $z \in B \cap P$:
		\begin{enumerate}[(i)]
			\item the definable set $X_f \cap (f^{-1}(J) \cap (I \times \{z\}))$ is a finite set;
			\item the definable set $X_c \cap (f^{-1}(J) \cap (I \times \{z\}))$ is a finite union of open convex sets and $f$ is locally constant on the set;
			\item the definable set $X_- \cap (f^{-1}(J) \cap (I \times \{z\}))$ is a finite union of open convex sets and $f$ is locally strictly decreasing and continuous on the set;
			\item the definable set $X_+ \cap (f^{-1}(J) \cap (I \times \{z\}))$ is a finite union of open convex sets and $f$ is locally strictly increasing and continuous on the set.
		\end{enumerate}
	\end{enumerate}
\end{corollary}
\begin{proof}
	By Theorem \ref{thm:mono}, we can decompose $f^{-1}(J) \cap (I \times B)=X_f' \cup X_c \cup X_-' \cup X_+'$ into definable sets so that $X_f'$, $X_c$, $X_+'$ and $X_-'$ satisfy conditions (i) through (iv) except continuity, respectively.
	
	Consider the definable sets 
	\begin{align*}
		Z_+&=\{(x,z) \in X_+'\;|\;f_z \text{ is not continuous at }x \}\text{ and }\\
		Z_- &=\{(x,z) \in X_-'\;|\;f_z \text{ is not continuous at }x \}\text{.}
	\end{align*}
	Set $g_+=f|_{Z_+}$.
	The notation $(g_+)_z$ denotes the restriction of $g_+$ to $(Z_+)_z$.
	\medskip
	
	\textbf{Claim 1.}
	The image $(g_+)_z(I')$ of an open subinterval $I'$ of $(Z_+)_z$ is a finite union of points and open convex sets for any $z \in B$ by shrinking $I$, $J$ and $B$ if necessary.
	\begin{proof}[Proof of Claim 1]
	Consider the set
	\begin{equation*}
		V=\{(y,z,a_1,a_2)\;|\;y=(g_+)_z(x) \text{ for some } a_1 < x < a_2\}\text{.}
	\end{equation*}
	The definable set $V_{(z,a_1,a_2)} \cap J$ is a finite union of points and open convex sets for any $z \in B$ and $a_1,a_2 \in I$ by shrinking $I$, $J$ and $B$ if necessary.
	It means that the image $(g_+)_z(I')$ of an open subinterval $I'$ of $(Z_+)_z$ is a finite union of points and convex sets.
\end{proof}
	
	We next demonstrate that $(Z_+)_z$ is a finite set for any $z \in B \cap P$.
	Shrinking $I$, $J$ and $B$ if necessary, we may assume that $(Z_+)_z$ is a finite union of points and convex sets for any $z \in B$.
	Assume that the definable set $(Z_+)_z$ contains an open interval $I'$ for some $z \in B$.
	We may assume that $(g_+)_z$ is strictly increasing on $I'$ by shrinking $I'$ if necessary because $(g_+)_z$ is locally strictly increasing.
	The image $(g_+)_z(I')$ is a finite union of points and open convex sets by Claim 1.
	The image $(g_+)_z(I')$ is a discrete set by Lemma \ref{lem:mono_key_lemma3} because $(g_+)_z$ is strictly increasing on $I'$ and discontinuous everywhere on $I'$.
	Hence, the image $(g_+)_z(I')$ is a finite set.
	On the other hand, the image $(g_+)_z(I')$ is simultaneously an infinite set because $(g_+)_z$ is injective and $I'$ is an infinite set.
	Contradiction.
	We have shown that $(Z_+)_z$ is a finite set.
	We can prove that $(Z_-)_z$ is a finite set similarly.

	Set $X_+=X_+' \setminus Z_+$, $X_-=X_-' \setminus Z_-$ and $X_f=X_f' \cup Z_+ \cup Z_-$.
	Then, the definable sets $X_c$, $X_+$, $X_-$ and $X_f$ satisfy the conditions of the corollary.
\end{proof}

We prove the following two assertions simultaneously.
The strategy of the proof is very similar to that of \cite[Theorem 3.3, Theorem 3.4]{Fuji_uniform}.

\begin{theorem}\label{thm:simple_decomp}
	Consider a uniformly locally weakly o-minimal structure of the second kind $\mathcal M=(M,<,\ldots)$.
	Let $X$ be a definable subset of $M^n$ with a nonempty interior and $X=X_1 \cup X_2$ be a partition of $X$ into two definable subsets.
	Then, at least one of $X_1$ and $X_2$ has a nonempty interior.
\end{theorem}

\begin{lemma}\label{lem:func}
	Consider a uniformly locally weakly o-minimal structure of the second kind $\mathcal M=(M,<,\ldots)$.
	Let $U$ be an open box in $M^n$ and $f:U \to \overline{M}$ be a definable function.
	Assume that there exists $a \in M$ such that $f(x)>a$ for all $x \in U$.
	Then, there exists $b \in M$ such that $b>a$ and the definable set $\{x \in U\;|\; f(x)>b\}$ has a nonempty interior.
\end{lemma}
\begin{proof}[Proof of Theorem \ref{thm:simple_decomp} and Lemma \ref{lem:func}]
	When $n=1$, Theorem \ref{thm:simple_decomp} obviously holds because $\mathcal M$ is locally o-minimal by Proposition \ref{prop:local_ominmal}.
	
	We next prove Theorem \ref{thm:simple_decomp} for $n$ under the assumption that Theorem \ref{thm:simple_decomp} and Lemma \ref{lem:func} hold for $n-1$.
	We first reduce to the case in which $X$ is a bounded open box in Theorem \ref{thm:simple_decomp}.
	The definable set $X$ contains a bounded open box $B$. 
	If the theorem holds for the partition $B=(B \cap X_1) \cup (B \cap X_2)$, the theorem holds for the partition $X=X_1 \cup X_2$.
	We may assume that $X$ is a bounded open box without loss of generality.
	We take open intervals $I_1, \ldots, I_n$ so that $X=\prod_{i=1}^n I_i$.
	
	Let $\pi:M^n \to M^{n-1}$ be the projections forgetting the last coordinate.
	Take a point $a \in I_n$.
	Since $\mathcal M$ is locally o-minimal, for each $x \in \pi(X)$, there exists $b_x>a$ such that the interval $(a,b_x)$ is contained in the fibers $(X_1)_x$ or $(X_2)_x$. 
	Set $Y_i=\{x \in \pi(X)\;|\; (X_i)_x \text{ contains the interval }(a,y) \text{ for some }y \in M \text{ with }y>a\}$ for $i=1,2$.
	We have $\pi(X)=Y_1 \cup Y_2$.
	By the induction hypothesis, at least one of $Y_1$ and $Y_2$ has a nonempty interior.
	We may assume that $\myint(Y_1) \neq \emptyset$ without loss of generality.
	Shrinking $\pi(X)$ if necessary, we may assume that $\pi(X)=Y_1$.
	
	Consider the definable map $h:\pi(X) \to \overline{M}$ defined by $h(x)=\sup\{y \in I_n\;|\; y>a \text{ and }\{x\} \times (a,y) \subseteq X_1\}$.
	We have $h(x)>a$ for all $x \in \pi(X)$ because $\pi(X)=Y_1$.
	We can take $b>a$ such that $\{x \in \pi(X)\;|\; h(x)>b\}$ has a nonempty interior by the assumption.
	We may assume that $h(x)>b$ for all $x \in \pi(X)$ by shrinking $\pi(X)$ if necessary.
	We get $\pi(X) \times (a,b) \subseteq X_1$ by the definition of $h$.
	This means that $X_1$ has a nonempty interior.
	We have proven Theorem \ref{thm:simple_decomp} for $n$.
	
	We next prove Lemma \ref{lem:func} for $n$ assuming Theorem \ref{thm:simple_decomp} for $n$ and Lemma \ref{lem:func} for $n-1$.
	We have only to prove that, by shrinking $U$ if necessary, there exists $b >a$ such that $f(x)>b$ for all $x \in U$.
	We take open intervals $I_1, \ldots, I_n$ so that $U=\prod_{i=1}^n I_i$.
	Apply Theorem \ref{thm:mono} to the definable map $f:U=I_1 \times I_2 \times \cdots \times I_n \to \overline{M}$ by setting $A=I_1$ and $P=I_2 \times \cdots \times I_n$ under the terminology in Theorem \ref{thm:mono}.
	We can take open intervals $I, J$ with $a \in J$, an open box $B$ and a mutually disjoint definable partition $\{X_f,X_-,X_+,X_c\}$ of $f^{-1}(\overline{J}) \cap (I \times B)$ satisfying the conditions (1) and (2) in Theorem \ref{thm:mono}. 
	Shrinking $I_i$ for $1 \leq i \leq n$, we may assume that $I \times B=U$.
	We have $U= X_f \cup X_- \cup X_+ \cup X_c \cup (U \setminus f^{-1}(\overline{J}))$.
	The definable set $X_f$ has an empty interior by (2)(i) in Theorem \ref{thm:mono}.
	By the induction hypothesis, at least one of $X_-$, $X_+$, $ X_c$ and $(U \setminus f^{-1}(\overline{J}))$ has a nonempty interior.
	We may assume that $U$ coincides with one of the above four definable sets by shrinking $U$ if necessary.
	
	We first consider the case in which $U=(U\setminus f^{-1}(\overline{J}))$; that is, $U \cap f^{-1}(\overline{J})=\emptyset$.
	Let $b$ be an element in $J$ with $a<b$.
	Since $f(x)>a$ for all $x \in U$ and $U \cap f^{-1}(\overline{J})=\emptyset$, we have $f(x)>b$ for all $x \in U$.
	The lemma holds in this case.
	
	We next consider the case in which $U=X_+$.
	Set $K=I_2 \times \cdots \times I_n$.
	Take a point $c \in I_1$.
	Consider the definable function $g:K \to \overline{M}$ given by $g(x)=\sup\{y \in I_1\;|\; f(\cdot, x) \text{ is strictly increasing on } [c,y]\}$, where $f(\cdot,x)$ is the definable function given by $t \mapsto f(t,x)$.
	Since $U=X_+$, we have $g(x)>c$ for each $x \in K$.
	By the induction hypothesis, we can take $c'>c$ such that the definable set $\{x \in K\;|\;g(x)>c'\}$ has a nonempty interior.
	We may assume that $g(x)>c'$ for all $x \in K$ by shrinking $K$ if necessary.
	The restriction of $f(\cdot,x)$ to $[c,c']$ is strictly increasing for each $x \in K$ by the definition of $g$.
	We may assume that $f(\cdot,x)$ is strictly increasing for each $x \in K$ setting $I_1=(c,c')$.
	Applying the same argument to the cases in which $U=X_-$ and $U=X_c$, we can reduce to the case in which only one of the following three condition holds.
	\begin{enumerate}
		\item[$(1)_1$]  $f(\cdot,x)$ is strictly increasing for each $x \in K$;
		\item[$(2)_1$]  $f(\cdot,x)$ is strictly decreasing for each $x \in K$;
		\item[$(3)_1$]  $f(\cdot,x)$ is constant for each $x \in K$.
	\end{enumerate}
	We may assume that only one of the following three conditions holds for $1 \leq k \leq n$ by induction on $k$ in the same manner as above.
	\begin{enumerate}
		\item[$(1)_k$]  $f(x,\cdot, y)$ is strictly increasing for each $x \in I_1 \times \cdots \times I_{k-1}$ and $y \in I_{k+1} \times \cdots \times I_n$;
		\item[$(2)_k$]  $f(x,\cdot, y)$ is strictly decreasing for each $x \in I_1 \times \cdots \times I_{k-1}$ and $y \in I_{k+1} \times \cdots \times I_n$;
		\item[$(3)_k$]  $f(x,\cdot, y)$ is constant for each $x \in I_1 \times \cdots \times I_{k-1}$ and $y \in I_{k+1} \times \cdots \times I_n$.
	\end{enumerate}
	Here, $f(x, \cdot,y)$ denotes the definable function given by $t \mapsto f(x,t,y)$.
	Take $d_k,e_k \in I_k$ so that $d_k<e_k$ for each $1 \leq k \leq n$.
	Set $U=\prod_{i=1}^n (d_i,e_i)$.
	Set $p_k=d_k$ if the condition $(1)_k$ holds and set $p_k=e_k$ otherwise for every $1 \leq k \leq n$.
	Put $p=(p_1,\ldots, p_n)$.
	We have $f(x) \geq f(p)$ for every $x \in U$.
	We can take $b$ so that $a<b<f(p)$ because $a<f(p)$.
	We have $f(x)>b$ for all $x \in U$.
	We have proven Lemma \ref{lem:func} for $n$.
\end{proof}

\subsection{Locally o-minimal structure enjoying $*$-continuity property}\label{sec:star-cont}
We define univariate $*$-continuity property and $*$-continuity property which are discussed in Section \ref{sec:*-local}. 

\begin{definition}
	Consider an expansion of a dense linear order without endpoints $\mathcal M=(M,<,\ldots)$.
	We say that $\mathcal M$ enjoys the \textit{univariate $*$-continuity property} if, for every definable function $f:I \to \overline{M}$ from a nonempty open interval $I$, there exists a nonempty open subinterval $J$ of $I$ such that the restriction of $f$ to $J$ is continuous.
	We say that $\mathcal M$ enjoys the \textit{$*$-continuity property} if, for every definable function $f:B \to \overline{M}$ from a nonempty open box $B$, there exists a nonempty open box $U$ contained in $B$ such that the restriction of $f$ to $B$ is continuous. 
\end{definition}


\begin{proposition}\label{prop:equiv_*_cont}
	The $*$-continuity property is preserved under elementary equivalence, and the ultraproduct of structures possessing the $*$-continuity property enjoys the $*$-continuity property.
\end{proposition}
\begin{proof}
	Let $f:X \to \overline{M}$ be a definable function defined on a definable subset $X$ of $M^n$.
	We define the \textit{semi-graph} of $f$ as $\Gamma'(f)=\{(x,y) \in X \times M\;|\; y \leq f(x)\}$.
	A definable subset $Y$ of $M^{n+1}$ is the semi-graph of a definable function if and only if $\forall x \in \pi(Y)\ (  \forall y,y' \in M \ (y \in Y_x \wedge  y'<y) \rightarrow y' \in Y_x) \wedge (\exists \sup Y_x \in M \rightarrow \sup Y_x \in Y_x)$, where $\pi:M^{n+1} \to M^n$ is the projection forgetting the last coordinate and $Y_x:=\{y \in M\;|\; (x,y) \in Y\}$ as usual.
	Consider a definable function $g:B\to \overline{M}$ defined on an open box $B$ in $M^n$ whose semi-graph is $Y$.
	The definable function $g$ is continuous at $x \in B$ if and only if for every open interval $I$ with $Y_x \cap I \neq \emptyset$ and $(M \setminus Y_x) \cap I \neq \emptyset$, there exists an open box $B'$ with $x \in B' \subseteq B$ such that $Y_{x'} \cap I \neq \emptyset$  and $(M \setminus Y_{x'}) \cap I \neq \emptyset$ for every $x' \in B'$. 
	
	Let $\phi(\overline{x},y,\overline{z})$ be a formula, where $\overline{x}$ is a tuple of variables of length $n$ and $\overline{z}$ is a tuple of variables assigned to parameters.
	The $y$ is a single variable.
	Thanks to the observation in the previous paragraph, we can construct a formula $\Phi_{\phi}(\overline{z})$ expressing that the definable set defined by the formula $\phi(\overline{x},y,\overline{z})$ is the semi-graph of a definable function defined on a nonempty open box in $M^n$.
	We can also construct a formula $\Psi_{\phi}(\overline{z})$ saying that the definable function whose semi-graph is the definable set defined by the formula $\phi(\overline{x},y,\overline{z})$ is continuous on some nonempty open box.
	The structure $\mathcal M$ possesses the $*$-continuity property if and only if, for any formula $\phi$, the sentence $\forall \overline{z}\ \Phi_{\phi}(\overline{z}) \rightarrow \Psi_{\phi}(\overline{z})$ holds.
	
	The proposition is now obvious from the above observation.
	\end{proof}

	\begin{remark}\label{rem:equiv_*_cont}
		An assertion similar to Proposition \ref{prop:equiv_*_cont} holds for the univariate $*$-continuity property.
		Its proof is almost the same as that of Proposition \ref{prop:equiv_*_cont}.
		We omit the proof.
	\end{remark}

We next investigate when local monotonicity theorems hold in locally o-minimal structures.
We first introduce a technical definition.
\begin{definition}
	Consider an expansion of a dense linear order without endpoints $\mathcal M=(M,<,\ldots)$.
	A definable set $X$ of $M^2$ is called \textit{dimensionally wild} if the following conditions are satisfied:
	\begin{enumerate}
		\item[(i)] $X$ has an empty interior;
		\item[(ii)] $\pi(X)$ has a nonempty interior, where $\pi$ denotes the projection onto the first coordinate;
		\item[(iii)] $X_x:=\{y \in M\;|\;(x,y) \in X\}$ has a nonempty interior for every $x \in \pi(X)$.
	\end{enumerate}
\end{definition}
We choose the term `dimensionally wild' because a dimensionally wild definable set is a counterexample of the dimension formula called the addition property given in Theorem \ref{thm:dimension_basic}(2).  

We demonstrate two lemmas on the existence of dimensionally wild definable sets.

\begin{lemma}\label{lem:exceptional_case2}
Consider an expansion of a dense linear order without endpoints $\mathcal M=(M,<,\ldots)$.
Assume that there exists a definable function $f:I \to \overline{M}$ defined on an interval $I$ having local minimums throughout the intervals.
Then $\mathcal M$ has a definable function $g:I \to \overline{M}$ which is discontinuous everywhere and a dimensionally wide definable set $X$.
\end{lemma}
\begin{proof}
	We assume that $I$ is bounded by shrinking $I$ if necessary.
	We define the definable function $g:I \to \overline{M}$ by $g(x)=\inf\{r \in I\;|\; \forall t \in M \ (r<t<x) \rightarrow (f(t) > f(x))\}$.
	We show that, for each point $a \in I$, the inequalities $g(a)<a \leq g(b)$ hold for every $b>a$ sufficiently close to $a$.
	We have $g(a)<a$ and $f(b)>f(a)$ for any  point $b \neq a$ sufficiently close to the point $a$ because $f$ has a local minimum at the point $a$.
	We have $g(b) \geq a$ because $f(b)>f(a)$ for any point $b \in I$ sufficiently close to $a$ with $b>a$.
	The function $g$ is discontinuous at an arbitrary point $a$ because $g(a)<a \leq g(b)$ for any $b>a$ sufficiently close to $a$.
	
	We define $X \subseteq I \times M$ by $X:=\{(a,b) \in I \times X\;|\; g(a)<b<a\}$.
	We prove that $X$ is dimensionally wild.
	Let $\pi:M^2 \to M$ be the coordinate projection onto the first coordinate.
	The image $\pi(X)=I$ has a nonempty interior.
	For each $x \in X$, the fiber $X_x$ coincide with the open set $\{y \in M\;|\;g(x)< y <x\}$ and $X_x$ has a nonempty interior.
	The remaining task is to show that $X$ has an empty interior.
	Assume for contradiction that $X$ contains a nonempty open box $B$ in $M^2$.
	Take $(a,b) \in B$.
	Since $B$ is an open box contained in $X$, the point $(a',b) $ belongs to $X$ for any $a'>a$ sufficiently close to $a$.
	We have $g(a)<b<a$ because $(a,b) \in X$.
	We get $g(a')<b<a'$ because $(a',b) \in X$.
	However, we have $a \leq g(a')$ as we demonstrated.
	We have obtained a contradiction because $b<a \leq g(a') < b$.
\end{proof}

\begin{lemma}\label{lem:exceptional_case}
	Consider an expansion of dense linear order without endpoints $\mathcal M=(M,<,\ldots)$.
	Assume that there exists a definable strictly monotone function $f:I \to \overline{M}$ from an open interval $I$ which is discontinuous everywhere.
	Then $\mathcal M$ has a dimensionally wild definable set $X$.
\end{lemma}
\begin{proof}
	Assume that $f$ is strictly increasing.
	We can prove the lemma similarly in the other case.
	Fox any $x \in I$, set $$Y\langle x \rangle :=\{y \in M\;|\; \forall u \in I \ (u<x \rightarrow f(u)<y) \wedge (u>x \rightarrow f(u)>y)\}.$$
	We can easily show the following:
	\begin{enumerate}[(a)]
		\item $Y\langle x \rangle$ is convex and has a nonempty interior for each $x \in I$;
		\item $Y\langle x_1 \rangle < Y\langle x_2 \rangle$ for distinct $x_1,x_2 \in I$ with $x_1<x_2$.
	\end{enumerate}
	Here, the inequality $Y\langle x_1 \rangle < Y\langle x_2 \rangle$ means that $y_1<y_2$ for every $y_1 \in Y\langle x_1 \rangle$ and $y_2 \in Y\langle x_2 \rangle$.
	Both properties follow from the assumption that $f$ is strictly increasing and discontinuous everywhere.
	The set $X:=\bigcup_{x \in I} \{x\} \times Y\langle x \rangle$ is dimensionally wild.
\end{proof}

The following dichotomy theorem is the second main theorem in this section.
\begin{theorem}\label{thm:monotone-star}
Consider a locally o-minimal structure $\mathcal M=(M,M,\ldots)$.
At least one of the following two assertions holds:
\begin{enumerate}
\item[(1)] The structure $\mathcal M$ does not possess the univariate $*$-continuity property and has a dimensionally wide definable set.
\item[(2)] Let $f:I \to \overline{M}$ be an arbitrary  definable function defined on an  arbitrary open interval $I$.
The interval $I$ is decomposed into four definable sets $X_+,X_-,X_c,X_d$ satisfying the following conditions:
\begin{enumerate}
\item[(i)] $X_d$ is discrete and closed.
\item[(ii)] $X_c$ is open and the restriction of $f$ to $X_c$ is locally constant;
\item[(iii)] $X_-$ is open and the restriction of $f$ to $X_-$ is locally strictly decreasing;
\item[(iv)] $X_+$ is open and the restriction of $f$ to $X_+$ is locally strictly increasing.
\end{enumerate}
\end{enumerate}
\end{theorem}
\begin{proof}
	We consider the following two separate cases:
	\begin{enumerate}
		\item[(1)] There exists a definable function $f:I \to \overline{M}$ defined on an open interval $I$ such that $f$ has local minimums throughout the interval or $f$ has local maximums throughout the interval.
		\item[(2)] $\mathcal M$ does not have such definable functions.
	\end{enumerate}
	We first consider case (1).
	We assume that $f:I \to \overline{M}$ has local minimums throughout the interval.
	We can similarly treat the case $f:I \to \overline{M}$ has local maximums throughout the interval.
	The structure $\mathcal M$ does not possess the univariate $*$-continuity property and has a dimensionally wide definable set by Lemma \ref{lem:exceptional_case2}.
	
	We next consider case (2).
	Consider the following formulas:
	\begin{align*}
		\phi_0(x) &= \exists x_1\ (x_1>x) \wedge(\forall t\ (x<t<x_1) \to (f(x)<f(t)));\\
		\phi_1(x) &= \exists x_1\ (x_1>x) \wedge(\forall t\ (x<t<x_1) \to (f(x)=f(t)));\\
		\phi_2(x) &= \exists x_1\ (x_1>x) \wedge(\forall t\ (x<t<x_1) \to (f(x)>f(t)));\\
		\psi_0(x) &= \exists x_0\ (x_0<x) \wedge(\forall t\ (x_0<t<x) \to (f(x)<f(t)));\\
		\psi_1(x) &= \exists x_0\ (x_0<x) \wedge(\forall t\ (x_0<t<x) \to (f(x)=f(t)));\\
		\psi_2(x) &= \exists x_0\ (x_0<x) \wedge(\forall t\ (x_0<t<x) \to (f(x)>f(t)));\\
		\theta_{ij}(x) &= \phi_i(x) \wedge \psi_j(x) \ (0 \leq i,j \leq 2).
	\end{align*}
	By local o-minimality, at least one of $\theta_{ij}(x)$ holds for every $x \in I$.
	Put $X_{ij}=\{x \in I\;|\; \mathcal M \models \theta_{ij}(x)\}$ for $0 \leq i,j \leq 2$.
	We can easily show that $X_{01}, X_{21}, X_{10}, X_{12}$ have empty interiors.
	We omit the proof.
	They are discrete and closed by Lemma \ref{lem:local_dim0}.
	The definable set $X_c:=X_{11}$ is open and the restriction of $f$ to it is locally constant.
	By the assumption, $X_{00}$ and $X_{22}$ have empty interiors.
	There exists a definable discrete closed set $D_{02}$ such that $X_{+}:=X_{02} \setminus D_{02}$ is open and the restriction of $f$ to $X_+$ is locally strictly increasing by Lemma \ref{lem:mono_key_lemma1} and the assumption.
	We can construct a definable discrete closed set $D_{20}$ such that $X_{-}:=X_{20} \setminus D_{20}$ is open and the restriction of $f$ to $X_-$ is locally strictly decreasing in the same manner.
	Set $X_d=X_{01} \cup X_{12} \cup X_{01} \cup X_{21} \cup D_{02} \cup D_{20}$.
	It is obvious that $X_d$ is discrete and closed.
\end{proof}

As a corollary of Theorem \ref{thm:monotone-star}, we get a local monotonicity theorem for locally o-minimal structures enjoying the univariate $*$-continuity property.
This corollary was announced in Section \ref{sec:intro}.
\begin{corollary}\label{cor:monotone-star}
	Consider a locally o-minimal structure $\mathcal M=(M,M,\ldots)$ enjoying the univariate $*$-continuity property.
	Let $f:I \to \overline{M}$ be a definable map defined on an open interval $I$.
	The interval $I$ is decomposed into four definable sets $X_+,X_-,X_c,X_d$ satisfying the following conditions:
	\begin{enumerate}
		\item[(i)] $X_d$ is discrete and closed.
		\item[(ii)] $X_c$ is open and the restriction of $f$ to $X_c$ is locally constant;
		\item[(iii)] $X_-$ is open and the restriction of $f$ to $X_-$ is locally strictly decreasing and continuous;
		\item[(iv)] $X_+$ is open and the restriction of $f$ to $X_+$ is locally strictly increasing and continuous.
	\end{enumerate}
\end{corollary}
\begin{proof}
	The corollary is the same as Theorem \ref{thm:monotone-star}(2) except we require continuity in the items (iii) and (iv).
	The corollary immediately follow from Theorem \ref{thm:monotone-star}, the univariate $*$-continuity property and Proposition \ref{lem:local_dim0}.
\end{proof}

\begin{proposition}\label{prop:continuity2}
	Consider a locally o-minimal structure $\mathcal M=(M,<,\ldots)$.
	The structure $\mathcal M$ enjoys the $*$-continuity property if it possesses the univariate $*$-continuity property 
\end{proposition}
\begin{proof}
	We can prove the proposition in the same manner as \cite[Theorem 2.11(iii)]{Fuji_tame} using Corollary \ref{cor:monotone-star}.
	We omit the proof.
\end{proof}

\section{Dimension theory}\label{sec:dimension}
\subsection{Definition of topological dimension}\label{sec:topological}
We develop dimension theories for two sorts of structures; that is, locally l-visceral structures and $*$-locally weakly o-minimal structures in some restricted cases.
We first recall the definition of dimension.
\begin{definition}[Dimension]\label{def:dim}
	Consider an expansion of a densely linearly order without endpoints $\mathcal M=(M,<,\ldots)$.
	Let $X$ be a nonempty definable subset of $M^n$.
	The dimension of $X$ is the maximal nonnegative integer $d$ such that $\pi(X)$ has a nonempty interior for some coordinate projection $\pi:M^n \rightarrow M^d$.
	We consider that $M^0$ is a singleton equipped with the trivial topology.
	We set $\dim(X)=-\infty$ when $X$ is an empty set.
\end{definition}
We can find several other definitions of dimension in \cite{vdD,Fuji_uniform,Fuji_pregeo, PS_omin}.
See \cite{Mathews, KTTT, Fuji_pregeo} for the equivalence of these definitions in o-minimal structures and definably complete locally o-minimal structures.

The following proposition holds in our definition of dimension without extra assumptions on structures:

\begin{proposition}\label{prop:dim_basic}
	Consider an expansion of a densely linearly order without endpoints $\mathcal M=(M,<,\ldots)$.
	The following assertions hold true.
	\begin{enumerate}
			\item[(1)] Let $X \subseteq Y$ be definable sets.
			Then, the inequality $\dim(X) \leq \dim(Y)$ holds true.
			\item[(2)] Let $\sigma$ be a permutation of the set $\{1,\ldots,n\}$.
			The definable map $\overline{\sigma}:M^n \rightarrow M^n$ is defined by $\overline{\sigma}(x_1, \ldots, x_n) = (x_{\sigma(1)},\ldots, x_{\sigma(n)})$.
			Then, we have $\dim(X)=\dim(\overline{\sigma}(X))$ for any definable subset $X$ of $M^n$.
			\item[(3)] Let $X$ and $Y$ be definable sets.
			We have $\dim(X \times Y) = \dim(X)+\dim(Y)$.
		\end{enumerate}
\end{proposition}
\begin{proof}
	Assertions (1) and (2) are obvious.
	A proof for (3) is found in \cite[Theorem 3.8(3)]{Fuji_tame}.
\end{proof}

\subsection{Local l-visceral case}\label{sec:preliminary}

We consider locally l-visceral structures in this subsection.
We first consider topological properties (A) and (B) below.
In this subsection, we prove several simple formulas for the dimension function of sets definable in locally l-visceral structures enjoying properties (A) and (B).

\begin{definition}\label{def:tame_top2}
	Consider an expansion of a dense linear order without endpoints $\mathcal M=(M,<,\ldots)$.
	We consider the following properties on $\mathcal M$.
	\begin{enumerate}
		\item[(A)] The image of a nonempty definable discrete set under a coordinate projection	is again discrete.
		\item[(B)] Let $X$ be a definable subset of $M^{n+1}$ and $\pi:M^{n+1} \to M^n$ be the coordinate projection forgetting the last coordinate.
		Assume that $\pi(X)$ has a nonempty interior and the fiber $X_x:=\{y \in M\;|\; (x,y) \in X\}$ has a nonempty interior for each $x \in \pi(X)$.
		Then $X$ has a nonempty interior.
	\end{enumerate}
\end{definition} 
We introduced similar concepts called properties (a) through (d) in \cite{Fuji_tame}.
They are also discussed in this paper.
See Definition \ref{def:tame_top_old}.
Property (A) in this paper is identical to property (a) in \cite{Fuji_tame}.
Property (B) is proven in \cite[Lemma 3.3]{Fuji_tame} using definably completeness, properties (b) and (c).
The author does not know whether property (B) follows only from properties (b) and (c) without assuming definable completeness even when the structure is locally o-minimal.
Therefore, we employ property (B) as an assumption.

We first prove several basic topological properties of locally l-visceral structures enjoying the above properties.
\begin{lemma}\label{lem:key0}
	Consider a locally l-visceral structure with property (A) in Definition \ref{def:tame_top2}.
	A definable discrete set is closed.
\end{lemma}
\begin{proof}
	It was proved in \cite[Lemma 2.4]{Fuji_tame} when the structure is definably complete and locally o-minimal.
	We can prove the lemma similarly to \cite[Lemma 2.4]{Fuji_tame} using Proposition \ref{prop:char_lvisceral} in place of \cite[Lemma 2.3]{Fuji_tame}.
\end{proof}

\begin{lemma}\label{lem:aaa}
	Consider a locally l-visceral structure $\mathcal M=(M,<,\ldots)$ with property (A) in Definition \ref{def:tame_top2}.
	Let $f:X \rightarrow M$ be a definable map.
	If the image $f(X)$ and all fibers of $f$ are discrete, then so is $X$. 
\end{lemma}
\begin{proof}
	It was proved in \cite[Lemma 2.5]{Fuji_tame} when the structure is definably complete and locally o-minimal.
	We can prove the lemma in the same manner using Lemma \ref{lem:key0}.
\end{proof}

\begin{proposition}\label{prop:zero}
	Consider a locally  l-visceral structure satisfying property (A) in Definition \ref{def:tame_top2}.
	A definable set is of dimension zero if and only if it is discrete.
	When it is of dimension zero, it is also closed.
\end{proposition}
\begin{proof}
	We can prove it in the same manner as \cite[Proposition 3.2]{Fuji_tame} using Lemma \ref{lem:key0}.
\end{proof}

\begin{lemma}\label{lem:l-b}
	Consider a locally l-visceral structure $\mathcal M=(M,<,\ldots)$ enjoying property (B) in Definition \ref{def:tame_top2}.
	Let $X$ be a definable subset of $M^n$ having a nonempty interior.
	Let $X_1$ and $X_2$ be definable subsets of $X$ with $X=X_1 \cup X_2$. 
	Then at least one of $X_1$ and $X_2$ has a nonempty interior.
\end{lemma}
\begin{proof}
	Lemma \ref{lem:l-b} holds for $n=1$ by Proposition \ref{prop:char_lvisceral}.
	We next consider the case in which $n>1$.
	Let $\pi:M^n \to M^{n-1}$ be the projection forgetting the last coordinate.
	For any $x \in M^{n-1}$, we set $X_x=\{y \in M\;|\; (x,y) \in X\}$.
	We define $(X_1)_x$ and $(X_2)_x$ in the same manner.
	Since $X$ has a nonempty interior, the set $U=\{x \in \pi(X)\;|\; X_x \text{ has a nonempty interior}\}$ has a nonempty interior.
	We define $U_i =\{x \in U\;|\; (X_i)_x \text{ has a nonempty interior}\}$ for $i=1,2$.
	We have $U=U_1 \cup U_2$ because the lemma holds for $n=1$.
	At least one of $U_1$ and $U_2$ has a nonempty interior by the induction hypothesis.
	We may assume that $U_1$ has a nonempty interior without loss of generality.
	The definable set $X_1$ has a nonempty interior by property (B).
\end{proof}

\begin{lemma}\label{lem:pre0}
	Consider a locally l-visceral structure $\mathcal M=(M,<,\ldots)$ enjoying properties (A) and (B) in Definition \ref{def:tame_top2}.
	Let $X$ be a definable subset of $M^n$ of dimension $d$ and $\pi:M^n \rightarrow M^d$ be a coordinate projection such that the projection image $\pi(X)$ has a nonempty interior.
	There exists a definable open subset $U$ of $M^d$ contained in $\pi(X)$ such that the fibers $X \cap \pi^{-1}(x)$ are discrete for all $x \in U$. 
	In addition, the definable set $S=\{x \in \pi(X)\;|\; X \cap \pi^{-1}(x) \text{ is not discrete}\}$ has an empty interior.
\end{lemma}
\begin{proof}
	See \cite[Lemma 3.3]{Fuji_tame} and its proof.
	We can prove the lemma in the same manner as \cite[Lemma 3.3]{Fuji_tame} by using Proposition \ref{prop:zero} and Lemma \ref{lem:l-b} instead of Proposition 3.2 and property (b) in \cite{Fuji_tame}.
\end{proof}

We can prove several simple formulas on dimension of definable sets using the above assertions.
\begin{theorem}\label{thm:dimension_basic}
	Consider a locally l-visceral structure $\mathcal M=(M,<,\ldots)$ enjoying properties (A) and (B) in Definition \ref{def:tame_top2}.
	The following assertions hold true.
	We assume that the sets $X$ and $Y$ in the assertions are nonempty and definable.
	\begin{enumerate}
			\item[(1)] Let $X$ and $Y$ be definable subsets of $M^n$.
			We have 
			\begin{align*}
					\dim(X \cup Y)=\max\{\dim(X),\dim(Y)\}\text{.}
				\end{align*}
			\item[(2)] Let $\varphi:X \rightarrow Y$ be a definable surjective map whose fibers are equi-dimensional; that is, the dimensions of the fibers $\varphi^{-1}(y)$ are constant.
			We have $\dim X = \dim Y + \dim \varphi^{-1}(y)$ for all $y \in Y$.  
			This formula is called the addition property in \cite{W}.
			\item[(3)] Let $f:X \rightarrow M^n$ be a definable map. 
			We have $\dim(f(X)) \leq \dim X$.
		\end{enumerate}
\end{theorem}
\begin{proof}
	Assertion (1) can be proved in the same manner as \cite[Theorem 3.8(4)]{Fuji_tame} using Lemma \ref{lem:l-b}.
	We omit the details.
	
	We next tackle assertions (2) and (3).
	We first demonstrate the following claim:
	\medskip

	\textbf{Claim.} Let $X$ be a definable subset of $M^{n+1}$.
	We set $X_x:=\{y \in M\;|\; (x,y) \in X\}$ for each $x \in M^n$.
	Put $X(i)=\{x \in M^n\;|\; \dim X_x=i\}$ for $i=0,1$.
	Then the equality $\dim (X \cap (X(i) \times M))=\dim X(i)+i$ holds. 
	\begin{proof}[Proof of Claim]
		Let $\Pi:M^{n+1} \to M^n$ be the coordinate projection forgetting the last coordinate.
		Set $d=\dim \Pi(X)$.
		By considering $X \cap (X(i) \times M)$ instead of $X$, we may assume that $\Pi(X)=X(i)$ for some $i$.
		
		We first consider the case in which $i=1$.
		We first prove $\dim X \leq d+1$.
		Assume for contradiction that $\dim X>d+1$.
		Let $\pi':M^{n+1} \to M^{d+2}$ be a coordinate projection such that $\pi'(X)$ has a nonempty interior.
		The projection does not forget the last coordinate; otherwise, $\pi'(X)$ is the image of $\Pi(X)$ under some coordinate projection, which contradicts the equality $\dim \Pi(X)=d$.
		Let $\pi'':M^{n+1} \to M^{d+1}$ be the composition of $\pi'$ with the projection of $M^{d+2}$ forgetting the last coordinate.
		The image $\pi''(X)$ has a nonempty interior and it coincides with the image of $\Pi(X)$ under some coordinate projection.
		This is also a contradiction. 
		We get $\dim X \geq d+1$.
		
		We can find a coordinate projection $\pi:M^n \to M^d$ such that $\pi(\Pi(X))$ has a nonempty interior.
		We may assume that $\pi$ is the projection onto the first $d$ coordinates without loss of generality.
		Let us consider the projection $\rho:M^{n+1} \to M^{d+1}$ given by $\rho(x,y)=(\pi(x),y)$ for $x \in M^n$ and $y \in M$.
		Set $Y=\rho(X)$ and $Y_t=\{y \in M\;|\; (t,y) \in Y\}$ for each $t \in M^d$.
		The equality $Y_t=\bigcup_{x \in \pi^{-1}(t) \cap \Pi(X)} X_x$ holds and $X_x$ has a nonempty interior for each $x \in \Pi(X)$ by the assumption.
		This implies that the fiber $Y_t$ has a nonempty interior for each $t \in \pi(\Pi(X))$.
		The definable set $Y$ has a nonempty interior by property (B).
		This means that $\dim X \geq d+1$.
		We have shown that $\dim X= d+1$.
		
		We next consider the case in which $i=0$.
		It is obvious that $\dim X \geq d$.
		We have only to show $\dim X \leq d$.
		Assume for contradiction that $\dim X>d$.
		There exists a coordinate projection $\pi':M^{n+1} \to M^{d+1}$ such that $\pi'(X)$ has a nonempty interior.
		We can easily lead to the contradiction that $\dim \Pi(X) >d$ if we assume that $\pi'$ forgets the last coordinate.
		Therefore, there exists a coordinate projection $\pi:M^n \to M^d$ such that $\pi'(x,y)=(\pi(x),y)$ for every $x \in M^n$ and $y \in M$.
		We can take a nonempty open box $B$ in $M^d$ and a nonempty open interval $I$ such that $B \times I \subseteq \pi'(X)$.
		On the other hand, the definable set $S=\{t \in \pi(\Pi(X))\;|\; \Pi(X) \cap \pi^{-1}(t) \text{ is not discrete}\}$ has an empty interior by Lemma \ref{lem:pre0}  because $\dim \Pi(X)=d$.
		The difference $B \setminus S$ has a nonempty interior by Lemma \ref{lem:l-b}.
		We may assume that $\Pi(X) \cap \pi^{-1}(t)$ is discrete for every $t \in B$ by shrinking $B$ so that $B \cap S=\emptyset$ if necessary.
		
		For each $t \in B$, we have $$(\pi'(X))_t:=\{y \in M\;|\; (t,y) \in \pi'(X)\}=\bigcup_{x \in \pi^{-1}(t) \cap \Pi(X)}X_x.$$
		The definable set $\bigcup_{x \in \pi^{-1}(t) \cap \Pi(X)}\{x\} \times X_x$ is discrete by Lemma \ref{lem:aaa}.
		Its image under the coordinate projection onto the last coordinate coincides with  $(\pi'(X))_t$.
		The set $(\pi'(X))_t$ is discrete by property (A).
		On the other hand, $(\pi'(X))_t$ contains an open interval $I$, which is absurd.
	\end{proof}

	We now return to the proof of the theorem.
	The dimension function $\dim$ satisfies the requirements given in Definition of \cite{vdD-dim} by Proposition \ref{prop:dim_basic}, assertion (1) of the theorem and Claim.
	Assertions (2) and (3) follow from \cite[Corollary 1.5]{vdD-dim}.
\end{proof}

We find a sufficient condition for properties (A) and (B).
\begin{definition}
	Consider a model theoretic structure $\mathcal M=(M,\ldots)$.
	We say that $\mathcal M$ enjoys the \textit{definable choice property} if, for any coordinate projection $\pi:M^n \to M^d$ and a nonempty definable subset $X$ of $M^n$, there exists a definable map $\varphi:\pi(X) \to X$ such that the composition $\pi \circ \varphi$ is the identity map on $\pi(X)$.
\end{definition}

\begin{definition}
	Consider an expansion of a dense linear order without endpoints $\mathcal M=(M,<,\ldots)$.
	We say that $\mathcal M$ enjoys the \textit{univariate continuity property} if, for every definable function $f:I \to M$ from a nonempty open interval $I$, there exists a nonempty open subinterval $J$ of $I$ such that the restriction of $f$ to $J$ is continuous.
	We say that $\mathcal M$ enjoys the \textit{continuity property} if, for every definable function $f:B \to M$ from a nonempty open box $B$, there exists a nonempty open box $U$ contained in $B$ such that the restriction of $f$ to $U$ is continuous. 
\end{definition}

\begin{remark}
	An assertion similar to Proposition \ref{prop:equiv_*_cont} holds for the univariate continuity property and the continuity property.
	Its proof is almost the same as that of Proposition \ref{prop:equiv_*_cont}.
	We omit the proof.
\end{remark}

\begin{proposition}\label{prop:suff_visceral}
A locally l-visceral structure enjoys properties (A) and (B) in Definition \ref{def:tame_top2} if it possesses the definable choice property and the continuity property.
\end{proposition}
\begin{proof}
	Let $\mathcal M=(M,<,\ldots)$ be a structure satisfying the conditions in the proposition.
	We first show that it possesses property (A).
	Let $X$ be a definable discrete subset of $M^n$ and $\pi:M^n \to M^d$ be a coordinate projection.
	We have to show that $\pi(X)$ is discrete.
	We first reduce to the case in which $d=1$.
	Let $\rho_i:M^d \to M$ be the coordinate projection onto the $i$-th coordinate for $1 \leq i \leq d$.
	Assume that $\rho_i(\pi(X))$ is discrete for every  $1 \leq i \leq d$.
	The product $\prod_{i=1}^d \rho_i(\pi(X))$ of discrete sets is discrete.
	The product $\prod_{i=1}^d \rho_i(\pi(X))$ contains the definable set $\pi(X)$.
	The set $\pi(X)$ is discrete because a subset of a discrete set is discrete.
	We have reduced to the case $d=1$.
	
	We consider the case in which $d=1$.
	Assume for contradiction that $\pi(X)$ is not discrete.
	The set $\pi(X)$ has a nonempty interior by Proposition \ref{prop:char_lvisceral}.
	We can find a definable map $f:\pi(X) \to X$ so that $\pi \circ f$ is the identity map on $\pi(X)$ by the definable choice property.
	We can find an open interval $I$ contained in $\pi(X)$ so that the restriction of $f$ to $I$ is continuous thanks to the continuity property.
	The image $f(I)$ is not discrete and $X$ is not discrete because $X$ contains $f(I)$.
	We get a contradiction.
	
	We next tackle property (B).
	Let $X$ be a definable subset of $M^{n+1}$ and $\pi:M^{n+1} \to M^n$ be the coordinate projection forgetting the last coordinate.
	Assume that $\pi(X)$ has a nonempty interior and the fiber $X_x:=\{y \in M\;|\;(x,y) \in X\}$ has a nonempty interior for each $x \in \pi(X)$.
	We have to show that $X$ has a nonempty interior.
	
	We may assume that $X_x$ is open for every $x \in \pi(X)$ by considering $\bigcup_{x \in \pi(X)} \{x\} \times \myint(X_x)$ instead of $X$.
	Take a nonempty box $B$ contained in $\pi(X)$.
	We may assume that $\pi(X)$ is an open box by replacing $X$ with $X \cap \pi^{-1}(B)$.
	By the definable choice property, we can find a definable function $f:\pi(X) \to M$ such that $(x,f(x)) \in X$ for every $x \in \pi(X)$.
	Consider the set $Y=\{(x,y) \in X\;|\; \forall y'\ (f(x)<y'<y) \rightarrow ((x,y') \in X)\}$.
	Since $X_x$ is open, the fiber $Y_x$ is not empty for every $x \in \pi(X)$.
	Apply the definable choice property to $Y$ again, we can find a definable function $g:\pi(X) \to M$ such that $f(x)<g(x)$ and the interval $(f(x),g(x))$ is contained in $X_x$ for every $x \in \pi(X)$.
	We may assume that both $f$ and $g$ are continuous by shrinking $\pi(X)$ if necessary thanks to the continuity property.
	The definable set $\{(x,y) \in \pi(X) \times M\;|\; f(x)<y<g(x)\}$ is an open set contained in $X$.
	We have proven that $X$ has a nonempty interior.
\end{proof}

Expansions of dense linear orders without endpoints possessing the definable choice property and the continuity property enjoy a dimension formula other than those in Theorem \ref{thm:dimension_basic}.
\begin{proposition}\label{prop:local_dim_well}
	Consider an expansion of a dense linear order without endpoints $\mathcal M=(M,<,\ldots)$ enjoying the definable choice property and the continuity property.
	Let $X$ be a definable subset of $M^n$.
	Then there exists a point $x \in X$ such that $\dim (B \cap X)=\dim X$ for every open box $B$ containing the point $x$.
\end{proposition}
\begin{proof}
	Set $d=\dim X$.
	By the definition of dimension, there exists a coordinate projection $\pi:M^n \to M^d$ such that $\pi(X)$ has a nonempty interior.
	By the definable choice property, we can find a definable map $f:\pi(X) \to X$ such that the composition $\pi \circ f$ is the identity map on $\pi(X)$.
	We can take a nonempty open box $U$ in $M^d$ contained in $\pi(X)$ so that the restriction $g:U \to X$ of $f$ to $U$ is continuous by the continuity property.
	Take a point $t \in U$ and set $x=g(t)$.
	Let $B$ be an arbitrary open box containing the point $x$.
	The projection image $\pi(X \cap B)$ contains the inverse image $g^{-1}(B)$ because $\pi \circ g$ is the identity map on $U$.
	The inverse image $g^{-1}(B)$ is open because $g$ is continuous.
	This implies image $\pi(X \cap B)$  has a nonempty interior and we get $\dim (X \cap B) \geq d$.
	The opposite inequality $\dim (X \cap B) \leq d$ is obvious. 
\end{proof}

\subsection{$*$-locally weakly o-minimal case}\label{sec:*-local}
We prove fundamental formulas for topological dimension of sets definable in $*$-locally weakly o-minimal structure possessing $*$-continuity property. 

We first demonstrate that $*$-locally weakly o-minimal structures enjoying the univariate $*$-continuity property possess properties (A) and (B) in Definition \ref{def:tame_top2}. 
The first task is to prove the following two lemmas simultaneously.
\begin{lemma}\label{lem:*-a}
	Consider a $*$-locally l-visceral structure $\mathcal M=(M,<,\ldots)$ enjoying the univariate continuity property.
	Let $X \subseteq M^n$ be a definable discrete set and $\pi:M^n \to M^d$ be the projection onto the first coordinate.
	Then $\pi(X)$ is also discrete. 
\end{lemma}

\begin{lemma}\label{lem:*-d}
	Consider a $*$-locally l-visceral structure $\mathcal M=(M,<,\ldots)$ enjoying the univariate continuity property.
	Let $X$ be a definable subset of $M^{m+n}$ and $\pi:M^{m+n} \to M^m$ be the coordinate projection onto the first $m$ coordinates.
	Assume that $X_x:=\{y \in M^n\;|\;(x,y) \in X\}$ is discrete for each $x \in \pi(X)$.
	Then there exists a definable map $\varphi:\pi(X) \to X$ such that the composition $\pi \circ \varphi$ is the identity map on $\pi(X)$.
\end{lemma}
\begin{proof}[Proof of Lemma \ref{lem:*-a} and Lemma \ref{lem:*-d}]
	We can reduce to the case in which $d=1$ in Lemma \ref{lem:*-a} in the same manner as the proof of Proposition \ref{prop:suff_visceral}.
	
	We prove the lemmas by induction on $n$.
	We first consider the case in which $n=1$.
	Lemma \ref{lem:*-a} is trivial in this case.
	We prove Lemma \ref{lem:*-d}.
	Take $c \in M$ and define a definable map $\eta:\pi(X) \to \overline{M}$ by 
	\begin{align*}
		& \eta(x)=\left\{\begin{array}{cl}
			c & \text{ if } c \in X_x;\\
			\inf (X_x \cap \{x \in M\;|\; x>c\}) & \text{ if } c \notin X_x \text{ and } X_x \cap \{x \in M\;|\; x>c\} \neq \emptyset;\\
			\sup (X_x \cap \{x \in M\;|\; x<c\}) & \text{ elsewhere.}
		\end{array}\right.
	\end{align*}
	We have $\eta(x) \in X_x$ by Lemma \ref{lem:supinf}.
	The map $\varphi:\pi(X) \to X$ defined by $\varphi(x)=(x,\eta(x))$ is a desired map.
	
	We next consider the case in which $n>1$.
	We first prove that Lemma \ref{lem:*-a} holds for $n>1$.
	Assume for contradiction that $\pi(X)$ has a nonempty interior.
	Since Lemma \ref{lem:*-d} holds for $n-1$ by the induction hypothesis, we can construct a definable map $\varphi:\pi(X) \to X$.
	By the univariate continuity property, there exists an open interval $I$ contained in $\pi(X)$ such that the map $\varphi$ restricted to $I$ is continuous.
	The image $\varphi(I)$ is not discrete and contained in $X$.
	This contradicts the assumption that $X$ is discrete.	
	
	We next prove Lemma \ref{lem:*-d} for $n>1$.
	Consider the coordinate projections $\Pi_1:M^{m+n} \to M^{m+n-1}$ and $\Pi_2:M^{m+n-1} \to M^m$.
	The projection $\Pi_1$ is the projection forgetting the last coordinate, and $\Pi_2$ is the projection onto the first $m$ coordinates.
	It is obvious that $\pi=\Pi_2 \circ \Pi_1$.
	It is also trivial that $\{y \in M\;|\; (t,y) \in X\}$ is discrete for every $t \in \Pi_1(X)$.
	By applying Lemma \ref{lem:*-d} for $n=1$, there exists a definable map $\varphi_1: \Pi_1(X) \to X$ such that $\Pi_1 \circ \varphi_1$ is the identity map on $\Pi_1(X)$.
	
	The fiber $(\Pi_1(X))_x=\{y \in M^{n-1}\;|\; (x,y) \in \Pi_1(X)\}$ is discrete for every $x \in \pi(X)$ by Lemma \ref{lem:*-a} for $n$.	
	Apply Lemma \ref{lem:*-d} for $n-1$ to $\Pi_2$.
	There exists a definable map $\varphi_2:\pi(X) \to \Pi_1(X)$ such that the composition $\Pi_2 \circ \varphi_2$ is the identity map on $\pi(X)$.
	The composition $\varphi=\varphi_1 \circ \varphi_2$ is a desired map.  
\end{proof}

We next prove the following lemmas simultaneously.
\begin{lemma}\label{lem:*-key}
Consider a $*$-locally weakly o-minimal structure $\mathcal M=(M,<,\ldots)$ enjoying the univariate $*$-continuity property.
Let $X$ be a definable subset of $M^{n+1}$ and $\pi:M^{n+1} \to M^n$ be the coordinate projection forgetting the last coordinate.
Assume that $\pi(X)$ has a nonempty interior and the fiber $X_x$ has a nonempty interior for each $x \in \pi(X)$.
Then $X$ has a nonempty interior.
\end{lemma}

\begin{lemma}\label{lem:*-b}
	Consider a $*$-locally weakly o-minimal structure $\mathcal M=(M,<,\ldots)$ enjoying the univariate $*$-continuity property.
	Let $X$ be a definable subset of $M^n$ having a nonempty interior.
	Let $X_1$ and $X_2$ be definable subsets of $X$ with $X=X_1 \cup X_2$. 
	Then at least one of $X_1$ and $X_2$ has a nonempty interior.
\end{lemma}
\begin{proof}
	Lemma \ref{lem:*-b} holds for $n=1$ because $\mathcal M$ is locally o-minimal by Proposition \ref{prop:local_ominmal}.
	Recall that the univariate $*$-continuity property implies the $*$-continuity property in this setting by Proposition \ref{prop:continuity2}.
	We next prove Lemma \ref{lem:*-key} for arbitrary $n$ assuming that Lemma \ref{lem:*-b} holds for $n$.
	
	We may assume that $X_x$ is open for each $x \in \pi(X)$ considering the definable set $\bigcup_{x \in \pi(X)} (\{x\} \times \myint(X_x))$ instead of $X$.
	Fix an arbitrary point $c \in M$.
	Set $X_{>c}=\{(x,y) \in \pi(X) \times M\;|\; (x,y) \in X, y>c\}$ and $X_{<c}=\{(x,y) \in \pi(X) \times M\;|\; (x,y) \in X, y<c\}$.
	We have $\pi(X)=\pi(X_{>c}) \cup \pi(X_{<c})$.
	At least one of $\pi(X_{>c})$ and $\pi(X_{<c})$ has a nonempty interior by Lemma \ref{lem:*-b} for $n$.
	We consider the case in which $\pi(X_{>c})$ has a nonempty interior.
	We can treat the other case similarly, and we omit this case in the proof.
	We may assume that $\pi(X_{>c})$ has a nonempty interior and assume that each element in $X_x$ is larger than $c$ for each $x \in \pi(X)$ by considering $X_{>c}$ in place of $X$.
	
	We first consider the definable map $f_1:\pi(X) \to \overline{M}$ given by $f_1(x)=\inf X_x$.
	By the $*$-continuity property, the set of points at which $f_1$ is continuous has a nonempty interior.
	Shrinking $\pi(X)$ if necessary, we may assume that $f_1$ is continuous.
	We next consider the definable map $f_2 :\pi(X) \to \overline{M} \cup \{\infty\}$ given by $f_2(x)=\sup\{y \in M\;|\; \forall t \in M\ (f_1(x)<t<y) \rightarrow (x,t) \in X  \}$.
	By $*$-local weak o-minimality, for every $x \in \pi(X)$, there exists an open interval $I$ such that $f_1(x) \in \overline{I}$ and $I \cap X_x$ is a union of a finite set and finitely many open convex sets.
	The intersection $I \cap X_x$ is a union of finitely many open convex set because $X_x$ is open.
	By the definition of $f_1$ and the above fact, the set $\{y \in M\;|\; \forall t \in M\ (f_1(x)<t<y) \rightarrow (x,t) \in X  \}$ contains the leftmost nonempty open convex set in $I \cap X_x$.
	In particular, it is not empty.
	This implies the inequality $f_2(x)>f_1(x)$ for each $x \in \pi(X)$.
	We show that the set $Y:=\{(x,y) \in \pi(X) \times M\;|\; f_1(x)<y<f_2(x)\} \subseteq X$ has a nonempty interior.
	The set $Y$ obviously has a nonempty interior when $f_2^{-1}(\infty)$ has a nonempty interior.
	In the remaining case, $f_2^{-1}(\overline{M})$ has a nonempty interior by the hypothesis and Lemma \ref{lem:*-b} for $n-1$.
	We may assume that $\pi(X)$ is open and $f_2$ is continuous by shrinking $\pi(X)$ if necessary.
	The set $Y$ is open in this case.
	We have proven Lemma \ref{lem:*-key}.
	
	We next prove Lemma \ref{lem:*-b} for arbitrary $n$ assuming that Lemma \ref{lem:*-b} and Lemma \ref{lem:*-key} hold for $n-1$.
	We may assume that $X$ is an open box in the same manner as Theorem \ref{thm:simple_decomp}.
	Let $\pi_1:M^n \to M^{n-1}$ be the projection forgetting the last coordinate.
	We can choose a point $a \in M$ and construct definable function $h:\pi_1(X) \to \overline{M} \cup \{\infty\}$ so that $a<h(x)$ and $\{x\} \times (a,y) \subseteq X_1$ for every $x \in \pi(X)$ and $a<y<h(x)$ in the same manner as Theorem \ref{thm:simple_decomp}.
	Shrinking $\pi(X)$ if necessary, we may assume that $h$ is continuous by the $*$-continuity property and Lemma \ref{lem:*-b} for $n-1$ in the same manner as above.
	The set $\{(x,y) \in \pi(X) \times M\;|\; a < y <h(x)\}$ is open and contained in $X_1$.
\end{proof}

We proved fundamental formulas on the dimension function in definably complete locally o-minimal structures in \cite{Fuji_tame} assuming the following properties (a) through (d).  

\begin{definition}\label{def:tame_top_old}
	Consider an expansion of a dense linear order without endpoints $\mathcal M=(M,<,\ldots)$.
	We consider the following properties on $\mathcal M$.
	\begin{enumerate}
		\item[(a)] Same as property (A) in Definition \ref{def:tame_top2}.
		\item[(b)] Let $X_1$ and $X_2$ be definable subsets of $M^m$.
		Set $X=X_1 \cup X_2$.
		Assume that $X$ has a nonempty interior.
		At least one of $X_1$ and $X_2$ has a nonempty interior.
		\item[(c)] Let $A$ be a definable subset of $M^m$ with a nonempty interior and $f:A \rightarrow M^n$ be a definable map.
		There exists a definable open subset $U$ of $M^m$ contained in $A$ such that the restriction of $f$ to $U$ is continuous.
		\item[(d)] Let $X$ be a definable subset of $M^n$ and $\pi: M^n \rightarrow M^d$ be a coordinate projection such that the the fibers $X \cap \pi^{-1}(x)$ are discrete for all $x \in \pi(X)$.
		Then, there exists a definable map $\tau:\pi(X) \rightarrow X$ such that $\pi(\tau(x))=x$ for all $x \in \pi(X)$.
	\end{enumerate}
\end{definition} 
\begin{remark}\label{rem:cccc}
	Property (b) follows from property (c).
	Let $c_1,c_2,c_3$ be distinct elements in $M$.
	We define a definable function $f:X \to M$ by $f(x)=c_1$ if $x \in X_1 \cap X_2$, $f(x)=c_2$ if $x \in X_1 \setminus X_2$ and $f(x)=c_3$ if $x \in X_2 \setminus X_1$.
	Let $U$ be a definable open subset of $M^n$ contained $X$ such that the restriction of $f$ to $U$ is continuous.
	It is obvious that $f$ is constant on $U$.
	It implies that $U$ is entirely contained in at least one of $X_1$ and $X_2$.
	
	Property (a) follows from property (d) and the univariate continuity property when the structure is locally l-visceral.
	We can prove this fact in the same manner as \cite[Theorem 2.11(iv)]{Fuji_tame} and Proposition \ref{prop:suff_visceral}.
	We omit the proof.
\end{remark}

We have already proven that a $*$-locally weakly o-minimal structure possesses  properties (a), (b) and (d) in Lemma \ref{lem:*-a}, Lemma \ref{lem:*-b} and Lemma \ref{lem:*-d}, respectively, when it enjoys the univariate $*$-continuity property.
The property (c) is nothing but a weaker version of the $*$-continuity property.
Lemma \ref{lem:*-key} is nothing but property (B) in Definition \ref{def:tame_top2}.
The assumptions of Theorem \ref{thm:dimension_basic} are satisfied in $*$-locally weakly o-minimal structure enjoying the univariate $*$-continuity property.
Therefore, the formulas in Theorem \ref{thm:dimension_basic} hold. 
In addition, some extra formulas can be proven in this case.
We recall a technical lemma so as to prove the extra formulas.

\begin{lemma}\label{lem:pre1}
	Consider a locally l-visceral structure $\mathcal M=(M,<,\ldots)$ enjoying the $*$-continuity property and property (d) in Definition \ref{def:tame_top_old}.
	Let $X \subseteq Y$ be definable subsets of $M^n$.
	Assume that there exist a coordinate projection $\pi:M^n \rightarrow M^d$ and a definable open subset $U$ of $M^d$ contained in $\pi(X)$ such that the fibers $Y_x$ are discrete for all $x \in U$. 
	Then, there exist
	\begin{itemize}
		\item a definable open subset $V$ of $U$,
		\item a definable open subset $W$ of $M^n$ and 
		\item a definable continuous map $f:V \rightarrow X$
	\end{itemize}
	such that 
	\begin{itemize}
		\item $\pi(W)=V$, 
		\item $Y \cap W = f(V)$ and 
		\item the composition $\pi \circ f$ is the identity map on $V$.  
	\end{itemize}
\end{lemma}
\begin{proof}
	The same proof as \cite[Lemma 3.4]{Fuji_tame} works except a minor point.
	The proof of this lemma is brief, and we give a complete proof here for readers' convenience.
	Note that properties (a) and (b) hold under this assumption by Remark \ref{rem:cccc}.
	
	Permuting the coordinates if necessary, we may assume that $\pi$ is the projection onto the first $d$ coordinates.
	Let $\rho_j:M^n \rightarrow M$ be the coordinate projections onto the $j$-th coordinate for all $d<j \leq n$.
	The fiber $Y_x$ is discrete for any $x \in U$ by the assumption.
	Since $X_x$ is a definable subset of $Y_x$, $X_x$ is also a discrete set.
	There exists a definable map $g:U \rightarrow X$ such that the composition $\pi \circ g$ is the identity map on $U$ by property (d).
	The projection image $\rho_j(Y_x)$ is discrete and closed by property (a) and Lemma \ref{lem:key0}.
	Consider the definable functions $\kappa_j^+:U \rightarrow \overline{M} \cup \{+\infty\}$ defined by 
	\[
	\kappa_j^+(x)=\left\{
	\begin{array}{ll}
		\inf \{t \in \rho_j(Y_x)\;|\; t>\rho_j(g(x))\} & \text{if } \{t \in \rho_j(Y_x)\;|\; t>\rho_j(g(x))\}\not=\emptyset \text{,}\\
		+\infty & \text{otherwise}
	\end{array}
	\right.
	\]
	for all $d<j \leq n$.
	We define $\kappa_j^-:U \rightarrow \overline{M} \cup \{-\infty\}$ similarly.
	Note that $\kappa_j^-(x) \neq \rho_j(g(x))$ and $\kappa_j^+(x) \neq \rho_j(g(x))$ for each $x \in U$ because $\rho_j(Y_x)$ is discrete and closed.
	Then, we have
	\[
	\pi^{-1}(x) \cap Y \cap (\{x\} \times (\kappa_{d+1}^-(x),\kappa_{d+1}^+(x)) \times \cdots \times (\kappa_{n}^-(x),\kappa_{n}^+(x))) = \{g(x)\}
	\]
	for all $x \in U$.
	Here,  $(\kappa_{i}^-(x),\kappa_{i}^+(x))$ denotes the definable open set $\{t \in M\;|\; \kappa_{i}^-(x)<t<\kappa_{i}^+(x)\}$, which is not necessarily an open interval, for every $d<i \leq n$.
	There exists a  definable open subset $V$ of $U$ such that the restriction $f$ of $g$ to $V$ and the restrictions of $\kappa_j^-$ and $\kappa_j^+$ to $V$ are all continuous by properties (b) and the $*$-continuity property.
	Set $W=\{(x,y_{d+1},\ldots, y_n) \in V \times M^{n-d}\;|\; \kappa_j^-(x) < y_j < \kappa_j^+(x) \text{ for all } d<j \leq n\}$.
	The definable sets $V$ and $W$ and a definable continuous map $f$ satisfy the requirements.
\end{proof}

%

In \cite[Section 3]{Fuji_tame}, we proved four formulas on dimension of sets definable in definably complete locally o-minimal structures which are not found in Proposition \ref{prop:dim_basic} and Theorem \ref{thm:dimension_basic}.
We only use a counterpart to Lemma \ref{lem:pre1} for their proofs.
We can prove the following theorem using Lemma \ref{lem:pre1} instead of its counterpart.
\begin{theorem}\label{thm:dim}
	Consider a locally l-visceral structure $\mathcal M=(M,<,\ldots)$.
	Assume further that $\mathcal M$ enjoys property (d) in Definition \ref{def:tame_top_old}, property (B) in Definition \ref{def:tame_top2} and the $*$-continuity property.
	The following assertions hold true.
	We assume that the sets $X$ and $Y$ in the assertions are nonempty and definable.
	\begin{enumerate}
		\item[(1)] Let $f:X \rightarrow M^n$ be a definable map. 
		The notation $\mathcal D(f)$ denotes the set of points at which the map $f$ is discontinuous. 
		The inequality $\dim(\mathcal D(f)) < \dim X$ holds true.
		\item[(2)] Let $X$ be a definable set.
		The notation $\partial X$ denotes the frontier of $X$ defined by $\partial X = \mycl(X) \setminus X$.
		We have $\dim (\partial X) < \dim X$.
		\item[(3)] A definable set $X$ is of dimension $d$ if and only if the nonnegative integer $d$ is the maximum of nonnegative integers $e$ such that there exist an open box $B$ in $M^e$ and 
		a definable injective continuous map $\varphi:B \rightarrow X$ homeomorphic onto its image. 
		\item[(4)] Let $X$ be a definable subset of $M^n$.
		There exists a point $x \in X$ such that we have $\dim(X \cap B)=\dim(X)$ for any open box $B$ containing the point $x$.
	\end{enumerate}
\end{theorem}
\begin{proof}
	Note that $\mathcal M$ possess property (c) because it enjoys the $*$-continuity property which is stronger than property (c).
	In addition, properties (a) and (b) in Definition \ref{def:tame_top_old} follow from property (d) and the $*$-continuity property by Remark \ref{rem:cccc}.
	Observe that the assumptions in Theorem \ref{thm:dimension_basic} hold in this setting.
	
	Assertions (1) and (2) correspond to \cite[Theorem 3.8(6),(7)]{Fuji_tame}.
	We can prove them in the same manner as their counterparts using Lemma \ref{lem:pre1}.
	
	Assertion (3) is a counterpart to \cite[Theorem 3.11]{Fuji_tame}.
	It can be proven in the same manner as \cite[Theorem 3.11]{Fuji_tame} using Lemma \ref{lem:pre1} and the formulas in Proposition \ref{prop:dim_basic} and Theorem \ref{thm:dimension_basic}.
	Assertion (4) is a corollary of assertion (3), and it is proven in the same manner as \cite[Corollary 3.12]{Fuji_tame}. 
\end{proof}

We get the following corollary, which is a main theorem of this paper, by summarizing the assertions in this section.
\begin{corollary}\label{cor:*-dim}
	Assertions (1) through (3) in Theorem \ref{thm:dimension_basic} and assertions (1) through (4) in Theorem \ref{thm:dim} hold in $*$-locally weakly o-minimal structures enjoying the univariate $*$-continuity property.
\end{corollary}
\begin{proof}
	Note that a $*$-local weakly o-minimal structure is locally o-minimal by Proposition \ref{prop:local_ominmal}.
	In particular, it is also locally l-visceral by Proposition \ref{prop:char_lvisceral}.
	It also enjoys the $*$-continuity property by Proposition \ref{prop:continuity2}.
	The corollary follows from Theorem \ref{thm:dimension_basic}, Theorem \ref{thm:dim}, Lemma \ref{lem:*-a}, Lemma \ref{lem:*-d} and Lemma \ref{lem:*-key}.
\end{proof}

Note that a weakly o-minimal structure is $*$-locally weakly o-minimal.
In \cite[Theorem 4.2]{W}, Wencel gave a necessary and sufficient condition for a weakly o-minimal structure to satisfy the equality in Theorem \ref{thm:dimension_basic}(2).
The following corollary is a generalization of Wencel's result, and it is another main theorem of this paper.
\begin{corollary}\label{cor:wencel}
Let $\mathcal M=(M,<,\ldots)$ be a $*$-locally weakly o-minimal structure.
The equality in Theorem \ref{thm:dimension_basic}(2) holds if and only if $\mathcal M$ has the univariate $*$-continuity property.
\end{corollary}
\begin{proof}
The `if' part follows from Corollary \ref{cor:*-dim}.
We next show the contraposition of the `only if' part.
We have only to show that $\mathcal M$ has a dimensionally wild set, which obviously violates the equality in Theorem \ref{thm:dimension_basic}(2).
We consider two separate cases.
We first consider the case in which assertion (1) in Theorem \ref{thm:monotone-star} holds.
The structure $\mathcal M$ has a dimensionally wild set in this case.

We next consider the case in which assertion (2) in Theorem \ref{thm:monotone-star} holds.
Let $f:I \to \overline{M}$ be a definable function from an open interval $I$ such that the set $C$ of points at which $f$ is continuous has an empty interior.
The set $C$ is discrete and closed by Proposition \ref{prop:char_lvisceral} because $\mathcal M$ is locally l-visceral.
We may assume that $f$ is discontinuous everywhere by replacing $I$ with an open interval contained in $I$ and having an empty intersection with $C$.
We may assume that $f$ is strictly monotone by shrinking $I$ if necessary because assertion (2) in Theorem \ref{thm:monotone-star} holds in this case.
The structure $\mathcal M$ has a dimensionally wild definable set by Lemma \ref{lem:exceptional_case}.
\end{proof}

\subsection{Decomposition into definable submanifolds}\label{sec:decomposition}
We showed that any definable set is partitioned into finitely many quasi-special submanifolds, which are definable submanifolds satisfying an extra property, in \cite[Section 3]{Fuji_tame} when the structure is a definably complete locally o-minimal structure.
We prove a weaker decomposition theorem when the structure is a $*$-locally weakly o-minimal structure enjoying the univariate $*$-continuity property.

\begin{definition}
	Consider an expansion of a densely linearly order without endpoints $M = (M, <,\ldots)$. 
	Let $X$ be a definable subset of $M^n$ and $\pi : M^n \to M^d$
	be a coordinate projection. 
	A point $x \in X$ is \textit{$(X,\pi)$-normal} if there exists an open
	box $B$ in $M^n$ containing the point $x$ such that $B \cap X$ is the graph of a continuous map defined on $\pi(B)$ after permuting the coordinates so that $\pi$ is the projection	onto the first $d$ coordinates.
	The definable subset $X$ of $M^n$ is called a \textit{definable submanifold} if, for any point $x \in X$, there exists a definable open subsets $U$ and $V$ of $M^n$ and a definable homeomorphism $\varphi:U \to V$ such that $U$ contains the $x$, $V$ contains the origin $\overline{0}$ of $M^n$, $\varphi(x)=\overline{0}$ and $\varphi(X \cap U)=V \cap (M^d \times \{(0,\ldots,0)\})$.
	It is obvious that the definable set $X$ is a definable submanifold if every point $x$ in $X$ is $(X,\pi)$-normal.
	We call $X$ a \textit{$\pi$-normal submanifold} or simply a \textit{normal submanifold} if every point in $X$ is $(X,\pi)$-normal.
	
	Let $\{X_i\}_{i=1}^m$ be a finite family of definable subsets of $M^n$. A decomposition of $M^n$ into normal submanifolds partitioning $\{X_i\}_{i=1}^m$ is a finite family of normal submanifolds $\{C_i \}_{i=1}^N$ such that
	either $C_i$ has an empty intersection with $X_j$ or is contained in $X_j$ for any $1 \leq i \leq N$ and $1 \leq j \leq m$. A decomposition $\{C_i \}_{i=1}^N$ into normal submanifolds satisfies the frontier condition if the closure of any normal submanifold $C_i$ is the union of a subfamily of the decomposition.
\end{definition}
The notion of normal submanifolds is a weaker concept than the notion of quasi-special submanifolds defined in \cite[Definition 4.1]{Fuji_tame}.
This means that a quasi-special submanifold is always a normal submanifold.
The opposite implication holds by \cite[Lemma 4.2]{Fuji_tame} and \cite[Theorem 2.5]{FKK} when the structure is definably complete and locally o-minimal.

We give a decomposition theorem below:

\begin{theorem}\label{thm:decomposition}
	Consider a $*$-locally weakly o-minimal structure $\mathcal M=(M,<,\ldots)$ enjoying the $*$-continuity property.
	Let $\{X_i\}_{i=1}^m$ be a finite family of definable subsets of $M^n$.  
	There exists a decomposition $\{C_i \}_{i=1}^N$ into normal submanifolds
	partitioning $\{X_i\}_{i=1}^m$ and satisfying the frontier condition. 
	Furthermore,
	the number $N$ of normal submanifolds is not greater than the number uniquely
	determined only by $m$ and $n$.
\end{theorem}
\begin{proof}
	A similar assertion is already proven in \cite[Section 4]{Fuji_tame}.
	We can prove the theorem in the same manner as Lemma 4.3 through Theorem 4.5 in \cite{Fuji_tame} using Corollary \ref{cor:*-dim}.
	We omit the proof.
\end{proof}

\subsection{Other definition of dimension and their equivalence}\label{sec:other}
We introduced two other definitions of dimension and showed their equivalence to the topological dimension when the structure is definably complete and locally o-minimal.
We can show analogous results for $*$-locally weakly o-minimal structures possessing the univariate $*$-continuity property.

 We recall the notion of a pregeometry.
 It is found in \cite[Chapter 2]{Pillay} and \cite[Appendix C]{TZ}.
 
 \begin{definition}[Pregeometry]\label{def:pregeometry}
 	Consider a set $S$ and a map $\mycl:\mathcal P(S) \rightarrow \mathcal P(S)$, where $\mathcal P(S)$ denotes the power set of $S$.
 	The pair $(S,\mycl)$ is a \textit{(combinatorial) pregeometry} if the following conditions are satisfied for any subset $A$ of $S$:
 	\begin{enumerate}
 		\item[(i)] $A \subseteq \mycl(A)$;
 		\item[(ii)] $\mycl(\mycl(A))=\mycl(A)$;
 		\item[(iii)] For any $a,b \in S$, if $a \in \mycl(A \cup \{b\}) \setminus \mycl(A)$, then $b \in \mycl(A \cup \{a\})$;
 		\item[(iv)] For any $a \in \mycl(A)$, there exists a finite subset $Y$ of $A$ such that $a \in \mycl(Y)$. 
 	\end{enumerate} 
 	The condition (iii) is called the \textit{exchange property}.
 	
 	Consider a pregeometry $(S,\mycl)$.
 	Let $A$ and $B$ be subsets of $S$.
 	The set $A$ is $\mycl$-\textit{independent} over $B$ if, for any $a \in A$, we have $a \not\in \mycl((A \setminus \{a\}) \cup B)$.
 	A subset $A_0$ of $A$ is a $\mycl$-\textit{basis} for $A$ over $B$ if $A$ is contained in $\mycl(A_0 \cup B)$ and $A_0$ is $\mycl$-independent over $B$.
 	Each basis for $A$ over $B$ has the same cardinality and it is denoted by $\myrk^{\mycl}(A/B)$.
 \end{definition}

\begin{definition}[Discrete closure]\label{def:dcl}
	Let $\mathcal L$ be a language containing a binary predicate $<$.
	Consider an $\mathcal L$-structure $\mathcal M=(M,<,\ldots)$ which is an expansion of an order without endpoints.
	The \textit{discrete closure} $\mydcl_{\mathcal M}(A)$ of a subset $A$ of $M$ is the set of points $x$ in $M$ having an $\mathcal L(A)$-formula $\phi(t)$ such that 
	\begin{enumerate}
		\item[(a)] the set $\phi(\mathcal M)$ contains the point $x$ and 
		\item[(b)] it is discrete and closed.
	\end{enumerate}
	We often omit the subscript $\mathcal M$ of $\mydcl_{\mathcal M}(A)$.
\end{definition}

\begin{theorem}\label{thm:pregeometry}
	Let $\mathcal L$ be a language containing the binary predicate $<$.
	Consider a $*$-locally weakly o-minimal $\mathcal L$-structure $\mathcal M=(M,<,\ldots)$ possessing the univariate $*$-continuity property.
	The pair $(M,\mydcl)$ is a pregeometry.
\end{theorem}
\begin{proof}
We have already proven a similar theorem when the structure is definably complete and locally o-minimal in \cite[Theorem 3.2]{Fuji_pregeo}.
We can prove the theorem using Proposition \ref{prop:zero}, Lemma \ref{lem:*-d}, Corollary \ref{cor:monotone-star} and the dimension formulas in Proposition \ref{prop:dim_basic}, Theorem \ref{thm:dimension_basic} and Theorem \ref{thm:dim}.
We omit the proof.
\end{proof}

\begin{definition}\label{def:rank}
	Let $\mathcal L$ be a language and $\mathcal M=(M,\ldots)$ be an $\mathcal L$-structure.
	Let $\mycl:\mathcal P(M) \rightarrow \mathcal P(M)$ be a closure operator such that the pair $(M,\mycl)$ is a pregeometry.
	Consider subsets $A$ and $S$ of $M$ and $M^n$, respectively.
	We define 
	$$\myrk_{\mathcal M}^{\mycl}(S/A)=\max \{\myrk^{\mycl}(\{a_1,\ldots,a_n\}/A)\;|\; (a_1,\ldots, a_n) \in S\}.$$
	
	Let $\mathcal L$ be a language and $T$ be its theory.
	Consider a model $\mathcal M=(M,\ldots)$ of $T$ and a monster model $\mathbb M$ of $T$.
	The universe of $\mathbb M$ is denoted by the same symbol $\mathbb M$ for simplicity.
	Assume that there exists a closure operator $\mycl:\mathcal P(\mathbb M) \rightarrow \mathcal P(\mathbb M)$ such that the pair $(\mathbb M,\mycl)$ is a pregeometry. 
	Let $A$ be a subset of $M$ and $S$ be a definable set.
	We set $$\myrk_T^{\mycl} (S/A)=\myrk_{\mathbb M}^{\mycl}(S^{\mathbb M}/A).$$
	We simply denote it by $\myrk(S/A)$ when $T$ and $\mycl$ are clear from the context.
	It is often called the \textit{rank of $S$ over $A$}, but we call it \textit{$\mycl$-dimension of $S$ over $A$} in this paper in order to avoid the confusion with Pillay's dimension rank introduced later.
	Here, the prefix $\mycl$ of $\mycl$-dimension is the closure operator of the pregeometry $(\mathbb M,\mycl)$.
	%
\end{definition}

\begin{definition}
	Let $\mathcal L$ be a language containing a binary predicate $<$ and $T$ be a complete $\mathcal L$-theory.
	The theory $T$ is called a \textit{$*$-locally weakly o-minimal $\mathcal L$-theory possessing the univariate $*$-continuity property} if every model of $T$ is a $*$-locally weakly o-minimal structure possessing the univariate $*$-continuity property.
	The theory $T$ is a $*$-locally weakly o-minimal $\mathcal L$-theory possessing the univariate $*$-continuity property if a model of $T$ is a $*$-locally weakly o-minimal structure possessing the univariate $*$-continuity property by Proposition \ref{prop:*-local_elementary} and Remark \ref{rem:equiv_*_cont}.
\end{definition}

\begin{theorem}\label{thm:rank}
	Let $\mathcal L$ be a language containing a binary predicate $<$ and $T$ be a $*$-locally weakly o-minimal $\mathcal L$-theory possessing the univariate $*$-continuity property.
	Let $\mathcal M=(M,\ldots)$ be a model of $T$.
	Let $A$ be a subset of $M$ and $X$ be a subset of $M^n$ definable over $A$.
	We have $\dim X=\myrk_T^{\mydcl}(X/A)$.
	In particular, the $\mydcl$-dimension of $X$ over $A$ is independent of the parameter set $A$.
\end{theorem}
\begin{proof}
	The theorem follows in the same manner as \cite[Theorem 3.5]{Fuji_pregeo}.
	We used a decomposition into quasi-special submanifolds in the proof of \cite[Theorem 3.5]{Fuji_pregeo}, but a decomposition into normal submanifold in Theorem \ref{thm:decomposition} is sufficient for our purpose.
	We also need to use Lemma \ref{lem:supinf}.
	We omit the complete proof here.
\end{proof}

Pillay defined a first-order topological structure and the dimension rank for definable sets \cite{Pillay3}.
\begin{definition}
	Let $\mathcal L$ be a language and $\mathcal M=(M,\ldots)$ be an $\mathcal L$-structure.
	The structure $\mathcal M$ is called a \textit{first-order topological structure} if there exists an $\mathcal L$-formula $\phi(x,\overline{y})$ such that the family $\{\phi(x,\overline{a})\;|\; \overline{a} \subseteq M\}$ is a basis for a topology on $M$.
	When $\mathcal M$ is an expansion of a dense linear order, then $\mathcal M$ is a first-order topological structure.
	
	Pillay developed the theory of dimension rank for first-order topological structures whose definable sets are constructible.
	Recall that a set is \textit{constructible} if it is a finite boolean combination of open sets.
	Note that any definable set is constructible thanks to the formula in Theorem \ref{thm:dim}(2) when the structure is a $*$-locally weakly o-minimal structure and it possess the univariate $*$-continuity property.
	For a definable set $X$, a dimension rank $\mydim(X)$ is defined as follows:
	\begin{enumerate}
		\item[(1)] If $X$ is nonempty, then $\mydim(X) \geq 0$. Otherwise, set $\mydim(X)=-\infty$.
		\item[(2)] If $\mydim(X) \geq \alpha$ for all $\alpha<\delta$, where $\delta$ is limit, then $\mydim(X) \geq \delta$.
		\item[(3)] $\mydim(X) \geq \alpha+1$ if there exists a definable closed subset $Y$ of $X$ such that $Y$ has an empty interior in $X$ and $\mydim(Y) \geq \alpha$.
	\end{enumerate}
	We put $\mydim(X)=\alpha$ if $\mydim(X) \geq \alpha$ and $\mydim(X) \not\geq \alpha+1$.
	We set $\mydim(X)=\infty$ when $\mydim(X) \geq \alpha$ for all $\alpha$.
\end{definition}

\begin{proposition}\label{prop:top_st_dim}
	Let $\mathcal M=(M,<,\ldots)$ be a $*$-locally weakly o-minimal structures enjoying the univariate $*$-continuity property and $X$ be a definable set.
	We have $\dim X=\mydim(X)$.
\end{proposition}
\begin{proof}
	We have already proven the proposition in \cite[Proposition 4.3]{Fuji_pregeo} under the assumption that the structure is definably complete and locally o-minimal.
	We only use the dimension formulas in Proposition \ref{prop:dim_basic}, Theorem \ref{thm:dimension_basic} and Theorem \ref{thm:dim} in the proof of \cite[Proposition 4.3]{Fuji_pregeo}.
	Therefore, we can prove this proposition in the same manner.
	We omit the proof.
\end{proof}

\newpage
\appendix
\section{Complete proofs of assertions}
We skipped several proofs of the assertions in the main body of this paper because they are almost the same as those in the published papers \cite{Fuji_uniform, Fuji_tame, Fuji_pregeo} though the assumptions employed in the published papers are not identical to ours.
The author gives complete proofs of such assertions here in order to clarify that the proofs in the published papers work in our setting.
A large part of this section consists of copies from the published papers.
This appendix will be removed when this paper is submitted to a journal.
The author does not care about redundancy of this part, he may prove the same claim multiple times. 

\begin{proof}[Proof of Proposition \ref{prop:char_lvisceral}]
	The proof of `in particular' part is missing in the main body.
	We give a proof for them.
	
	We first show that local l-viscerality is preserved under elementary equivalence.
	Let $\phi(x,\overline{z})$ be an formula, where $x$ is a single variable and $\overline{z}$ is a tuple of variables.
	We can construct a formula $\Phi_{\phi}(\overline{z})$ saying that the definable subset $X_{\overline{z}}$ of $M$ defined by $\phi(x,\overline{z})$ is discrete and closed after removing its interior when the tuple of variables $\overline{z}$ is fixed as parameters.
	
	We show that a structure $\mathcal N=(N,<,\ldots)$ elementary equivalent to a locally l-visceral structure $\mathcal M=(M,<,\ldots)$ is locally l-visceral.
	We have $\mathcal M \models \forall z \ \Phi_{\phi}(\overline{z})$ for every formula of the form $\phi(x,\overline{z})$ by the first part of the proposition.
	Since $\mathcal N$ is elementary equivalent to $\mathcal M$, we have $\mathcal N \models \forall z \ \Phi_{\phi}(\overline{z})$.
	This implies that $\mathcal N$ is locally l-visceral by the first part of the proposition.
	
	We next show that the ultraproduct $\mathcal N:=\prod_{i \in I}\mathcal M_i/\mathcal F$ is locally l-visceral when every $\mathcal M_i$ is locally l-visceral.
	Here, $\mathcal F$ is an ultrafilter on the set $I$.
	Since $\mathcal M_i$ is locally l-visceral, we have $\mathcal M_i \models  \forall z \ \Phi_{\phi}(\overline{z})$ for every $i \in I$.
	We have $\mathcal N \models  \forall z \ \Phi_{\phi}(\overline{z})$ by the \L o\'s's theorem.
\end{proof}

\begin{proof}[Proof of Proposition \ref{prop:*-local_ultraproduct}]
	We show that the ultraproduct $\mathcal N:=\prod_{i \in I}\mathcal M_i/\mathcal F$ is $*$-locally weakly o-minimal when every $\mathcal M_i$ is so.
	Here, $\mathcal F$ is an ultrafilter on the set $I$.
	Let $\phi(x,\overline{z})$ be an formula, where $x$ is a single variable and $\overline{z}$ is a tuple of variables.
	
	We can construct a formula $\Pi_{\phi}(\overline{z})$ saying that, for any point $a$, there exists an open interval $(b_1,b_2)$ containing the point $a$ so that $X \cap (b_1,b_2)$ satisfies the condition in Remark \ref{rem:local_omin}.
	Let $\Phi, \Psi$ be the formulas constructed in the proof of Proposition \ref{prop:*-local_elementary}.
	
	We have $\mathcal M_i \models \forall \overline{z} \ \Pi_{\phi}(\overline{z})$ and $\mathcal M_i \models \forall \overline{y} (\Phi(\overline{y}) \to \Psi(\overline{y}))$ for every $i \in I$.
	By \L o\'s's theorem, we have $\mathcal N \models \forall \overline{z} \ \Pi_{\phi}(\overline{z})$ and $\mathcal N \models \forall \overline{y} (\Phi(\overline{y}) \to \Psi(\overline{y}))$.
	These imply that $\mathcal N$ is $*$-locally weakly o-minimal.
\end{proof}

\begin{proof}[Proof of Lemma \ref{lem:mono_key_lemma2}]
	Recall several definitions.
	For any $a \in I$, set $U_a$ as follows:
	\begin{align*}
		U_a&=\{x \in I\;|\; x>a \text{ and } f(y)>f(a) \text{ for all } a < y \leq x\} \cup \{a \} \cup\\
		& \{x \in M\;|\; x<a \text{ and }   f(y)>f(a) \text{ for all } x \leq y <a\}\text{.}
	\end{align*}
	The definable set $U_a$ is convex and it is a neighborhood of the point $a$.
	The notation $a \prec b$ denotes the relation $U_a \varsupsetneq U_b$.
	We prove the following assertions (i) through (ix) only using the assumption that $f$  has local minimums throughout $I$.
	\begin{enumerate}[(i)]
		\item $U_a \not= U_b$ if $a \not= b$;
		\begin{proof}
			We may assume that $f(a) \leq f(b)$ without loss of generality.
			We have $a \notin U_b$ and $a \in U_a$.
			This implies the inequality $U_a \not= U_b$.
		\end{proof}
		\item $a \in \myint(U_a)$;
		\begin{proof}
			It follows from the assumption that $f$ attains local minimum at the point $a$.
		\end{proof}
		\item $a \prec b \Leftrightarrow b \in U_a$;
		\begin{proof}
			We first recall that $a \prec b \Leftrightarrow U_a \supsetneq U_b$. 
			The implication $a \prec b  \Rightarrow b \in U_a$ is obvious because $b \in U_b$.
			We prove the opposite implication.
			We have $f(y)>f(a)$  for every element $y$ between $b$ and $a$ not equal to $a$ because $b \in U_a$.
			Take an arbitrary point $x \in U_b$.
			We have $f(y)> f(b) > f(a)$ for every element $y$ between $b$ and $x$ not equal to $b$ because $x \in U_b$.
			We have proven that $f(y)>f(a)$ for every element between $x$ and $a$ not equal to $a$.
			This means that $x \in U_a$.
		\end{proof}
		\item $U_a \cap U_b \not= \emptyset \Rightarrow a \prec b \text{ or } b \prec a$;
		\begin{proof}
			We assume that $f(a) \leq f(b)$ without loss of generality.
			Take $c \in U_a \cap U_b$.
			We have $f(y)> f(b) \geq f(a)$ for every element between $b$ and $c$ not equal to $b$ because $c \in U_b$.
			We have $f(y) > f(a)$ for every element between $a$ and $c$ not equal to $a$ because $c \in U_a$.
			In summary,  We have $f(y)> f(a)$ for every element between $b$ and $a$ not equal to $a$.
			It means that $b \in U_a$.
			We get $a \prec b$ by (iii).
		\end{proof}
		\item $b \prec a \text{ and } c \prec a \Rightarrow b \prec c \text{ or } c \prec b$;
		\begin{proof}
			By (iii), we have $a \in U_b \cap U_c$.
			The assertion follows from (iv).
		\end{proof}
		\item $a \prec b \prec c \text{ and } a < c \Rightarrow a \leq b $;
		\begin{proof}
			Assume for contradiction that $b<a$.
			Since $b,c \in U_b$ and $U_b$ is convex, we have $a \in U_b$.
			By (iii), we have $b \prec a$.
			It implies that $U_a \supsetneq U_b$.
			On the other hand, $U_b \supsetneq U_a$ because $a \prec b$.
			We get a contradiction.
		\end{proof}
		\item $a \prec b \prec c \text{ and } a > c \Rightarrow a \geq b$;
		\begin{proof}
			We can prove it in the same manner as (vi).
		\end{proof}
		\item $a \prec b \prec c \text{ and } a < b \Rightarrow a<c$;
		\begin{proof}
			Assume for contradiction that $c \leq a$.
			We can lead to a contradiction in the same manner as (vi).
		\end{proof}
		\item the definable set 
		\begin{equation*}
			C_a = \{x \in I\;|\; x \prec a\}
		\end{equation*}
		does not contain an open interval for any $a \in I$.
		\begin{proof}
			We lead to a contradiction assuming that $C_a$ contains an
			open interval $J$ for some $a \in I$. Let $b \in J$. Take $c, d \in U_b \cap J$ with $c < b < d$. Since $c, d \in C_a$ , we have $c \prec a$
			and $d \prec a$; hence, we get $c \prec d$ or $d \prec c$ by (v). We only consider the case in which $c \prec d$ because
			the proof is similar when $d \prec c$. The point $b$ is an element of $U_c$ because $U_c$ is convex, $c < b < d$ and $d \in U_c$.
			It is a contradiction to the conditions that $c \in U_b$ and the assertions (i) and (iii).
		\end{proof}
	\end{enumerate}
	Consider the definable set $C=\{(a,x) \in I^2\;|\; x \in C_a\}$.
	By shrinking $I$ if necessary, we may assume that $C_a$ is a finite set.

	We considered the following sets in the proof in the main body:
	\begin{align*}
		K &= \{ x \in I\;|\; y \not\prec x \text{ for all } y \in I\}\text{ and }\\
		\widetilde{a} &= \{ x \in I\;|\; a \prec x \text{ and } \nexists y \in I \ a \prec y \prec x\}\text{,}
	\end{align*}
	We proved the following equality:
	\begin{equation}\label{eq:5aaa}
		I \setminus K = \displaystyle\mathop{\dot{\bigcup}}_{a \in I}\widetilde{a}\text{.}
	\end{equation}
	We also proved that $\widetilde{a}$ and $K$ are locally finite.
	We skipped the proof for that $\widetilde{a}$ is finite by shrinking $I$ if necessary.
	Consider the definable set
	\begin{align*}
		X&=\{(a,x,\alpha,\beta) \in I^4\;|\;\alpha < \beta,\ a \prec x,\ \alpha < a < \beta,\ \alpha < x < \beta,\\
		&\quad \nexists y \ \alpha < y < \beta \text{ and }\ a \prec y \prec x\}\text{.}
	\end{align*}
	Taking a sufficiently small subinterval $I'$ of $I$, the definable set $X_{(a,\alpha,\beta)} \cap I'$ is a union of a finite set and an open set for any $a, \alpha', \beta' \in I'$, where $X_{(a,\alpha,\beta)}= \{x \in M\;|\; (a,x,\alpha,\beta) \in X\}$.
	Consequently, $X_{(a,\alpha,\beta)} \cap I'$ consists of finite points because $\widetilde{a}$ does not contain an open interval.
	Take $\alpha, \beta \in I'$ with $\alpha < \beta$ and set $I=(\alpha,\beta)$.
	Then, $\widetilde{a}$ is a finite set for any $a \in I$.

	By the equality (\ref{eq:5aaa}), the family $\{ \widetilde{a} \}_{a \in I}$ is infinite because $I \setminus K$ is infinite and the set $\widetilde{a}$ is finite for any $a \in I$.
	
	We define a definable relation $E$ on $I \setminus K$ by
	\begin{equation*}
		E(a,b) \Leftrightarrow \mathcal M \models \exists c\ ( a \in \widetilde{c} \wedge b \in \widetilde{c})\text{.}
	\end{equation*}
	It is an equivalence relation by the equality (\ref{eq:5aaa}).
	Set 
	\begin{equation*}
		Y=\{x \in I \setminus K\;|\;\  \mathcal M \models \forall y \in I \setminus K \ (E(x,y) \rightarrow x \leq y)\}\text{.}
	\end{equation*}
	We showed that $Y$ is an infinite set and reduced to the case in which $Y$ is a finite union of points and open convex sets.
	We can lead to a contradiction as follows:
	
	Let $Z$ be the largest convex subset of $Y$ with $t <Z$ for all $t \in Y \setminus Z$.
	In other word, $Z$ is the rightmost convex set in $Y$.
	Take $a \in Z$ and $b_1, b_2 \in U_a$ with $b_1 < a < b_2$.
	We have $a \prec b_1$ and $a \prec b_2$ by (iii).
	Since $C_{b_1}$ and $C_{b_2}$ are finite sets, there are $b_1'$ and $b_2'$ with $b_1' \in \widetilde{a}$, $b_2' \in \widetilde{a}$, $b_1' \preceq b_1$ and $b_2' \preceq b_2$.
	$E(b_1',b_2')$ holds true.
	We get $b_1' <  a < b_2'$ by the relations (vi) and (vii) and the definition of $\widetilde{a}$. 
	We have $b_2' \not\in Y$ by the definition of $Y$.
	Take an element $c \in U_{b_2'}$ with $b_2' < c$.
	There exists a $d \in \widetilde{c}$ with $d \in Y$ because the smallest element of $\widetilde{c}$ is contained in the set $Y$.
	We obtain $b_2' < d$ by the claim (viii) because $b_2' \prec c \prec d$ and $b_2'<c$.
	Since $a <b_2' < d$, $a \in Z$, $d \in Z$ and $Z$ is convex, we finally get $b_2' \in Z \subseteq Y$. 
	It is a contradiction.
\end{proof}

\begin{proof}[Proof of Theorem \ref{thm:mono}]
	In the course of the proof of Claim 1 in Theorem \ref{thm:mono}, we did not give a proof for that $(X_n)_z$ is dense $Y_z$ for $z \in B$.
	We give a proof for this part here.
	Recall that we constructed a definable set $Y$ so that the function $f_z$ is finite-to-one on the fiber $Y_z$ for any $z \in B$.

	Set $X_n=\{(x,z) \in Y\;|\; \exists x_1 < x, \exists x_2>x, f_z \text{ is injective on the interval }(x_1,x_2)\}$.
	We show that $(X_n)_z$ is dense in $Y_z$ for any $z \in B$.
	Fix arbitrary points $z \in B$ and $x' \in Y_z$. 
	There exists $x_2 \in Y_z$ such that the open interval $I_z=(x',x_2)$ is contained in $Y_z$.
	Consider the definable map $g_z:f_z(I_z) \rightarrow I_z$ given by $g_z(y)=\inf\{x \in I_z\;|\; f(x)=y\}$.
	We can take the infimum without assuming that $\mathcal M$ is definably complete because the set $f_z$ is finite-to-one on $Y_z$.
	The image of $g_z$ contains an open interval one of whose endpoints is $x'$.
	Otherwise, there exists an open interval $(x', x_2') \subseteq I_z$ with $(x',x_2') \cap g_z(f_z(I_z)) = \emptyset$.
	Take a point $u_1 \in (x', x_2')$.
	There exists a point $u_2 \in (x', x_2')$ with $u_2 < u_1$ and $f(u_1)=f(u_2)$ by the definition of $g_z$.
	We can take a point $u_3 \in (x', x_2')$ such that $u_3 < u_2$ and $f(u_2)=f(u_3)$ similarly.
	We can get infinite number of points in $I_z$ at which the value of $f_z$ is identical in this way. 
	Contradiction to the fact that $f_z$ is finite-to-one.
	Shrinking $I_z$ if necessary, $f_z$ is injective on the open interval $I_z$.
	We have shown that $(X_n)_z$ is dense in $Y_z$. 
\end{proof}

\begin{proof}[Proof of Proposition \ref{prop:continuity2}]
	We use the following technical definition in the proof.
	\begin{definition}
		Consider an expansion of a dense linear order without endpoints $\mathcal M=(M,<,\ldots)$.
		Let $A$ be a definable subset of $M^m$ and $f:A \rightarrow M$ be a definable function.
		Let $1 \leq i \leq m$.
		The function $f$ is \textit{$i$-constant} if, for any $a_1, \ldots, a_{i-1},a_{i+1},\ldots, a_n \in M$, the univariate function $f(a_1,\ldots, a_{i-1}, x, a_{i+1},\ldots, a_n)$ is constant.
		We define that the function is \textit{$i$-strictly increasing} and \textit{$i$-strictly decreasing} in the same way.
		The function is \textit{$i$-strictly monotone} if it is $i$-constant, $i$-strictly increasing or $i$-strictly decreasing.
		The function $f$ is \textit{$i$-continuous} if, for any $a_1, \ldots, a_{i-1},a_{i+1},\ldots, a_n \in M$, the univariate function $f(a_1,\ldots, a_{i-1}, x, a_{i+1},\ldots, a_n)$ is continuous.
	\end{definition}
	
	Let $f:B \to \overline{M}$ be a definable function defined on an open box $B$.
	We define an open interval $I_1$ and an open box $B_1$ so that $B=I_1 \times B_1$.
	Set 
	\begin{align*}
		X_+ &= \{(x,x') \in I_1 \times B_1\;|\; \text{the univariate function } f(\cdot, x') \text{ is }\\
		& \quad \text{strictly increasing and continuous on a neighborhood of }x\}\text{,}\\
		X_- &= \{(x,x') \in I_1 \times B_1\;|\; \text{the univariate function } f(\cdot, x') \text{ is }\\
		& \quad \text{strictly decreasing and continuous on a neighborhood of }x\}\text{,}\\
		X_c &= \{(x,x') \in I_1 \times B_1\;|\; \text{the univariate function } f(\cdot, x') \text{ is }\\
		& \quad \text{constant on a neighborhood of }x\}\text{ and }\\
		X_p &= B \setminus (X_+ \cup X_- \cup X_c) \text{.}
	\end{align*}
	The fibers $(X_p)_x$ are discrete for all $x \in B_1$ by Corollary \ref{cor:monotone-star}.
	In particular, $X_p$ has an empty interior.
	Recall that property (b) for $n-1$ holds by the proof in Remark \ref{rem:cccc}.
	At least one of $X_+$, $X_-$ and $X_c$ has a nonempty interior by property (b).
	Therefore, we may assume that $f$ is $1$-strictly monotone and $1$-continuous by considering an open box contained in one of them instead of $B$.
	Applying the same argument $(m-1)$-times, we may assume that $f$ is $i$-strictly monotone and $i$-continuous for all $1 \leq i \leq m$.
	The function $f$ is continuous on $B$ by \cite[Lemma 3.2.16]{vdD}.
\end{proof}

\begin{proof}[Proof of Proposition \ref{prop:dim_basic}(3)]
	Assume that $X$ and $Y$ are definable subsets of $M^m$ and $M^n$, respectively.
	Set $d=\dim(X)$, $e=\dim(Y)$ and $f=\dim(X \times Y)$.
	We first show that $d+e \leq f$.
	In fact, let $\pi:M^m \rightarrow M^d$ and $\rho:M^n \rightarrow M^e$ be coordinate projections such that both $\pi(X)$ and $\rho(Y)$ have nonempty interiors.
	The definable set $(\pi \times \rho)(X \times Y)$ has a nonempty interior.
	Therefore, we have $d+e \leq f$.
	We show the opposite inequality.
	Let $\Pi:M^{m+n} \rightarrow M^f$ be a coordinate projection with $\myint(\Pi(X \times Y)) \not=\emptyset$.
	There exist coordinate projections $\pi_1:M^m \rightarrow M^{f_1}$ and $\pi_2:M^n \rightarrow M^{f_2}$ with $\Pi=\pi_1 \times \pi_2$.
	In particular, we get $f=f_1+f_2$.
	Since $\Pi(X \times Y)$ has a nonempty interior, there exist open boxes $C \subseteq M^{f_1}$ and $D \subseteq M^{f_2}$ with $C \times D \subseteq \Pi(X \times Y)$.
	We get $C \subseteq \pi_1(X)$ and $D \subseteq \pi_2(Y)$.
	Hence, we have $d \geq f_1$ and $e \geq f_2$.
	We finally obtain $d+e \geq f_1+f_2 =f$.
\end{proof}

\begin{proof}[Proof of Lemma \ref{lem:key0}]
	Let $\mathcal M=(M,<,\ldots)$ be the structure in consideration.
	Let $X$ be a nonempty discrete definable subset of $M^n$.
	Let $\pi_k:M^n \rightarrow M$ be the coordinate projection onto the $k$-th coordinate for all $1 \leq k \leq n$.
	The images $\pi_k(X)$ are discrete by property (A).
	They are closed by Proposition \ref{prop:char_lvisceral}.
	Let $x$ be an accumulation point of $X$.
	We have $\pi_k(x) \in \pi_k(X)$ for all $1 \leq k \leq n$ because $\pi_k(x)$ are accumulation points of $\pi_k(X)$ and $\pi_k(X)$ are closed.
	We can take open intervals $I_k$ so that $\pi_k(X) \cap I_k =\{\pi_k(x)\}$ because $\pi_k(X)$ are discrete.
	It implies that $X \cap (I_1 \times \cdots \times I_n)$ consists of at most one point $x$ because $X \cap (I_1 \times \cdots \times I_n) \subseteq \prod_{k=1}^n \pi_k(X) \cap I_k = \{x\}$.
	It means that $x \in X$ because $x$ is an accumulation point of $X$.
\end{proof}

\begin{proof}[Proof of Lemma \ref{lem:aaa}]
	We first reduce to the case in which $f$ is the restriction of a coordinate projection.
	Let $X$ be a definable subset of $M^n$ and $\pi:M^{n+1} \rightarrow M$ be the coordinate projection onto the last coordinate.
	Consider the graph $\Gamma(f)$ of the definable map $f$.
	The image $\pi(\Gamma(f))=f(X)$ and all the fibers $\Gamma(f) \cap \pi^{-1}(x)$ are discrete by the assumption.
	If the graph $\Gamma(f)$ is discrete, the definable set $X$ is also discrete by property (A) because $X$ is the projection image of the discrete set  $\Gamma(f)$.
	We have reduced to the case in which $f$ is the restriction of the coordinate projection onto the last coordinate $\pi:M^{n+1} \rightarrow M$ to a definable subset $X$ of $M^{n+1}$.
	
	Take an arbitrary point $x \in X$.
	Since $\pi(X)$ is discrete by the assumption, we can take an open interval $I$ containing the point $\pi(x)$ such that $\pi(X) \cap I$ is a singleton.
	Since the fiber $\pi^{-1}(\pi(x)) \cap X$ is discrete, there exists an open box $B$ containing the point $x$ such that $X \cap (B \times \{\pi(x)\})$ is a singleton.
	The open box $B \times I$ contains the point $x$ and the intersection of $X$ with $B \times I$ is a singleton.
	We have demonstrated that $X$ is discrete.
\end{proof}

\begin{proof}[Proof of Proposition \ref{prop:zero}]
	Let $X$ be a definable subset of $M^n$.
	The definable set $X$ is discrete if and only if the projection image $\pi(X)$ has an empty interior for all the coordinate projections $\pi:M^n \rightarrow M$ by the property (a).
	Therefore, $X$ is discrete if and only if $\dim X=0$.
	A discrete definable set is always closed by Lemma \ref{lem:key0}.
\end{proof}

\begin{proof}[Proof of Lemma \ref{lem:pre0}]
Permuting the coordinates if necessary, we may assume that $\pi$ is the projection onto the first $d$ coordinates.
Set $$S=\{x \in \pi(X)\;|\; \text{the fiber }X \cap \pi^{-1}(x) \text{ is not discrete}\}\text{.}$$
We have $S=\{x \in \pi(X)\;|\; \dim(X \cap \pi^{-1}(x))>0\}$ by Proposition \ref{prop:zero}.
We want to show that $S$ has an empty interior.
Assume the contrary.
Let $\rho_j:M^n \rightarrow M$ be the coordinate projections onto the $j$-th coordinate for all $d<j \leq n$.
Set $$S_j=\{x \in \pi(X)\;|\; \rho_j(X \cap \pi^{-1}(x)) \text{ contains an open interval}\}\text{.}$$
We have $S=\bigcup_{d<j \leq n}S_j$ by the definition of dimension.
The definable set $S_j$ has a nonempty interior by Lemma \ref{lem:l-b} for some $d<j \leq n$.
Fix such $j$.
Let $\Pi:M^n \rightarrow M^{d+1}$ be the coordinate projection given by $\Pi(x)=(\pi(x),\rho_j(x))$.
The definable set $T=\{x \in M^d \;|\; \text{the fiber }(\Pi(X))_x \text{ contains an open interval}\}$ contains $S_j$ and it has a nonempty interior.
Therefore, the projection image $\Pi(X)$ has a nonempty interior by property (B).
It contradicts the assumption that $\dim(X)=d$.
We have shown that $S$ has an empty interior.
Since $\pi(X)$ has a nonempty interior, there exists a definable open subset $U$ of $\pi(X)$ with $U \cap S=\emptyset$ by Lemma \ref{lem:l-b}.
\end{proof}

\begin{proof}[Proof of Theorem \ref{thm:dimension_basic}(1)]
	The inequality $\dim(X \cup Y) \geq \max\{\dim(X),\dim(Y)\}$ is obvious by Proposition \ref{prop:dim_basic}(1).
	We show the opposite inequality.
	Set $d=\dim(X \cup Y)$.
	There exists a coordinate projection $\pi:M^n \rightarrow M^d$ such that $\pi(X \cup Y)$ has a nonempty interior by the definition of dimension.
	At least one of $\pi(X)$ and $\pi(Y)$ has a nonempty interior by Lemma \ref{lem:l-b} because $\pi(X \cup Y)=\pi(X) \cup \pi(Y)$.
	We may assume that $\pi(X)$ has a nonempty interior without loss of generality.
	We have $d \leq \dim(X)$ by the definition of dimension.
	We have demonstrated that $\dim(X \cup Y) \leq \max\{\dim(X),\dim(Y)\}$.
\end{proof}

\begin{proof}[Proof of Remark \ref{rem:cccc}]
We prove that property (a) holds for a locally l-visceral structure enjoying the univariate continuity property $\mathcal M=(M,<,\ldots)$.

Let $X$ be a discrete definable subset of $M^n$.
Let $\pi:M^n \rightarrow M^d$ be a coordinate projection.
We prove that $\pi(X)$ is discrete.
We can reduce to the case in which $d=1$ in the same manner as the proof of Proposition \ref{prop:suff_visceral}.

When $d=1$, there exists a definable map $\tau:\pi(X) \rightarrow X$ such that the composition $\pi \circ \tau$ is an identity map by the property (d).
If $\pi(X)$ is not discrete, it contains an open interval $I$ because of local o-minimality.
Shrinking the interval $I$ if necessary, the restriction of $\tau$ to $I$ is continuous by the univariate continuity property.
It means that $X$ contains the graph of a continuous map defined on an open interval.
It contradicts the assumption that $X$ is discrete.
\end{proof}

\begin{proof}[Proof of Theorem \ref{thm:dim}]
	We first prove assertion (1).
	Let $X$ be a definable subset of $M^m$.
	We lead to a contradiction assuming that $d=\dim X=\dim \mathcal D(f)$.
	By Lemma \ref{lem:pre0} and Lemma \ref{lem:pre1}, there exist a coordinate projection $\pi:M^m \rightarrow M^d$, definable open subsets $V \subseteq \pi(\mathcal D(f))$ and $W \subseteq M^m$ and a definable continuous function $g:V \rightarrow \mathcal D(f)$ such that $\pi(W)=V$, $X \cap W=g(V)$ and $\pi \circ g$ is the identity map on $V$.
	Shrinking $V$ and replacing $W$ with $W \cap \pi^{-1}(V)$ if necessary, we may assume that $f \circ g$ is continuous by the $*$-continuity property.
	Since $g$ is a definable homeomorphism onto its image, the function $f$ is continuous on $g(V)= X \cap W$.
	On the other hand, $f$ is discontinuous everywhere on $X \cap W$ because $X \cap W$ is open in $X$ and $X \cap W =g(V)$ is contained in $\mathcal D(f)$.
	Contradiction.
	
	We next prove assertion (2).
	Take distinct elements $c,d \in M$.
	Consider the definable function $f:\mycl(X) \rightarrow M$ given by
	\[
	f(x)= \left\{
	\begin{array}{ll}
		c & \text{if } x \in X \text{ and }\\
		d & \text{otherwise.} 
	\end{array}
	\right.
	\]
	It is obvious that $\mathcal D(f)$ contains $\partial X$.
	Assertion (2) follows from assertion (1) and Proposition \ref{prop:dim_basic}(1). 
	 
	Our next task is to prove assertion (3). 
	Let $d'$ be the maximum of nonnegative integers $e$ satisfying the condition given in the assertion.
	We first demonstrate $d' \leq d$.
	In fact, let $B$ be an open box contained in $M^{d'}$ and $\varphi:B \rightarrow X$ be a definable injective continuous map homeomorphic onto its image.
	We have $\dim \varphi(B)=\dim B=d'$ by Theorem \ref{thm:dimension_basic}(3).
	We get $d = \dim X \geq \dim(\varphi(B)) = d'$ by Proposition \ref{prop:dim_basic}(1).
	
	We next demonstrate $d \leq d'$.
	Applying Lemma \ref{lem:pre0} and Lemma \ref{lem:pre1} to the definable set $X$, we can get a coordinate projection $\pi:M^n \rightarrow M^d$, a definable open box $U$ in $\pi(X)$ and a definable continuous map $\tau:U \rightarrow X$ such that $\pi \circ \tau$ is the identity map on $U$.
	In particular, $\tau$ is a definable continuous injective map homeomorphic onto its image.
	Therefore, we have $d \leq d'$ by the definition of $d'$.  
	
	The final task is to prove assertion (4).
	Set $d=\dim(X)$.
	There exists an open box $U$ in $M^d$ and a definable continuous injective map $\varphi:U \rightarrow X$ homeomorphic onto its image by assertion (3).
	Take an arbitrary point $t \in U$ and set $x=\varphi(t)$.
	For any open box $B$ containing the point $x$, the inverse image $\varphi^{-1}(B)$ is a definable open set.
	Take a open box $V$ with $t \in V \subseteq \varphi^{-1}(B)$.
	The restriction $\varphi|_{V}: V \rightarrow X \cap B$ is a definable continuous injective map homeomorphic onto its image.
	Hence, we have $\dim(X \cap B) \geq d$ by Theorem \ref{thm:dim}(3)
	The opposite inequality follows from Proposition \ref{prop:dim_basic}(1).
\end{proof}

\begin{proof}[Proof of Theorem \ref{thm:decomposition}]
	We prove the theorem step by step.
	We first show the following claim:
	\medskip
	
	\textbf{Claim 1.} Let $X$ be a definable subset of $M^n$.
	There exists a family $\{C_i\}_{i=1}^N$ of mutually disjoint normal submanifolds with $X= \bigcup_{i=1}^N C_i$ and $N \leq 2^{n}$.
	\begin{proof}[Proof of Claim 1]
		We first define the full dimension of a definable subset $X$ of $M^n$.
		Set $d=\dim X$.
		The notation $\Pi_{n,d}$ denotes the set of all the coordinate projections of $M^n$ onto $M^d$.
		The set $\Pi_{n,d}$ is a finite set. 
		The \textit{full dimension} $\myfdim(X)$ is $(d,e)$ by definition if $d=\dim(X)$ and $e$ is the number of elements in $\Pi_{n,d}$ under which the projection image of $X$ has a nonempty interior.
		The pairs $(d,e)$ are ordered by the lexicographic order.
		
		We prove the the claim by induction on $\myfdim(X)$.
		When $\dim(X)=0$, $X$ is closed and discrete by Proposition \ref{prop:zero}.
		The definable set $X$ is obviously a normal submanifold in this case.
		
		We consider the case in which $\dim(X) >0$.
		Set $(d,e)=\myfdim(X)$.
		Take a coordinate projection $\pi:M^n \rightarrow M^d$ such that $\pi(X)$ has a nonempty interior.
		Set $G=\{x \in X\;|\; x \text{ is } (X,\pi) \text{-normal}\}$ and $B=X \setminus G$.
		It is obvious that any point $x \in G$ is $(G,\pi)$-normal.
		
		We demonstrate that $\pi(B)$ has an empty interior.
		Assume the contrary.
		There exists an open box $U$ such that the fibers $B_x=\pi^{-1}(x) \cap B$ are discrete for all $x \in U$ by Lemma \ref{lem:pre0}.
		We can take a definable map $\tau:U \rightarrow B$ with $\pi(\tau(x))=x$ for all $x \in U$ by Lemma \ref{lem:*-d}.
		The dimension of points at which the map $\tau$ is discontinuous is of dimension smaller than $d$ by Theorem \ref{thm:dim}(1).
		We may assume that the restriction of $\tau$ to $U$ is continuous shrinking $U$ if necessary.
		
		Set $Z=\partial (X \setminus \tau(U))$.
		We get $\dim Z = \dim \partial (X \setminus \tau(U)) < \dim (X \setminus \tau(U)) \leq \dim X=d$ by Proposition \ref{prop:dim_basic}(1) and Theorem \ref{thm:dim}(2).
		We have $\dim \mycl(Z) = \dim Z <d$ again by Theorem \ref{thm:dimension_basic}(1) and Theorem \ref{thm:dim}(2). 
		On the other hand, we have $d=\dim U =\dim \pi(\tau(U)) \leq \dim \tau(U) \leq \dim X =d$ by Proposition \ref{prop:dim_basic}(1) and Theorem \ref{thm:dim}(3). 
		We get $\dim(\tau(U))=d$.
		It means that $\tau(U) \not\subseteq \mycl(Z)$ by Proposition \ref{prop:dim_basic}(1).
		
		Take a point $p$ in $\tau(U) \setminus \mycl(Z)$.
		Take a sufficiently small open box $V$ containing the point $p$. 
		We have $X \cap V=\tau(U) \cap V$ by the definition of $Z$ and $p$.
		Since the restriction of $\tau$ to $U$ is continuous, there exists an open box $U'$ contained in $U \cap \tau^{-1}(V)$.
		Consider the open box $V'=V \cap \pi^{-1}(U')$.
		It is obvious that $X \cap V'=\tau(U) \cap V'$ is the graph of the restriction of $\tau$ to $U'$ by the definition.
		Any point $\tau(U) \cap V'$ is $(X,\pi)$-normal, but it contradicts to the definition of $B$ and the inclusion $\tau(U) \subseteq B$.
		We have shown that $\pi(B)$ has an empty interior.
		In particular, we get $\myfdim(B) < \myfdim(X)$.
		
		There exists a decomposition $B=C_1 \cup \ldots \cup C_k$ of $B$ satisfying the conditions in Claim 1 by the induction hypothesis.
		The decomposition $X=G \cup C_1 \cup \ldots \cup C_k$ is a desired decomposition of $X$.
		
		It is obvious that the number of normal submanifolds $N$ is not greater than $$\sum_{d=0}^n (\text{the cardinality of }\Pi_{n,d})=\sum_{d=0}^n \left(\begin{array}{c}n\\d\end{array}\right)=2^n\text{.}$$
	\end{proof}

	\textbf{Claim 2.} Let $\{X_i\}_{i=1}^m$ be a finite family of definable subsets of $M^n$.
	There exists a decomposition $\{C_i\}_{i=1}^N$ of $M^n$ into normal submanifolds partitioning $\{X_i\}_{i=1}^m$ with $N \leq 2^{m+n}$.
	\begin{proof}[Proof of Claim 2]
	Set $X_i^0=X_i$ and $X_i^1=M^n \setminus X_i$ for all $1 \leq i \leq m$.
	For any $\sigma \in \{0,1\}^m$, the notation $\sigma(i)$ denotes the $i$-th component of $\sigma$.
	Set $X_\sigma = \bigcap_{i=1}^m X_i^{\sigma(i)}$ for any $\sigma \in \{0,1\}^m$.
	The family $\{X_\sigma\}_{\sigma \in \{0,1\}^m}$ is mutually disjoint and satisfies the equality $M^n=\bigcup_{\sigma \in \{0,1\}^m} X_\sigma$.
	For all $\sigma \in \{0,1\}^m$, there exist families $\{C_{\sigma,j}\}_{j=1}^{N_\sigma}$ of mutually disjoint normal submanifolds with $X_\sigma= \bigcup_{j=1}^{N_\sigma} C_{\sigma,j}$ and $N_\sigma \leq 2^n$ by Claim 1.
	The family $\bigcup_{\sigma  \in \{0,1\}^m} \{C_{\sigma,j}\}_{j=1}^{N_\sigma}$ gives the decomposition we are looking for.
	\end{proof}

	We are now ready to prove the theorem.
	By reverse induction on $d$, we construct a decomposition $\{C_{\lambda}\}_{\lambda \in \Lambda_d}$ of $M^n$ into normal submanifolds partitioning $\{X_i\}_{i=1}^m$ such that the closures of all the normal submanifolds of dimension not smaller than $d$ are the unions of subfamilies of the decomposition. 
	
	When $d=n$, take a decomposition $\{D_\lambda\}_{\lambda \in \Lambda}$ of $M^n$ into normal submanifolds partitioning $\{X_i\}_{i=1}^m$ by Claim 2.
	Set $\Lambda_n'=\{\lambda \in \Lambda\;|\;\dim(D_{\lambda})=n\}$.
	Get a decomposition $\{E_\lambda\}_{\lambda \in \widetilde{\Lambda_n}}$ of $M^n$ into normal submanifolds partitioning the family $\{D_\lambda\}_{\lambda \in \Lambda} \cup \{\mycl({D_{\lambda}})\setminus D_{\lambda}\}_{\lambda \in \Lambda_n'}$.
	Consider the set 
	$$
	\widetilde{\Lambda_n}'=\{\lambda \in \widetilde{\Lambda_n} \;|\; E_\lambda \text{ is not contained in any }D_{\lambda'}\text{ with }\lambda' \in \Lambda_n'\}\text{.}
	$$
	We always have $\dim(E_\lambda)<n$ for all $\lambda \in \widetilde{\Lambda_n}'$ by Theorem \ref{thm:dim}(2).
	Hence, the family $\{D_\lambda\}_{\lambda \in \Lambda_n'} \cup \{E_\lambda\}_{\lambda \in \widetilde{\Lambda_n}'}$
	is trivially a decomposition of $M^n$ into normal submanifolds partitioning $\{X_i\}_{i=1}^m$ we are looking for.
	
	We next consider the case in which $d < n$.
	Let $\{D_{\lambda}\}_{\lambda \in \Lambda_{d+1}}$ be a decomposition of $M^n$ into normal submanifolds partitioning $\{X_i\}_{i=1}^m$ such that the closures of all the normal submanifolds of dimension not smaller than $d+1$ are the unions of subfamilies of the decomposition. 
	It exists by the induction hypothesis.
	Set $\Lambda_d'=\{\lambda \in \Lambda_{d+1}\;|\;\dim(D_{\lambda})=d\}$ and $\Lambda_d''=\{\lambda \in \Lambda_{d+1}\;|\;\dim(D_{\lambda}) \geq d\}$.
	Get a decomposition $\{E^d_\lambda\}_{\lambda \in \widetilde{\Lambda_d}}$ of $M^n$ into normal submanifolds partitioning the family $\{D_\lambda\}_{\lambda \in \Lambda_{d+1}} \cup \{\mycl({D_{\lambda}})\setminus D_{\lambda}\}_{\lambda \in \Lambda_d'}$.
	Set 
	$\widetilde{\Lambda_d}'=\{\lambda \in \widetilde{\Lambda_d} \;|\; E_\lambda \text{ is not contained in any }D_{\lambda'}\text{ with }\lambda' \in \Lambda_d''\}\text{.}
	$
	The family $\{D_\lambda\}_{\lambda \in \Lambda_d''} \cup \{E_\lambda\}_{\lambda \in \widetilde{\Lambda_d}'}$
	is a decomposition of $M^n$ into normal submanifolds partitioning $\{X_i\}_{i=1}^m$ we want to construct.
	
	The `furthermore' part of the theorem is obvious from the proof.
\end{proof}

\begin{proof}[Proof of Theorem \ref{thm:pregeometry}]
The conditions (i) and (iv) in Definition \ref{def:pregeometry} are obviously satisfied.

We show that that the condition (ii) in Definition \ref{def:pregeometry} is satisfied.
The inclusion $\mydcl(A) \subseteq \mydcl(\mydcl(A))$ is obvious by Definition \ref{def:pregeometry}(i).
We demonstrate the opposite inclusion.
Take $b \in \mydcl(\mydcl(A))$.
There exists a tuple $a=(a_1,\ldots, a_m)$ of elements in $\mydcl(A)$ and an $\mathcal L(A)$-formula $\phi(x,y_1,\ldots, y_m)$ such that $\mathcal M \models \phi(b,a_1,\ldots, a_m)$ and the definable set $\{x \in M\;|\; \mathcal M \models \phi(x,a_1,\ldots, a_m)\}$ is discrete and closed.
Since $a_i \in \mydcl(A)$, there exists an $\mathcal L(A)$-formula $\psi_i(x)$ such that $\mathcal M \models \psi_i(a_i)$ and the definable set $\{x \in M\;|\; \mathcal M \models \psi_i(x)\}$ is discrete and closed for any $1 \leq i \leq m$.
Consider the $\mathcal L(A)$-formula $\eta(y_1,\ldots, y_m)$ representing that the set $$E_{y_1,\ldots, y_m}:=\{x \in M\;|\; \mathcal M \models \phi(x,y_1,\ldots, y_m)\}$$
is neither empty nor contains an open interval.
By Proposition \ref{prop:zero}, we have $\mathcal M \models \eta(y_1,\ldots, y_m)$ if and only if the definable set $E_{y_1,\ldots, y_m}$ is nonempty, discrete and closed.
Consider the definable sets 
\begin{align*}
	&F=\{(y_1,\ldots, y_m) \in M^m\;|\; \mathcal M \models \bigwedge_{i=1}^m \psi_i(y_i) \wedge \eta(y_1,\ldots, y_m)\} \text{ and }\\
	&G =\{(x,y_1,\ldots,y_m) \in M^{m+1}\;|\; \mathcal M \models \phi(x,y_1,\ldots, y_m)\ \text{ and } (y_1,\ldots, y_m) \in F\}.
\end{align*}
By the assumption, the set $F$ is discrete and closed.
The set $F$ is the image of $G$ under the coordinate projection forgetting the first coordinate and its fibers are discrete and closed by the definition of the formula $\eta$.
By Proposition \ref{prop:zero} and Theorem \ref{thm:dimension_basic}(2), $G$ is also discrete and closed.
The projection image $H$ of $G$ defined by
$$H=\{x \in M\;|\; \mathcal M \models \exists y_1, \ldots, y_m\ (x,y_1,\ldots, y_m) \in G\}$$
is also discrete and closed by Proposition \ref{prop:zero} and Theorem \ref{thm:dimension_basic}(3).
The definable set $H$ is defined by an $\mathcal L(A)$-formula by its construction, and we have $b \in H$.
This implies that $b \in \mydcl(A)$.
We have demonstrated that the condition (ii) is satisfied.

The remaining task is to demonstrate the exchange property (iii).
Fix elements $a,b \in M$ and a subset $A \subseteq M$ with $a \in \mydcl(A \cup \{b\}) \setminus \mydcl(A)$.
We demonstrate that $b \in \mydcl(A \cup \{a\})$.
Take an $\mathcal L(A)$-formula $\varphi(x,y)$ such that $D_b$ is discrete and closed, and it contains the point $a$.
Here, the notation $D_y$ denotes the set $\{x \in M\;|\; \mathcal M \models \varphi(x,y)\}$ for any $y \in M$.
Consider the definable sets
\begin{align*}
	& S= \{y \in M\;|\; \mathcal M \models \exists x\  \varphi(x,y)\}\text{ and }\\
	& T=\{y \in S\;|\; \text{ the set } D_y \text{ is discrete and closed}\}.
\end{align*}
We get $b \in T$ by the assumption.

We first consider the case in which there exist no open intervals contained in $T$ and containing the point $b$.
By Proposition \ref{prop:local_ominmal} and Remark \ref{rem:local_omin}, the point $b$ is either an isolated point of $T$ or a point in the boundary of $T$.
Consider the set of isolated points of $T$ and the boundary in $T$.
It is obviously defined by an $\mathcal L(A)$-formula and contains the point $b$.
It does not contain an open interval by its definition.
Therefore, it is discrete and closed by Proposition \ref{prop:zero}.
It implies that $b \in \mydcl(A) \subseteq \mydcl(A \cup \{a\})$.

We next treat the case in which there exists an open interval $I$ contained in $T$ and containing the point $b$.
We may assume that $S$ is open, $S=T$ and the sets $D_y$ are discrete and closed for all $y \in S$ by considering the $\mathcal L(A)$-formula $\varphi(x,y) \wedge (y \in \myint(T))$ instead of $\varphi(x,y)$.
Let $\pi_i:M^2 \rightarrow M$ be the coordinate projections onto the $i$-th coordinates for $i=1,2$.
Set 
\begin{align*}
	&\Gamma = \{(y,x) \in M^2\;|\; \mathcal M \models \varphi(x,y)\},\\
	& U=\{(y,x) \in \Gamma\;|\; \text{ there exists an open box }B \subseteq M^2 \text{ such that } (y,x) \in B \text{ and }\\
	&\qquad B \cap \Gamma \text{ is the graph of a monotone definable continuous function}\\
	&\qquad\text{defined on }\pi_1(B)\} \text{ and }\\
	&V = \Gamma \setminus U.
\end{align*}
We have $\dim \Gamma =1$ by Proposition \ref{prop:zero}, Theorem \ref{thm:dimension_basic}(2) and the assumption that $S=T$.
Note that the fiber $\Gamma_y=\{x \in M\;|\; (y,x) \in \Gamma\}$ is discrete for any $y \in \pi_1(\Gamma)$ by the above assumption.
We prove that $\pi_1(V)$ is discrete and closed.
By Proposition \ref{prop:zero} and Theorem \ref{thm:dimension_basic}(3), we have only to demonstrate that $\dim V \leq 0$.
We also consider the definable sets
\begin{align*}
	& U'=\{(y,x) \in \Gamma\;|\; \text{ there exists an open box }B \subseteq M^2 \text{ such that } (y,x) \in B \text{ and }\\
	&\qquad B \cap \Gamma \text{ is the graph of a (not necessarily monotone) definable}\\
	&\qquad\text{continuous function defined on }\pi_1(B)\} \text{,}\\
	&V' = \Gamma \setminus U' \text{ and }\\
	& W= V \setminus V'.
\end{align*}
It is obvious that $V=V' \cup W$.
The inequality $\dim \pi_1(V) \leq 0$ is equivalent to the condition that $\dim \pi_1(V') \leq 0$ and $\dim \pi_1(W) \leq 0$ by Theorem \ref{thm:dimension_basic}(1).

We demonstrate that $\pi_1(V')$ has an empty interior.
Assume the contrary.
There exists an open interval $B_1$ such that the fibers $V'_x=\pi_1^{-1}(x) \cap V'$ are discrete for all $x \in B_1$ by Lemma \ref{lem:pre0}.
We can take a definable map $\tau:B_1 \rightarrow V'$ with $\pi_1(\tau(x))=x$ for all $x \in B_1$ by Lemma \ref{lem:*-d}.
The dimension of points at which the map $\tau$ is discontinuous is of dimension smaller than one by Theorem \ref{thm:dim}(1).
We may assume that the restriction of $\tau$ to $B_1$ is continuous shrinking $B_1$ if necessary.

Set $Z=\partial (\Gamma \setminus \tau(B_1))$.
We get $\dim Z = \dim \partial (\Gamma \setminus \tau(B_1)) < \dim (\Gamma \setminus \tau(B_1)) \leq \dim \Gamma=1$ by Proposition \ref{prop:dim_basic}(1) and Theorem \ref{thm:dim}(2).
This implies that $Z$ is discrete and closed by Proposition \ref{prop:zero}.
On the other hand, we have $1=\dim B_1 =\dim \pi_1(\tau(B_1)) \leq \dim \tau(B_1) \leq \dim \Gamma =1$ by Proposition \ref{prop:dim_basic}(1) and Theorem \ref{thm:dim}(3). 
We get $\dim(\tau(B_1))=1$.
It means that $\tau(B_1) \not\subseteq Z$ by Proposition \ref{prop:dim_basic}(1).

Take a point $p$ in $\tau(B_1) \setminus Z$.
Take a sufficiently small open box $B_2$ containing the point $p$. 
We have $\Gamma \cap B_2=\tau(B_1) \cap B_2$.
Since the restriction of $\tau$ to $B_1$ is continuous, there exists an open interval $B_3$ contained in $B_1 \cap \tau^{-1}(B_2)$.
Consider the open box $B_4=B_2 \cap \pi^{-1}(B_3)$.
It is obvious that $X \cap B_4=\tau(B_3) \cap B_4$ is the graph of the restriction of $\tau$ to $B_3$ by the definition.
Any point $\tau(B_3) \cap B_4$ is $(\Gamma,\pi_1)$-normal, but it contradicts to the definition of $V'$ and the inclusion $\tau(B_3) \subseteq V'$.
We have shown that $\pi_1(V')$ has an empty interior.
This means the inequality $\dim \pi_1(V')\leq 0$.

The remaining task is to demonstrate the inequality $\dim \pi_1(W) \leq 0$.
Assume that $\dim \pi_1(W)=1$ for contradiction.
A nonempty interval $J$ is contained in $\pi_1(W)$.
Note that the fiber $W_y=\{x \in M\;|\; (y,x) \in W\}$ for any $y \in J$ is discrete because $W$ is a subset of $\Gamma$ and the fiber of $\Gamma$ is also discrete.
We can find a definable map $f: J \rightarrow W$ such that the composition $\pi_1 \circ f$ is the identity map on $J$ by Lemma \ref{lem:*-d}.
By Corollary \ref{cor:monotone-star}, we may assume that $g:=\pi_2 \circ f:J \rightarrow M$ is monotone and continuous by shrinking the interval $J$ if necessary.
We prove that one of the following conditions is satisfied by shrinking $J$ if necessary.
\begin{itemize}
	\item For any $s \in J$, there are no $t \in M$ with $(s,t) \in \Gamma$ and $t<g(s)$;
	\item There exists a definable continuous map $g_1:J \rightarrow M$ such that, for any $s \in J$, $(s,g_1(s)) \in \Gamma$, $g_1(s)<g(s)$ and there are no $t \in M$ with $(s,t) \in \Gamma$ and $g_1(s)<t<g(s)$.
\end{itemize}
Set $J'=\{s \in J\;|\; \neg \exists t,\ ((s,t) \in \Gamma) \wedge (t<g(s))\}$.
If $J'$ contains an open interval, the first condition is satisfied by replacing $J$ with the open interval contained in $J'$.
Otherwise, $J'$ is discrete and closed by Proposition \ref{prop:zero}.
We may assume that $J'=\emptyset$ by shrinking $J$ if necessary.
Consider the definable function $g_1:J \rightarrow M$ given by $g_1(s)=\sup\{t \in M\;|\; (s,t) \in \Gamma \text{ and } t<g(s)\}$.
This function is well-defined and we have $(s,g_1(s)) \in \Gamma$ by Lemma \ref{lem:supinf}.
By the univariate $*$-continuity property, we may assume that $g_1$ is continuous by shrinking $J$ once again.
The function $g_1$ obviously satisfies the second condition.

In the same manner, we can demonstrate that one of  the following conditions is satisfied by shrinking $J$ if necessary.
We omit the proof.
\begin{itemize}
	\item For any $s \in J$, there are no $t \in M$ with $(s,t) \in \Gamma$ and $t>g(s)$;
	\item There exists a definable continuous map $g_2:J \rightarrow M$ such that, for any $s \in J$, $(s,g_2(s)) \in \Gamma$, $g_2(s)>g(s)$ and there are no $t \in M$ with $(s,t) \in \Gamma$ and $g(s)<t<g_2(s)$.
\end{itemize}
Under this circumstance, the graph of $g$ is contained in both $U$ and $W$, which is a contradiction to the fact that $U \cap W=\emptyset$. 
We have demonstrated that $\dim \pi_1(V) \leq 0$.

When $b \in \pi_1(V)$, we have $b \in \mydcl(A) \subseteq \mydcl(A \cup \{a\})$.
We consider the other case, that is, the case in which $b \not\in \pi_1(V)$.
By replacing $\varphi(x,y)$ with the $\mathcal L(A)$-formula $\varphi(x,y) \wedge ((y,x) \notin V)$ if necessary, we may assume that $\Gamma=U$ without loss of generality.
Consider the set 
\begin{align*}
	& X=\{(y,x) \in \Gamma\;|\; \text{ there exists an open box }B \subseteq M^2 \text{ such that } (y,x) \in B \text{ and }\\
	&\qquad B \cap \Gamma \text{ is the graph of a constant function defined on }\pi_1(B)\}.
\end{align*}
It is defined by an $\mathcal L(A)$-formula.
If $\pi_2(X)$ is not discrete, we have $\dim \pi_2(X)=1$ by Proposition \ref{prop:zero}.
Therefore, we have $\dim \Gamma =2$ by Theorem \ref{thm:dimension_basic}(2).
On the other hand, we have $\dim \Gamma \leq 1$ by the same theorem because $\dim \pi_1(\Gamma) \leq 1$ and $\dim (\Gamma \cap \pi_1^{-1}(y))=0$ for all $y \in \pi_1(\Gamma)$.
A contradiction.
It implies that $\pi_2(X)$ is discrete and closed.

If the point $a$ is contained in $\pi_2(X)$, we have $a \in \mydcl(A)$, which is a contradiction to the assumption. 
We obtain $a \not\in \pi_2(X)$.
Replacing $\Gamma$ with $\Gamma \setminus X$, we may assume that, for any $(y,x) \in \Gamma$, there exists an open box $B$ containing the point $(y,x)$ such that $B \cap \Gamma$ is the graph of a strictly monotone continuous definable function defined on $\pi_1(B)$.
The set $\pi_1(\Gamma \cap \pi_2^{-1}(a))$ is defined by an $\mathcal L(A \cup \{a\})$-formula. 
By the above assumption, it does not contain an open interval.
It is discrete and closed by Proposition \ref{prop:zero}.
On the other hand, we have $(b,a) \in \Gamma$.
It implies that $b \in \pi_1(\Gamma \cap \pi_2^{-1}(a))$.
We finally get $b \in \mydcl(A \cup \{a\})$.
\end{proof}

\begin{proof}[Proof of Theorem \ref{thm:rank}]
	We prepare two lemmas for our proof of the theorem.
	\begin{lemma}\label{lem:infetesimal}
		Let $\mathcal M=(M,<,\ldots)$ be a model of $T$.
		Let $(a_1,\ldots, a_n) \in M^n$.
		There exist a sequence of elementary extensions $$\mathcal M_0=\mathcal M \preceq \mathcal M_1 \preceq \cdots \preceq \mathcal M_n$$ and elements $b_i \in M_i$ for all $1 \leq i \leq n$ such that $b_i>a_i$ and the inequalities $t > b_i$ hold for all $t \in M_{i-1}$ with $t>a_i$.
		Here, $M_i$ denotes the universe of $\mathcal M_i$ for all $1 \leq i \leq n$.
	\end{lemma}
	\begin{proof}
				We can easily reduce to the case in which $n=1$ by induction on $n$.
				Therefore, we only consider the case in which $n=1$.
				Consider the set $a_1^+$ of $\mathcal L(M_0)$-formulas $\Phi(x)$ with a single free variable $x$ satisfying $$\mathcal M \models \exists s\ s>a_1 \wedge (\forall t\  a_1<t<s \rightarrow \Phi(t)).$$
				Since $\mathcal M$ is locally o-minimal by Proposition \ref{prop:local_ominmal}, it is a complete $1$-type.
				There exist an elementary extension $\mathcal M_1=(M_1,\ldots)$ and an element $b_1 \in M_1$ realizing the $1$-type $a_1^+$.
				They satisfy the requirements of the lemma.
	\end{proof}
	
	\begin{lemma}\label{lem:infetesimal2}
		Let $\mathcal M=(M,<,\ldots)$ be a model of $T$ and $\mathcal N=(N,<,\ldots)$ be its elementary extension.
		Let $a \in M$ and $b \in N$ be points such that $a<b$ and, $t>b$ for all $t \in M$ with $t>a$.
		Then we have $b \not\in \mydcl_{\mathcal N}(M)$.
	\end{lemma}
	\begin{proof}
		Assume that $b \in \mydcl_{\mathcal N}(M)$ for contradiction.
		There exists a definable closed discrete subset $X$ of $M$ such that $b \in X^{\mathcal N}$.
		Here, $X^{\mathcal N}$ denotes the definable subset of $N$ defined by the same formula as $X$.
		Since $\mathcal M$ is locally o-minimal by Proposition \ref{prop:local_ominmal}, there exists an element $c \in M$ such that $c>a$ and the open interval $(a,c)$ does not intersect with $X$.
		In particular, for any $x \in (a,c)^{\mathcal N}$, we have $x \not\in X^{\mathcal N}$.
		It contradicts the condition that $b \in X^{\mathcal N}$ because the inequalities $a<b<c$ hold by the assumption.
	\end{proof}
	
	We now return to the proof of the theorem.
	
		Let $\mathbb M$ be a monster model of $T$.
	We use the same notation for its universe.
	We first reduce the theorem to a simple case.
	There exists a family $\{C_i\}_{i=1}^N$ of mutually disjoint normal submanifolds which are definable over $A$ such that $X= \bigcup_{i=1}^N C_i$ by Theorem \ref{thm:decomposition} and its proof.
	We obviously have $\myrk(X/A)=\max\{\myrk(C_i/A)\;|\; 1 \leq i \leq N\}$ by the definition of $\myrk$.
	On the other hand, we also have $\dim X = \max \{\dim C_i\;|\; 1 \leq i \leq N\}$ by Theorem \ref{thm:dimension_basic}(1).
	Hence, we may assume that $X$ is $\pi$-normal manifold without loss of generality, where $d=\dim X$ and $\pi:M^n \rightarrow M^d$ is a coordinate projection.
	Permuting the coordinates if necessary, we may assume that $\pi$ is the projection onto the first $d$ coordinates.
	
	We show that $\myrk(X/A) \leq \dim X$.
	We have nothing to prove when $d=\dim X=n$.
	We consider the other case.
	Take an arbitrary point $(a_1,\ldots, a_n) \in X^{\mathbb M}$.
	We have only to demonstrate that $a_k \in  \mydcl(\{a_1, \ldots, a_d\} \cup A)$ for any $d<k \leq n$.
	Fix $d<k \leq n$.
	For any $t \in \mathbb M^d$, the fiber $X_t^{\mathbb M} =\{ y \in \mathbb M^{n-d}\;|\; (t,y) \in X^{\mathbb M}\}$ is discrete and closed because $X$ is a $\pi$-normal submanifold.
	Let $\pi_1:M^n \rightarrow M^{d+1}$ be the coordinate projection given by $\pi_1(x_1,\ldots,x_n)=(x_1,\ldots, x_d,x_k)$.
	The projection image of a discrete closed definable set is again discrete and closed by Proposition \ref{prop:zero} and Theorem \ref{thm:dimension_basic}(3).
	Therefore, the fiber $\pi_1(X^{\mathbb M})_{(a_1,\ldots, a_d)}=\{y \in \mathbb M\;|\; (a_1,\ldots, a_d,y) \in \pi_1(X^{\mathbb M})\}$ is discrete and closed.
	The point $a_k$ is obviously contained in $\pi_1(X^{\mathbb M})_{(a_1,\ldots, a_d)}$.
	They imply that $a_k \in  \mydcl(\{a_1, \ldots, a_d\} \cup A)$.
	We have obtained the inequality $\myrk(X/A) \leq d$.
	
	We next prove the opposite inequality.
	Take $a=(a_1,\ldots, a_n) \in X$.
	The intersection $V \cap X$ is the graph of a definable continuous map $f=(f_1,\ldots,f_{n-d}):\pi(V) \rightarrow M^{n-d}$ for some open box $V$ containing the point $a$ because $X$ is a $\pi$-normal manifold.
	Apply Lemma \ref{lem:infetesimal} to the sequence $(a_1,\ldots, a_d)$.
	We get elementary extensions $\mathcal M_0=\mathcal M \preceq \mathcal M_1 \preceq \cdots \preceq \mathcal M_d$ and elements $b_1, \ldots, b_d$ satisfying the conditions in Lemma \ref{lem:infetesimal}.
	We may assume that $\mathcal M_i$ are elementary substructures of the monster model $\mathbb M$ for all $1 \leq i \leq d$.
	We also have $b_d \not\in \mydcl(M_{d-1})$ by Lemma \ref{lem:infetesimal2}.
	It implies that the set $\{b_1,\ldots, b_d\}$ is $\mydcl$-independent over $M$.
	It immediately follows from the definition of $\mydcl$-independence that $\{b_1,\ldots, b_d\}$ is $\mydcl$-independent over $A$.
	We have $(b_1,\ldots, b_d) \in \pi(X)$ because $\pi(X)$ is open.
	Set $b_{j+d}=f_j(b_1,\ldots, b_d)$ for all $1 \leq j \leq n-d$.
	Since $\{b_1,\ldots, b_d\}$ is $\mydcl$-independent over $A$, we have $\myrk(\{b_1,\ldots, b_n\}/A) \geq d$.
	We also have $(b_1,\ldots, b_n) \in X^{\mathbb M}$ because the graph of $f$ is contained in $X$.
	We have demonstrated the inequality $\myrk(X/A) \geq d$.
	
	The `in particular' part is obvious because  $\myrk(X/A)=\dim X$ and $\dim X$ is independent of the parameter set $A$.
\end{proof}

\begin{proof}[Proof of Proposition \ref{prop:top_st_dim}]
	We first prepare a lemma.
	\begin{lemma}\label{lem:for_top_st_dim}
		Let $X$ and $Y$ be definable subsets of $M^n$ with $Y \subseteq X$ and $\dim X=\dim Y$.
		Then, the set $Y$ has a nonempty interior in $X$.
	\end{lemma}
	\begin{proof}
		Set $d=\dim X=\dim Y$.
		Consider the set $Z=\mycl({X \setminus Y}) \cap Y$.
		We have $\dim Z<d$ by Theorem \ref{thm:dim}(2) because $Z \subseteq \partial (X \setminus Y)$.
		The set $Z$ is strictly contained in $Y$ by Theorem \ref{thm:dimension_basic}(1).
		Take a point $x \in Y \setminus Z$.
		By the definition of $Z$, there exists an open box $B$ containing the point $x$ such that $B$ has an empty intersection with $X \setminus Y$.
		In other words, the intersection $X \cap B$ is contained  in $Y$.
		It means that $Y$ has a nonempty interior in $X$.
	\end{proof}

	We start to prove the proposition.
	It is obvious when $X$ is an empty set.
	We concentrate on the case in which $X$ is not empty.
	Let $M^n$ be the ambient space of $X$.
	Set $d=\dim(X)$.
	We prove the proposition by induction on $d$.
	When $d=0$, the definable set $X$ is discrete and closed.
	In particular, any nonempty definable subset of $X$ has a nonempty interior in $X$.
	It implies that $\mydim(X)=0$.
	
	We next consider the case in which $d>0$.
	Let $Y$ be an arbitrary definable closed subset of $X$ having an empty interior in $X$.
	We have $\dim Y<d$ by Lemma \ref{lem:for_top_st_dim}.
	We get $\mydim(Y)<d$ by the induction hypothesis.
	It means that $\mydim(X) \leq d$.
	
	We demonstrate the opposite inequality $\mydim(X) \geq d$.
	We have only to construct a definable closed subset $Y$ of $X$ such that $\mydim(Y)=d-1$ and $Y$ has an empty interior in $X$.
	We decompose $X$ into normal submanifolds $X_1, \ldots, X_t$ by Theorem \ref{thm:decomposition}.
	At least one of $X_1,\ldots, X_t$ is of dimension $d$  by Theorem \ref{thm:dimension_basic}(1).
	We may assume that $\dim X_1=d$ without loss of generality.
	Let $\pi:M^n \rightarrow M^d$ be a coordinate projection such that $X_1$ is a $\pi$-normal submanifold.
	Note that $\pi(X_1)$ is an open set, and, as a consequence, it is of dimension $d$.
	We may assume that $\pi$ is the projection onto the first $d$ coordinates by permuting the coordinates if necessary.
	
	Set $Z_1=\mycl({X \setminus X_1}) \cap X_1$.
	It is of dimension smaller than $d$ by Theorem \ref{thm:dimension_basic}(1) and Theorem \ref{thm:dim}(2) because $Z_1$ is contained in the frontier $\partial (X \setminus X_1)$.
	Its projection image $\pi(Z_1)$ is also of dimension smaller than $d$ by Theorem \ref{thm:dimension_basic}(3).
	We have $\dim \pi(X_1) \setminus \pi(Z_1)=d$ by Theorem \ref{thm:dimension_basic}(1).
	It implies that $\pi(X_1) \setminus \pi(Z_1)$ is not an empty set.
	Take a point $x \in X_1$ with $\pi(x) \not\in \pi(Z_1)$.
	By the definition of $Z_1$ and the assumption that $X_1$ is $\pi$-normal submanifold, we can take an open box $U$ in $M^n$ containing the point $x$ such that $U \cap (X \setminus X_1)=\emptyset$ and $U \cap X_1$ is the graph of a definable continuous map defined on $\pi(U)$.
	In particular, we have $U \cap X \subseteq X_1$.
	
	Take a nonempty closed box $B$ contained in $\pi(U)$.
	The definable set $W=\pi^{-1}(B) \cap U \cap X$ is a definable subset of $X$ which is the graph of a definable continuous function $f$ defined on $B$.
	In particular, the continuity of $f$ implies that the definable set $W$ is closed in $X$.
	Take a point $a=(a_1,\ldots,a_d) \in \myint(B)$ and consider the hyperplane $H=\{x=(x_1,\ldots, x_d) \in M^d\;|\; x_d=a_d\}$.
	We obviously have $\dim H \cap B = d-1$.
	Put $Y= \pi^{-1}(H \cap B) \cap W$.
	It is closed in $X$.
	The fiber $\pi^{-1}(x) \cap Y$ is always a singleton for any point $x \in H \cap B$ because $W$ is the graph of a function.
	In particular, the fiber is of dimension zero.
	Using Theorem \ref{thm:dimension_basic}(2), we get $\dim Y=d-1$.
	The induction hypothesis implies that $\mydim(Y)=d-1$.
	It is obvious that the set $H \cap B$ has an empty interior in $B$.
	The definable set $W$ is definably homeomorphic to $B$ because it is the graph of a definable continuous map defined on $B$.
	The set $Y$ has an empty interior in $W$ because $Y$ is the image of $H \cap B$ under the above definable homeomorphism.
	Since $W$ is a subset of $X$,  $Y$ has an empty interior in $X$.
	We have finally demonstrated the inequality $\mydim(X) \geq d$.
\end{proof}

\end{document}